\documentclass[11pt]{amsart}
\usepackage{epigraph}
\usepackage{amsthm}
\usepackage{bm}
\usepackage{tikz-cd}
\usepackage{amsmath}
\usepackage{amssymb}
\usepackage{amsfonts}
\usepackage{amsxtra}
\usepackage{braket}
\usepackage{lmodern}
\usepackage{mathrsfs}
\usepackage{array}
\usepackage{xcolor}
\definecolor{citecol}{rgb}{0.07,0.07,0.05}
\definecolor{urlcol}{rgb}{0.06,0.04,0.09}
\definecolor{linkcol}{rgb}{0.01,0.03,0.08}
\usepackage[colorlinks,linkcolor=linkcol,urlcolor=urlcol,citecolor=citecol,pagebackref,breaklinks]{hyperref}
\usepackage{appendix} 
\usepackage[lite,abbrev,msc-links,alphabetic]{amsrefs}
\usepackage{indentfirst}
\usepackage{mathtools}
\usepackage[all]{xy}
 \numberwithin{equation}{subsection}
\theoremstyle{plain}
\newtheorem{theorem}{Theorem}[section]
\newtheorem{lemma}[theorem]{Lemma}
\newtheorem{proposition}[theorem]{Proposition}

\theoremstyle{definition}
\newtheorem{definition}[theorem]{Definition}
\newtheorem{remark}[theorem]{Remark}

\newtheorem*{acknowledgement}{Acknowledgement}

\theoremstyle{remark}

\newcommand{\td}[1]{\tilde{#1}}

\newcommand{\tb}[1]{\textbf{#1}}
\newcommand{\tx}[1]{\text{#1}}
\newcommand{\BA}{{\mathbb A}}
\newcommand{\BB}{{\mathbb B}}
\newcommand{\BC}{{\mathbb C}}

\newcommand{\BF}{{\mathbb F}}
\newcommand{\BG}{{\mathbb G}}

\newcommand{\BL}{{\mathbb L}}
\newcommand{\BM}{{\mathbb M}}
\newcommand{\BN}{{\mathbb N}}

\newcommand{\BP}{{\mathbb P}}
\newcommand{\BQ}{{\mathbb Q}}
\newcommand{\BR}{{\mathbb R}}

\newcommand{\BV}{{\mathbb V}}

\newcommand{\BX}{{\mathbb X}}
\newcommand{\BY}{{\mathbb Y}}
\newcommand{\BZ}{{\mathbb Z}}

\newcommand{\CA}{{\mathcal A}}
\newcommand{\CB}{{\mathcal B}}
\newcommand{\CC}{{\mathcal C}}

\newcommand{\CF}{{\mathcal F}}

\newcommand{\CH}{{\mathcal H}}

\newcommand{\CL}{{\mathcal L}}
\newcommand{\CM}{{\mathcal M}}
\newcommand{\CN}{{\mathcal N}}

\newcommand{\CP}{{\mathcal P}}

\newcommand{\CR}{{\mathcal R}}

\newcommand{\CV}{{\mathcal V}}

\newcommand{\CY}{{\mathcal Y}}
\newcommand{\CZ}{{\mathcal Z}}

\newcommand{\FF}{{\mathfrak F}}

\newcommand{\FK}{{\mathfrak K}}

\newcommand{\Fa}{{\mathfrak a}}

\newcommand{\Fl}{{\mathfrak l}}

\DeclareMathOperator{\Gal}{Gal}
\DeclareMathOperator{\Lie}{Lie}
\DeclareMathOperator{\Charpol}{Charpol}
\DeclareMathOperator{\End}{End}

\DeclareMathOperator{\Spec}{Spec}

\DeclareMathOperator{\Spf}{Spf}
\DeclareMathOperator{\length}{length}
\DeclareMathOperator{\Ker}{Ker}
\DeclareMathOperator{\pr}{pr}
\DeclareMathOperator{\Frac}{Frac}
\DeclareMathOperator{\Hom}{Hom}
\DeclareMathOperator{\Sch}{Sch}
\DeclareMathOperator{\Aut}{Aut}
\DeclareMathOperator{\Res}{Res}
\DeclareMathOperator{\inv}{inv}
\DeclareMathOperator{\Nm}{Nm}
\DeclareMathOperator{\sig}{sig}

\DeclareMathOperator{\inc}{inc}

\DeclareMathOperator{\Sh}{Sh}

\author{Sungyoon Cho}
\address[Sungyoon Cho]{Department of Mathematics, Northwestern University}
\email{sungyoon@math.northwestern.edu}

\title[Supersingular locus]{The basic locus of the unitary Shimura variety with parahoric level structure, and special cycles}
\date{\today}

\begin{document}

\begin{abstract}
In this paper, we study the basic locus in the fiber at $p$ of a certain unitary Shimura variety with a certain parahoric level structure. The basic locus $\widehat{\CM^{ss}}$ is uniformized by a formal scheme $\CN$ which is called Rapoport-Zink space. We show that the irreducible components of the induced reduced subscheme $\CN_{red}$ of $\CN$ are Deligne-Lusztig varieties and their intersection behavior is controlled by a certain Bruhat-Tits building. Also, we define special cycles in $\CN$ and study their intersection multiplicities.
\end{abstract}

\maketitle
\tableofcontents{}

\section{Introduction}
This paper is a contribution to the theory of integral models of certain Shimura varieties. In particular, we will give a concrete description of their basic loci. These problems have important applications to Kudla's program which relates arithmetic intersection numbers of special cycles on integral models of certain Shimura varieties to Eisenstein series (see \cite{KR2}, \cite{KR3}), and Arithmetic Gan-Gross-Prasad conjecture (see \cite{Zha}, \cite{RSZ1}, \cite{RSZ2}, \cite{RSZ3}). In this paper, we study the basic locus of the special fiber of a certain unitary Shimura variety at an inert prime with parahoric level structure. Let $(\tilde{G},h_{\tilde{G}})$ be a Shimura datum and let $K_{\tilde{G}}$ be an open compact subgroup in $\tilde{G}(\BA_{f})$. We refer to Section \ref{sec:section4} for the precise definition. This Shimura variety has a moduli interpretation $M_{K_{\tilde{G}}}(\tilde{G})$ as a moduli space of abelian varieties with additional structure. This Shimura variety is a variant of the Shimura variety which appears in \cite{GGP} and its integral model $\CM_{K_{\tilde{G}}}(\tilde{G})$ is defined in \cite{RSZ2}.  The basic locus of the special fiber of $\CM_{K_{\tilde{G}}}(\tilde{G})$ can be studied using the uniformization theorem of Rapoport and Zink, \cite[Theorem 6.30]{RZ} (more precisely, see Theorem \ref{theorem}). Therefore, we can study the corresponding Rapoport-Zink space and use its explicit description to study the basic locus of the special fiber of the Shimura variety.

We will now describe our main results in more detail. First, let us consider the Rapoport-Zink spaces which are local analogues of Shimura varieties.

\subsection{The local result : relative Rapoport-Zink spaces}\label{subsec:section1.1}
Let $F$ be a finite extension of $\BQ_p$, and let $E$ be a quadratic unramified extension of $F$ with ring of integers $O_E$ and residue field $\BF_{q^2}$. We fix a uniformizer $\pi$. Let $\breve{E}$ be the completion of a maximal unramified extension of $E$. Fix integers $n$ and $0 \leq h,r \leq n$. Here, $h$ is related to a certain self-dual lattice chain, and $r$ is related to the determinant condition. We define a moduli space $\CN^h_{E/F}(r,n-r)$ over $\Spf O_E$ of quasi-isogenies of strict formal $O_F$-modules with additional structure (see Section \ref{sec:section2} for its definition). If $h=0$, $r=1$, $F=\BQ_p$, and $E=\BQ_{p^2}$, then this moduli space coincides with the Rapoport-Zink space that is studied by Vollaard and Wedhorn (\cite{VW}). This case corresponds to the hyperspecial level structure case. In their paper, they proved that the irreducible components of the induced reduced scheme of $\CN^0_{\BQ_{p^2}/\BQ_p}(1,n-1)$ are Deligne-Lusztig varieties, and their intersection behavior is controlled by a certain Bruhat-Tits building. Howard and Pappas studied the moduli space $\CN^0_{\BQ_{p^2}/\BQ_p}(2,2)$ in \cite{HP} (also, see Remark \ref{remark2.20}). When $h$ is not equal to 0, we have a parahoric level structure. When $h=1$, $n=2$, the moduli space $\CN^1_{E/F}(1,1)$ is studied in \cite{KR1}. In this paper, Kudla and Rapoport proved that the moduli space is represented by a Drinfeld $p$-adic half-plane. Furthermore, they studied $\CN^1_{\BQ_{p^2}/\BQ_p}(1,n-1)$ in their unpublished notes \cite{KR4}. They showed that its reduced scheme has two kinds of Bruhat-Tits strata: One consists of projective spaces and the other consists of Deligne-Lusztig varieties. Our result is the generalization of theirs to arbitrary $h$ and $F$.

The cases that $E$ is a ramified extension of $F$ are also studied in literature. For example, we refer to \cite{RTW}, \cite{Wu} (also, see \cite{RSZ1}, \cite{RSZ2}, \cite{RSZ3} for their connection to Arithmetic Gan-Gross-Prasad conjecture).

We now state our main result in local situation. Let $(\BX,i_{\BX},\lambda_{\BX})$ be a framing object of $\CN_{E/F}^h(1,n-1)$: $\BX$ is a supersingular strict formal $O_F$-module of $F$-height $2n$ over $\BF_{q^2}$; $i_{\BX}$ is an $O_E$-action on $\BX$, and $\lambda_{\BX}$ is a polarization. We note that the integer $h$ is related to this polarization. For this triple, there is an associated hermitian $E$-vector space $N_{k,0}^{\tau}$. An $O_E$-lattice $\Lambda$ in $N_{k,0}^{\tau}$ is called a vertex lattice of type $t(\Lambda)$, if $\pi^{i+1}\Lambda^{\vee} \subset \Lambda \subset \pi^i\Lambda^{\vee}$ for some $i$ and the dimension of $\Lambda/ \pi^{i+1}\Lambda^{\vee}$ is $t(\Lambda)$ as $\BF_{q^2}$-vector space. Here, $\Lambda^{\vee}$ is the dual lattice of $\Lambda$. For each $i=0,1$, we denote by $\CL_i$ the set of vertex lattices. We also define the following sets of vertex lattices:
\begin{equation*}
\CL_0^+ := \lbrace O_E\text{-lattices } \Lambda \text{ }  \vert \text{ } \pi \Lambda^{\vee} \subset \Lambda \subset \Lambda^{\vee}, t(\Lambda) \geq h+1 \rbrace;
\end{equation*}
\begin{equation*}
\CL_0^- := \lbrace O_E\text{-lattices } \Lambda \text{ }  \vert \text{ } \pi \Lambda^{\vee} \subset \Lambda \subset \Lambda^{\vee}, t(\Lambda) \leq h-1 \rbrace;
\end{equation*}
\begin{equation*}
\CL_1^+ := \lbrace O_E\text{-lattices } \Lambda \text{ }  \vert \text{ } \pi^2 \Lambda^{\vee} \subset \Lambda \subset \pi\Lambda^{\vee}, t(\Lambda) \geq n-h+1 \rbrace;
\end{equation*}
\begin{equation*}
\CL_1^- := \lbrace O_E\text{-lattices } \Lambda \text{ }  \vert \text{ } \pi^2 \Lambda^{\vee} \subset \Lambda \subset \pi\Lambda^{\vee}, t(\Lambda) \leq n-h-1 \rbrace.
\end{equation*} 
Note that there is a bijection between $\CL_1^+$ and $\CL_0^-$ via the map sending $\Lambda \in \CL_1^+$ to $\pi \Lambda^{\vee} \in \CL_0^-$. In this way, the union $\CL_0^+ \sqcup \CL_1^+$ can be identified with $\CL_0^+ \sqcup \CL_0^-$ and then this can be identified with the set of vertices of a certain Bruhat-Tits building. For each vertex lattices $\Lambda$ in $\CL_0^+ \sqcup \CL_1^+$, we define a projective subscheme $\CN_{\Lambda}$ of the reduced subscheme of $\CN_{E/F}^h(1,n-1)_{O_{\breve{E}}}$. For $i=0,1$ and $\Lambda \in \CL_i^+$, we define the set $\CL_{\Lambda}^+:=\lbrace \Lambda' \in \CL_i^+ \vert \Lambda' \subsetneq \Lambda \rbrace$. We define the subscheme $\CN_{\Lambda}^0:=\CN_{\Lambda} \backslash \bigcup_{\Lambda' \in \CL_{\Lambda}^+} \CN_{\Lambda'}$.
The schemes $\CN_{\Lambda}$, $\CN_{\Lambda}^0$ have the following properties (see Theorem \ref{theorem 4.11} and Section \ref{subsec:subsection3.8}).
\begin{theorem}\label{thmintro1}The following properties of $\CN_{E/F}^h(1,n-1)$ hold.
	\begin{enumerate}
	\item  For $\Lambda \in \CL_0^+$ (resp. $\Lambda \in \CL_1^+)$, $\CN_{\Lambda}$ is isomorphic to a Deligne-Lusztig variety and it is projective, smooth, and geometrically irreducible of dimension $\frac{1}{2}(t(\Lambda)-h-1)+h$ (resp. $\frac{1}{2}(t(\Lambda)-(n-h+1))+n-h$).
	
	\item  For $i=0,1$, consider $\Lambda \in \CL_i^+$. Then $\CN_{\Lambda}^0$ is open and dense in $\CN_{\Lambda}$ and we have a stratification $(\CN_{\Lambda}^0)_{\Lambda \in \CL_i^+, i=0,1}$ of $\CN_{E/F}^h(1,n-1)_{O_{\breve{E}}}$ which is called the Bruhat-Tits stratification. The closed subschemes $\CN_{\Lambda}$ of $\CN_{E/F}^h(1,n-1)_{O_{\breve{E}}}$ are called the closed Bruhat-Tits strata.
	
	\item  For $i=0,1$, consider two vertex lattices $\Lambda' \subset \Lambda$ in $\CL_i^+$. Then we have $\CN_{\Lambda'} \subset \CN_{\Lambda}$.
	
	\item  For $i=0,1$, consider two vertex lattices $\Lambda', \Lambda$ in $\CL_i^+$. Then two closed Bruhat-Tits strata $\CN_{\Lambda}$, $\CN_{\Lambda'}$ have nonempty intersection if and only if $\Lambda \cap \Lambda' \in \CL_i^+$, and in this case $\CN_{\Lambda} \cap \CN_{\Lambda'}=\CN_{\Lambda \cap \Lambda'}$.
	
	\item  For vertex lattices $\Lambda_0 \in \CL_0^+$, $\Lambda_1 \in \CL_1^+$, two closed Bruhat-Tits strata  $\CN_{\Lambda}$, $\CN_{\Lambda'}$ have nonempty intersection if and only if $\pi \Lambda_1^{\vee} \subset \Lambda_0$.
\end{enumerate}
\end{theorem}

We also have the following properties of $\CN_{E/F}^h(1,n-1)_{O_{\breve{E}}}$.
\begin{theorem}\label{thmintro2} The following assertions hold.
\begin{enumerate} 
		\item In case $h \neq 0,n$, the formal scheme $\CN_{E/F}^h(1,n-1)_{O_{\breve{E}}}$ has semistable reduction. If $h=0,n$, $\CN_{E/F}^h(1,n-1)_{O_{\breve{E}}}$ is formally smooth over $\Spf O_{\breve{E}}$. In particular, it is regular for all $h$.
		
		\item There exists a Rapoport-Zink space $\CN^h_{E/\BQ_p}(1,n-1)_{O_{\breve{E}}}$ of PEL type that is isomorphic to $\CN^h_{E/F}(1,n-1)_{O_{\breve{E}}}$.
	\end{enumerate}
\end{theorem}

\begin{remark}
	In case $F$ is unramified over $\BQ_p$, the above statements in Theorem \ref{thmintro1} and Theorem \ref{thmintro2} hold without base change to $O_{\breve{E}}$.
\end{remark}

We now describe \S\ref{sec:section2}-\ref{sec:section3} in more detail. In Section \ref{sec:section2}, we study the $k$-points of $\CN^h_{E/F}(1,n-1)$ by using the relative Dieudonne theory, where $k$ is an algebraic closure of the residue field of $E$. In Section \ref{sec:section3}, we define a subscheme $\CN_{\Lambda}$ for each vertex lattice $\Lambda$ and prove that this is isomorphic to a Deligne-Lusztig variety. Furthermore, we prove the regularity of $\CN^h_{\BQ_{p^2}/\BQ_p}(1,n-1)$ via the theory of local model. Also, we prove that there is a stratification of $\CN_{E/F}^h(1,n-1)$ so called Bruhat-Tits stratification. Finally, we relate $\CN_{E/F}^h(1,n-1)$ to a certain PEL-type Rapoport-Zink space as Mihatsch did in \cite{Mih}. By using this result, we prove the regularity of $\CN_{E/F}^h(1,n-1)$.

\subsection{The global result: non-archimedean uniformization}\label{subsec:section1.2}
In the global situation, we write $F$ for a CM field, $F^+$ for its totally real subfield of index $2$, and $\Phi$ for a CM type. We fix an embedding $\tau_1^- \in \Phi$ and an embedding $\tilde{v}:\bar{\BQ} \rightarrow \bar{\BQ}_p$. These two determine places $v_0$ of $F^+$ and $w_0$ of $F$. We assume further that $v_0$ is unramified over $p$ and inert in $F$. We denote by $S_p$ the set of places of $F^+$ over $p$.
We will define three Shimura data: $(G,h_G), (Z,h_Z), (\tilde{G},h_{\tilde{G}})$. The first Shimura datum is associated to a unitary group $\Res_{F^+/\BQ}U(V)$ for a hermitian space $V$. This Shimura variety is of abelian type and appears in \cite{GGP}. The second Shimura datum is associated to a torus $Z$. The third Shimura datum is the product of the first two Shimura data, and is our main interest. This Shimura variety is studied in \cite{RSZ2}, and the authors formulate a moduli problem $M_{K_{\tilde{G}}}(\tilde{G})$ of abelian varieties with additional structure. Here, $K_{\tilde{G}}$ is a certain open compact subgroup of $\tilde{G}(\BA_f)$. We should note that an integer $0 \leq h \leq n$ also appears in global situation, and this is closely related to $K_{\tilde{G}}$. In particular, if $h=0$, $K_{\tilde{G}}$ gives a hyperspecial level structure, and if $h \neq 0$, $K_{\tilde{G}}$ gives a parahoric level structure. This $h$ is also closely related to the $h$ in local situation. The moduli problem $M_{K_{\tilde{G}}}(\tilde{G})$ gives a model over a reflex field $E$ of the Shimura variety $\Sh_{K_{\tilde{G}}}(\tilde{G})$. We write $u$ for the place of $E$ that is determined by $\tilde{v}$. In \cite{RSZ2}, the authors define global integral models of $M_{K_{\tilde{G}}}(\tilde{G})$ over $\Spec O_E$ and semi-global integral models over $\Spec O_{E,(u)}$ in case $h=0$, and in case $h=1$, $F^+_{v_0}=\BQ_p$. In our paper, we construct semi-global integral models $\CM_{K_{\tilde{G}}}(\tilde{G})$ over $\Spec O_{E,(u)}$ for arbitrary $h$.

Now we can formulate the following proposition.

\begin{proposition}(Proposition \ref{proposition4.4.1}, Proposition \ref{proposition4.4.2})
	We can formulate a moduli problem that is representable by a Deligne-Mumford stack $\CM_{K_{\tilde{G}}}(\tilde{G})$ flat over $\Spec O_{E,(u)}$. For $K^p_{G}$ small enough, $\CM_{K_{\tilde{G}}}(\tilde{G})$ is relatively representable over $\CM_0^{\Fa,W}$. The generic fiber $\CM_{K_{\tilde{G}}}(\tilde{G}) \times_{\Spec O_{E,(u)}} \Spec E$ is canonically isomorphic to $M_{K_{\tilde{G}}}(\tilde{G})$ and $M_{K_{\tilde{G}}}(\tilde{G})$ is naturally isomorphic to the canonical model of $\Sh_{K_{\tilde{G}}}(\tilde{G})$. Furthermore, if $h=0, n$, then $\CM_{K_{\tilde{G}}}(\tilde{G})$ is smooth over $\Spec O_{E,(u)}$. If $h \neq 0,n$, then $\CM_{K_{\tilde{G}}}(\tilde{G})$ has semistable reduction over $\Spec O_{E,(u)}$ provided that $E_{u}$ is unramified over $\BQ_p$.
\end{proposition}

Now we will state the non-archimedean uniformization theorem of Rapoport and Zink in our situation. By this theorem, we can relate the basic locus of $\CM_{K_{\tilde{G}}}(\tilde{G})$ and the Rapoport-Zink space $\CN_{F_{w_0}/F^+_{v_0}}^h(1,n-1)$. In order to simplify notation, we write $\CM$ for $\CM_{K_{\tilde{G}}}(\tilde{G})$ and $\CN$ for $\CN_{F_{w_0}/F^+_{v_0}}^h(1,n-1)$. Let $\breve{E}_u$ be the completion of a maximal unramified extension of $E_u$, and let $k$ be the residue field of $O_{\breve{E}_u}$. Let $\widehat{\CM^{ss}}$ be the completion of $\CM_{O_{\breve{E}_u}}$ along the basic locus of $\CM_{O_{\breve{E}_u}} \otimes k$. Then we have the following non-archimedean uniformization theorem.

\begin{theorem}(Theorem \ref{theorem})
	There is a non-archimedean uniformization isomorphism
	\begin{equation*}
	\Theta:I(\BQ) \backslash  \CN' \times \tilde{G}(\BA^p_f) /K^p_{\tilde{G}} \overset{\sqcup\Theta_W}{\simeq } \widehat{\CM^{ss}},
	\end{equation*}
	where
	\begin{equation*}
	\CN' \simeq (Z(\BQ_p)/K_{Z,p}) \times \CN_{O_{\breve{E}_u}} \times \prod_{v \in S_p \backslash \lbrace v_0 \rbrace} U(V)(F^+_v)/K_{G,v}.
	\end{equation*}
	
\end{theorem}

Here, $I$ is an inner twist of $\tilde{G}$. We refer to Section \ref{subsec:section4.3} for all notation above and its detail.

\subsection{Special cycles}\label{subsec:section1.3}
In this subsection, we use the notation in Section \ref{subsec:section1.1}. In \cite{KR4}, Kudla and Rapoport defined the special cycles $\CZ(x)$ in $\CN^1_{\BQ_{p^2}/\BQ_p}(1,n-1)$ and computed its reduced scheme as in their another paper \cite{KR2}. By following their work, we define special cycles $\CZ(x)$ and another special cycles $\CY(y)$ in $\CN^h_{E/F}(1,n-1)_{O_{\breve{E}}}$. We also study their reduced schemes and arithmetic intersection numbers in some cases.

Let $k$ be the residue field of $O_{\breve{E}}$, and let $(\overline{\BY},i_{\overline{\BY}},\lambda_{\overline{\BY}})$ (resp. $(\BX,i_{\BX},\lambda_{\BX})$) be the framing object of $\CN^0_{E/F}(0,1)_{O_{\breve{E}}}$ (resp. $\CN^h_{E/F}(1,n-1)_{O_{\breve{E}}}$). The space of special homomorphisms $\BV$ is the $E$-vector space
\begin{equation*}
\BV:=\Hom_{O_E}(\overline{\BY}, \BX) \otimes_{\BZ} \BQ,
\end{equation*}
with a $E$-valued hermitian form $h$ such that for all $x,y \in \BV$,
\begin{equation*}
h(x,y):=\lambda_{\overline{\BY}}^{-1} \circ y^{\vee} \circ \lambda_{\BX} \circ x \in \End_{O_E}(\overline{\BY}) \otimes \BQ \overset{i_{\overline{\BY}}^{-1}}{\simeq} E.
\end{equation*}
For each $x \in \BV$, we define the special cycle $\CZ(x)$ as follows.
For each $O_{\breve{E}}$-scheme $S$ such that $\pi$ is locally nilpotent, $\CZ(x)(S)$ is the subfunctor of collections $(\overline{Y},i_{\overline{Y}},\lambda_{\overline{Y}},\rho_{\overline{Y}},X,i_X,\lambda_X,\rho_X)$ such that
\begin{equation*}
\overline{Y} \times_S \overline{S} \xrightarrow{\rho_{\overline{Y}}} \overline{\BY} \times_k \overline{S} \xrightarrow{x} \BX \times_k \overline{S} \xrightarrow{\rho_X^{-1}} X \times_S \overline{S}
\end{equation*}
extends to a homomorphism from $\overline{Y}$ to $X$.

For each $y \in \BV$, we define the special cycle $\CY(y)$ in a similar way, but here we use the isomorphism $\CN^h_{E/F}(1,n-1)_{O_{\breve{E}}} \simeq \CN^{n-h}_{E/F}(1,n-1)_{O_{\breve{E}}}$ to define the cycle. We refer to Definition \ref{definition5.1.4} for the precise definition. All of these cycles are relative divisors in $\CN^h_{E/F}(1,n-1)_{O_{\breve{E}}}$. Therefore we can consider the arithmetic intersections of these cycles as in \cite{KR2}.

We prove the following theorem.

\begin{theorem}(Theorem \ref{theorem5.5.12})
		Let $\lbrace x_1, \dots ,x_{n-h}, y_1, \dots, y_h\rbrace$ be an orthogonal basis of $\BV$. Assume that
	\begin{equation*}
	\begin{array}{ll}
	val(h(x_i,x_i))=0  &\text{ for all } 3 \leq i \leq n-h,\\
	val(h(y_j,y_j))=-1 &\text{ for all } 1 \leq j \leq h,\\
	\end{array}
	\end{equation*}
	and write $a:=val(h(x_1,x_1))$, $b:=val(h(x_2,x_2))$. We assume that $a \leq b$ and $a \not\equiv b \mod 2$. Then we have
	\begin{equation*}
	\chi(O_{\CY(y_1)} \otimes^{\BL}_{O_{\CN}} \dots \otimes^{\BL}_{O_{\CN}} O_{\CZ(x_h)})=\dfrac{1}{2} \sum_{l=0}^{a} q^l(a+b+1-2l).
	\end{equation*}
	
	More generally, consider another basis $[\td{\tb{x}},\td{\tb{y}}]:=[ \td{x}_1, \dots ,\tilde{x}_{n-h}, \tilde{y}_1, \dots, \tilde{y}_h]$ of $\BV$ such that $\td{\tb{x}}=\td{x} g_1, \td{\tb{y}}=\td{y} g_2$ for $g_1 \in \tx{GL}_{n-h}(O_E)$ and $g_2 \in \tx{GL}_h(O_E)$. Then we have
	\begin{equation*}
	\chi(O_{\CY(\td{y}_1)} \otimes^{\BL}_{O_{\CN}} \dots \otimes^{\BL}_{O_{\CN}} O_{\CZ(\td{x}_h)})=\dfrac{1}{2} \sum_{l=0}^{a} q^l(a+b+1-2l).
	\end{equation*}
\end{theorem}

In this case, the reduced scheme of the intersection has dimension $0$. Therefore we can use the deformation theory as in \cite{KR2} for $F=\BQ_p$ and \cite{Liu} in general.

We have one more case that seems to be realistic, but we do not include it in this paper. See Remark \ref{remark5.5.16}. Also, we believe that the similar conjecture to \cite[Conjecture 1.3]{KR2} can be formulated in our case.

\begin{acknowledgement}
I want to heartily thank my advisor Yifeng Liu for suggesting this problem, for his advice, and for many helpful discussions and comments. I would like to thank Stephen S. Kudla and Michael Rapoport for their very helpful unpublished notes. 

\end{acknowledgement}

\bigskip
\section{The moduli space $\mathcal{N}$ of strict formal $O_F$-modules}\label{sec:section2}
In this section, we will define the moduli problem $\mathcal{N}$ and study its structure.
\subsection{The moduli space $\CN_{E/F}^h(r,n-r)$}\label{subsec:section2.1}
We fix a prime $p>2$. Let $F$ be a finite extension of $\BQ_p$, with ring of integers $O_F$, and residue field $\BF_q$. We fix a uniformizer $\pi$.  Let $E$ be a quadratic unramified extension of $F$, with ring of integers $O_E$ and residue field $\BF_{q^2}$. Let $\breve{E}$ be the completion of a maximal unramified extension of $E$. Denote by $^*$ the nontrivial Galois automorphism of $E$ over $F$. We recall the definition of strict formal $O_F$-module from \cite{RZ2}.

\begin{definition}
	Let $S$ be a scheme such that $p$ is locally nilpotent in $O_S$. A \textit{formal $O_F$-module} over a scheme $S$ is a formal $p$-divisible group $X$ over $S$ with an $O_F$-action
	\begin{equation*}
	i:O_F \rightarrow \End X.
	\end{equation*}
	
	Let $X$ be a formal $O_F$-module over an $O_F$-scheme $S$. We call $X$ a \textit{strict} formal $O_F$-module if $O_F$ acts on $\Lie X$ via the structure morphism $O_F \rightarrow O_S$. A strict formal $O_F$-module $X$ is called \textit{supersingular} if all slopes of $X$ as a strict $O_F$-module are $1/2$.
\end{definition}

Let $h$ be an integer with $0 \leq h \leq n$. We fix a triple $(\mathbb{X}, i_{\BX}, \lambda_{\BX})$ consisting of the following data:

(1) $\mathbb{X}$ is a supersingular strict formal $O_F$-module of $F$-height $2n$ over $\BF_{q^2}$; 

(2) $i_{\BX}:O_E \rightarrow \End \BX$ is an $O_E$-action on $\BX$ that extends the $O_F$-action on $\BX$;

(3) $\lambda_{\BX}$ is a polarization
\begin{equation*}
\lambda_{\BX} :\BX \rightarrow \BX^{\vee},
\end{equation*}
such that the corresponding Rosati involution induces the involution $^*$ on $O_E$.

We also assume that $(\BX,i_{\BX},\lambda_{\BX})$ satisfies the following conditions.

(a) For all $a \in O_E$, the action $i_{\BX}$ satisfies
\begin{equation*}
\Charpol(i_{\BX}(a) \vert \Lie \BX) = (T-a)^r(T-a^*)^{n-r}.
\end{equation*}

Here, we view $(T-a)^r(T-a^*)^{n-r}$ as an element of $O_S[T]$ via the structure morphism.
We call this condition the determinant condition of signature $(r,n-r)$.

(b) We assume that $\Ker \lambda_{\BX} \subset \BX[\pi]$ and its order is $q^{2h}$.

Now, we can define our moduli problem.

Let (Nilp) be the category of $O_E$-schemes $S$ such that $\pi$ is locally nilpotent on $S$. 
Let $\CN_{E/F}^h(r,n-r)$ be the set-valued functor on (Nilp) which sends a scheme $S \in$ (Nilp) to the set of isomorphism classes of tuples $(X,i_X,\lambda_X,\rho_X)$.

Here $X$ is a (supersingular) formal $O_F$-module of $F$-height $2n$ over $S$ and $i_X$ is an $O_E$-action on $X$ satisfying the determinant condition of signature $(r,n-r)$
\begin{equation*}
\Charpol(i_X(a)\vert \Lie X)=(T-a)^r(T-a^*)^{n-r}, \quad \forall a \in E.
\end{equation*}
Here we view $(T-a)^r(T-a^*)^{n-r}$ as an element of $O_S[T]$ via the structure morphism $O_E \rightarrow O_S$.

 Furthermore, $\rho_X$ is an $O_E$-linear quasi-isogeny \begin{equation*}
\rho_X:X_{\overline{S}} \rightarrow \mathbb{X}\times_{\BF_{q^2}}\overline{S},
\end{equation*}
of height 0, where $\overline{S}=S\times_{O_E}\BF_{q^2}$ and $X_{\overline{S}}$ is the base change $X\times_S\overline{S}$.

 Finally, $\lambda_X:X \rightarrow X^{\vee}
$ is a polarization
  such that its Rosati involution induces the involution $^*$ on $O_E$, and the following diagram commutes up to a constant in $O_F^{\times}$

\begin{center}
\begin{tikzcd}
X_{\overline{S}} \arrow{r}{\lambda_{X_{\overline{S}}}}
\arrow{d}{\rho_X}
&X^{\vee}_{\overline{S}} \\
\mathbb{X}_{\overline{S}}
 \arrow{r}{\lambda_{\BX_{\overline{S}}}}  &\mathbb{X}_{\overline{S}}^{\vee}
 \arrow{u}{\rho_X^{\vee}}.
\end{tikzcd}
\end{center}

Two quadruples $(X,i_X,\lambda_X,\rho_X)$ and $(X',i_{X'},\lambda_{X'},\rho_{X'})$ are isomorphic if there exists an $O_E$-linear isomorphism $\alpha:X \rightarrow X'$ such that $\rho_{X'} \circ (\alpha \times_S \overline{S})=\rho_X$ and $\alpha^{\vee} \circ \lambda_{X'} \circ \alpha$ differs locally on $S$ from $\lambda_X$ by a scalar in $O_F^{\times}$. \\

The functor $\CN_{E/F}^h(r,n-r) \otimes O_{\breve{E}}$ is representable by a formal scheme over $\Spf O_{\breve{E}}$ which is locally formally of finite type. This is explained in \cite{Mih}. Indeed, we can use \cite[Theorem 2.16]{RZ}, and the fact that the condition that the $O_F$-action on $X$ lifts from $\BX$, and the condition that the lifted action is strict are closed conditions.

Furthermore, when $F$ is unramified extension of $\BQ_p$, we will fix a decent $(\BX,i_{\BX},\lambda_{\BX})$ in Remark \ref{remark5.11}. Then $\CN_{E/F}^h(r,n-r)$ is representable by a formal scheme over $\Spf O_E$ which is locally formally of finite type. For the moment assume that we fix this triple $(\BX,i_{\BX},\lambda_{\BX})$ so that $\CN_{E/F}^h(r,n-r)$ is representable by a formal scheme over $\Spf O_E$ which is locally formally of finite type, where $F$ is unramified over $\BQ_p$.

From now on, we will restrict ourselves to the case $r=1$. Note that the case $(r=1, h=0, F=\BQ_p)$ is studied in \cite{VW}. For simplicity, denote by $\CN$ the moduli problem $\CN_{E/F}^h(1,n-1)$.
\bigskip

\subsection{Description of the points of $\mathcal{N}$}\label{subsec:section2.2}

Let $k$ be a fixed algebraic closure of $O_E/\pi O_E=\BF_{q^2}$. In this subsection, we will study the set $\CN(k)$. For this, we need to use relative Dieudonne theory in the sense of \cite[Proposition 3.56]{RZ}. We use the following notation.

Let $\breve{F}$ be the completion of a maximal unramified extension of $F$ containing $E$ and $O_{\breve{F}}$ its ring of integers. Let $F^u$ be the maximal unramified extension of $\BQ_p$ in $F$ and $O_{F^u}$ its ring of integers. Let $L$ be a perfect field with $O_F/\pi O_F=\BF_q$-algebra structure $\alpha_0: \BF_q \rightarrow L$. Then, we get a map $O_{F^u} \rightarrow W(L)$ induced from $\alpha_0: \BF_q \hookrightarrow L$. We define $W_{O_F}(L)=O_F \otimes_{O_{F^u},\alpha_0}W(L)$. This is the ring of relative Witt vectors of $L$. In particular $W_{O_F}(k)=O_{\breve{F}}.$

Let $\sigma$ be the Frobenius element in $\Gal(\breve{F}/F)$.\\

We recall from \cite[Proposition 3.56]{RZ} (or \cite[Notation]{KR1}) the definition of the relative Dieudonne module. Let $X$ be a formal $O_F$-module of $F$-height $2n$ over $k$. Let $(\tilde{M}, \tilde{\CV})$ be the (absolute) Dieudonne module of $X$. Consider the decomposition
\begin{equation*}
O_F \otimes_{\BZ_p}W(k)=\prod_{\alpha:\BF_{q} \rightarrow k} O_F \otimes_{O_{F^u},\alpha} W(k).
\end{equation*}
Here, $\alpha$ runs over the set of $\BF_p$-embeddings $\alpha:\BF_{q} \rightarrow k$. Via this decomposition, the action of $O_F$ on $\tilde{M}$ induces the decomposition
\begin{equation*}
\tilde{M}=\bigoplus_{\alpha:\BF_{q} \rightarrow k} \tilde{M}^{\alpha}.
\end{equation*}
We define the \textit{relative Dieudonne module} of $X$ as
\begin{equation*}
(M^{\alpha_0},\CV=\tilde{\CV}^f),
\end{equation*}
where $f=\vert F^u : \BQ_p \vert=\vert \BF_{q}:\BF_p \vert$.

Now, let $(\mathbb{M},\CV)$ be the relative Dieudonne module of $\mathbb{X}$, and let $N=\mathbb{M} \otimes_{\BZ}\BQ$ be its relative Dieudonne crystal. Denote by $N_k=\mathbb{M} \otimes_{E}\breve{F}$ its base change. The $O_E$-action $i_{\BX}$ on $\BX$ induces an $E$-action on $N_k$. Let $\CF$ be the Frobenius of $\BM$. The polarization $\lambda_{\BX}$ of $\BX$ induces a nondegenerate $\breve{F}$-bilinear alternating form on $N_k$
\begin{equation*}
	\langle\cdot, \cdot\rangle:N_k \times N_k \rightarrow \breve{F},
\end{equation*}
such that for all $x,y \in N_k, a \in E$, it satisfies
\begin{equation}\label{eq1}
	\langle \CF x,y \rangle=\langle x,\CV y \rangle^{\sigma} ,
\end{equation}
\begin{equation}\label{eq2}
	\langle ax,y \rangle =\langle x,a^{*}y \rangle.
\end{equation}

Since we have the decomposition $E \otimes_F \breve{F} \simeq \breve{F} \times \breve{F}$, the $E$-action $i$ on $N_k$ induces $\BZ/2\BZ$-grading
\begin{equation*}
	N_k=N_{k,0} \oplus N_{k,1}.
	\end{equation*}

Note that by \eqref{eq1}, \eqref{eq2}, each $N_{k,i}$ is totally isotropic with respect to $\langle \cdot, \cdot \rangle$. Also, for $i=0,1$, we have that $\CF:N_{k,i} \rightarrow N_{k,i+1}$, $\CV:N_{k,i} \rightarrow N_{k,i+1}$ are homogeneous of degree 1 with respect to the decomposition.

For an $O_{\breve{F}}$-lattice $M=M_0 \oplus M_1$, we define the dual lattice $M_i^{\perp}$ of $M_i$ as
\begin{equation*}
	M_i^{\perp}=\lbrace x \in N_{k,i+1} \vert \langle x, M_i \rangle \subset O_{\breve{F}} \rbrace.
\end{equation*}

For $O_{\breve{F}}$-lattices $M_i \subset M'_i \subset N_{k,i}$, we denote by $[M'_i:M_i]$ the index of $M_i$ in $M'_i$, i.e. the length of the $O_{\breve{F}}$-module $M'_i/M_i$. If $[M'_i:M_i]=t$, we write $M_i \overset{t}{\subset} M'_i$.\\

By the relative Dieudonne theory, we have the following proposition.

\begin{proposition}\label{proposition2.2} There is a bijection between the set $\CN(k)$ and the set of $O_{\breve{F}}$-lattices $M$ in $N_k$ such that
	
	$\bullet$ $M$ is stable under $\CF$, $\CV$, and $O_E$-action;
	
	$\bullet$ $\Charpol_k(a,M/\CV M)=(T-a)(T-a^{*})^{n-1}$ for all $ a\in O_E$;
	
	$\bullet$ $M_0 \overset{h}{\subset} M_1^{\perp} \overset{n-h}{\subset} \pi^{-1} M_0$, $M_1 \overset{h}{\subset} M_0^{\perp} \overset{n-h}{\subset} \pi^{-1} M_1$.
\end{proposition}

We will use the following lemma in the next subsection.

\begin{lemma}\label{lemma2.3}(\cite[Lemma 1.5]{Vol}) Let $M=M_0 \oplus M_1$ be an $O_E$-invariant lattice in $N_k$. Assume that $M$ is invariant under $\CF$ and $\CV$. Then $M$ satisfies the determinant condition of signature $(r,n-r)$ if and only if
	\begin{equation*}
	\begin{split}
		\pi M_0 \overset{n-r}{\subset} \CF M_1 \overset{r}{\subset} M_0,\\
		\pi M_1 \overset{r}{\subset} \CF M_0 \overset{n-r}{\subset} M_1.
	\end{split}
	\end{equation*}
\end{lemma}
\begin{proof}
	See \cite[Lemma 1.5]{Vol}.
\end{proof}

\bigskip

\subsection{Description of the points of $\mathcal{N}$ II}\label{subsec:section2.3}
In this subsection, we will describe the set $\mathcal{N}(k)$ as the set of lattices in $N_{k,0}$. We use the following notation.

 Let $\tau$ be the $\sigma^2$-linear operator $\CV^{-1}\CF$ on $N_k$, and let $N_{k,0}^{\tau}$ be the set of $\tau$-invariant elements in $N_{k,0}$. Then $N_{k,0}^{\tau}$ is an $E$-vector space. Note that for every $\tau$-invariant lattice $A$ in $N_{k,0}$, there exists a $\tau$-invariant basis of $A$ (see \cite[1.10]{Vol}). Therefore, we have $N_{k,0}=N_{k,0}^{\tau} \otimes_E \breve{F}$.

We define $\lbrace x, y \rbrace :=\langle x, \CF y \rangle$. This is a nondegenerate form on $N_{k,0}$ which is linear in the first variable, and $\sigma$-linear in the second variable.

Also, this form $\lbrace \cdot, \cdot \rbrace$ satisfies the following properties (see \cite[1.11]{Vol}):
\begin{equation*}
\lbrace x,y \rbrace=-\lbrace y, \tau^{-1}(x)\rbrace^{\sigma},
\end{equation*}
\begin{equation*}
\lbrace \tau(x),\tau(y) \rbrace=\lbrace x, y\rbrace^{\sigma^2}.
\end{equation*}

For an $O_{\breve{F}}$-lattice $A$ in $N_{k,0}$, we define $A^{\vee}$ the dual lattice of $A$ with respect to the form $\lbrace \cdot, \cdot \rbrace$ as
\begin{equation*}
A^{\vee}=\lbrace x \in N_{k,0} \vert \lbrace x, A \rbrace \subset O_{\breve{F}} \rbrace.
\end{equation*}

For an $O_{\breve{F}}$-lattice $A \subset N_{k,0}$, we have
\begin{equation*}
(A^{\vee})^{\vee}=\tau(A),
\end{equation*}
\begin{equation*}
\tau(A^{\vee})=\tau(A)^{\vee}.
\end{equation*}

We can now state the following description of $\mathcal{N}(k)$.

\begin{proposition}\label{proposition2.4}
	There is a bijection between $\mathcal{N}(k)$ and the set
	
	\begin{displaymath}
	\left\{\begin{array}{c} 
	O_{\breve{F}}$-lattices $ A \overset{h}{\subset} B \subset N_{k,0}
	\end{array}
	\middle|
	\begin{array}{c}
	\pi B^{\vee} \overset{1}{\subset} A \overset{n-1}{\subset} B^{\vee},\\
	\pi A^{\vee} \overset{1}{\subset} B \overset{n-1}{\subset} A^{\vee},\\
	\pi B \subset A \subset B.
	
	\end{array}
	\right\}
	\end{displaymath}
	
\end{proposition}
\begin{proof}
	For $M=M_0 \oplus M_1 \in \CN(k)$, let $A=M_0$, $B=M_1^{\perp}$. Then, by Proposition \ref{proposition2.2}, we have $\pi B \subset A \overset{h}{\subset} B$. Now, we will show the following equality.
	\begin{equation}\label{eq2.3.3}
	\pi(M_1^{\perp})^{\vee}=\CF M_1.
	\end{equation}
	
Indeed, we have
\begin{equation*}
(M_1^{\perp})^{\vee}=\lbrace y \in N_{k,0} \vert \lbrace y, M_1^{\perp} \rbrace \subset O_{\breve{F}} \rbrace
\end{equation*}
\begin{equation*}
=\lbrace y \in N_{k,0} \vert \langle y, \CF M_1^{\perp} \rangle \subset O_{\breve{F}} \rbrace
\end{equation*}
\begin{equation*}
=\lbrace y \in N_{k,0} \vert \langle \CF M_1^{\perp}, y  \rangle \subset O_{\breve{F}} \rbrace
\end{equation*}
\begin{equation*}
=\lbrace y \in N_{k,0} \vert \langle M_1^{\perp}, \CV y  \rangle \subset O_{\breve{F}} \rbrace
\end{equation*}
\begin{equation*}
=\CV^{-1}((M_1^{\perp})^{\perp})=\CV^{-1}M_1.
\end{equation*}

Therefore, by multiplying $\pi$, we get the equality \eqref{eq2.3.3}.

By Lemma \ref{lemma2.3} and \eqref{eq2.3.3}, we have $\pi B^{\vee} \overset{1}{\subset} A \overset{n-1}{\subset} B^{\vee}$.

Similarly, we have $\CV M_1 \overset{1}{\subset} M_0 \Longleftrightarrow M_1 \overset{1}{\subset} \CV^{-1}M_0 \Longleftrightarrow \CF M_1 \subset \CV^{-1}\CF(M_0) \Longleftrightarrow \pi(M_1^{\perp})^{\vee} \overset{1}{\subset} \tau(M_0) \Longleftrightarrow \pi M_0^{\vee} \overset{1}{\subset} M_1^{\perp}$. Therefore, we have $\pi A^{\vee} \overset{1}{\subset} B \overset{n-1}{\subset} A^{\vee}$.

Conversely, if we have $O_{\breve{F}}$-lattices $A,B$ satisfying the above conditions, then one can easily show that $A \oplus B^{\perp}$ is an element in $\CN(k)$.
\end{proof}

From now on, we identify $\CN(k)$ with the set defined in the Proposition \ref{proposition2.4}.

\bigskip

\subsection{The sets $R_{\Lambda}$, $S_{\Lambda}$ indexed by vertex lattices $\Lambda$.}\label{subsec:section2.4}

In this section, we will define the sets $R_{\Lambda}$ and $S_{\Lambda}$ indexed by the lattices $\Lambda$ which are called vertex lattices. First, we start with the definition of the vertex lattices.

\begin{definition}
	Let $\CL_i$ be the set of all lattices $\Lambda$ in $N_{k,0}^{\tau}$ (hence, $\tau$-invariant) satisfying $\pi^{i+1}\Lambda^{\vee} \subset \Lambda \subset \pi^i\Lambda^{\vee}$. An element in $\CL_i$ is called a \textit{vertex lattice}. We say that a vertex lattice $\Lambda \in \CL_i$ is \textit{of type $t$} if $\pi^{i+1}\Lambda^{\vee} \overset{t}{\subset} \Lambda$. We denote by $t(\Lambda)$ the type of the vertex lattice $\Lambda$.
	
\end{definition}

\begin{remark}
For $A \overset{h}{\subset} B$ a pair in $\CN(k)$, we define 
\begin{equation*}
T_iA:=A+\tau(A)+\cdots+\tau^{i-1}(A),
\end{equation*}
\begin{equation*}
T_iB:=B+\tau(B)+\cdots+\tau^{i-1}(B).
\end{equation*} Then, by \cite[Proposition 2.17]{RZ}, there exist positive integers $c,d$ such that $T_c(A)$ and $T_d(B)$ are $\tau$-invariant.
\end{remark}

Now, we will show the following lemma.

\begin{lemma}\label{lemma3.2} Let $A \overset{h}{\subset} B$ be a pair in $\CN(k)$. Let $c, d$ be the smallest positive integers such that $T_cA, T_dB$ are $\tau$-invariant, and write $\Lambda_A:=T_c(A)$, $\Lambda_B:=T_d(B)$. Then, at least one of the following assertions holds.

(1) $\Lambda_B$ is a vertex lattice in $\CL_0$, and
\begin{displaymath}
\begin{array}{cccccccc}
&\pi A^{\vee} &\overset{1}{\subset} &B &\subset& \Lambda_B &\subset &\Lambda_B^{\vee} \\
&\cup & &\cup & &&&\\
\pi \Lambda_B^{\vee} \subset&\pi B^{\vee}&\overset{1}{\subset}&A & &  &&
\end{array}
\end{displaymath}

(2) $\Lambda_A$ is a vertex lattice in $\CL_1$, and
\begin{displaymath}
\begin{array}{cccccccc}
 &\pi B^{\vee} &\overset{1}{\subset} &A &\subset& \Lambda_A &\subset &\pi \Lambda_A^{\vee} \\
 &\cup & &\cup & &&&\\
\pi^2 \Lambda_A^{\vee} \subset&\pi^2 A^{\vee}&\overset{1}{\subset}& \pi B& &  &&
\end{array}
\end{displaymath}

\end{lemma}

To prove the Lemma \ref{lemma3.2}, we need the following lemma.

\begin{lemma}\label{lemma3.3}
 For $1 \leq i <c$, $1 \leq j<d$,
	\begin{equation}\label{lemma3.3 eq1}
	T_{i}A \cap \tau(T_{i}A)=\tau(T_{i-1}A),
	\end{equation}
	\begin{equation}\label{lemma3.3 eq2}
	T_{i-1}A \overset{1}{\subset} T_{i}A,
	\end{equation}

	\begin{equation}\label{lemma3.3 eq3}
	T_{j}B \cap \tau(T_{j}B)=\tau(T_{j-1}B),
	\end{equation}
	\begin{equation}\label{lemma3.3 eq4}
	T_{j-1}B \overset{1}{\subset} T_{j}B.
	\end{equation}
	\end{lemma}

\begin{proof}
	We will show \eqref{lemma3.3 eq1}, \eqref{lemma3.3 eq2}. The proof of \eqref{lemma3.3 eq3}, \eqref{lemma3.3 eq4} is similar.
	
	Note that we have
	\begin{equation}\label{lemma3.3 eq5}
	\pi B^{\vee} \overset{1}{\subset} A \overset{n-1}{\subset} B^{\vee},
	\end{equation}
	\begin{equation}\label{lemma3.3 eq6}
	\pi A^{\vee} \overset{1}{\subset} B \overset{n-1}{\subset} A^{\vee}.
	\end{equation}
	
	Therefore, we have $\pi B^{\vee} \overset{1}{\subset} A$ and $\pi B^{\vee} \overset{1}{\subset} \tau(A)$ by taking the dual of \eqref{lemma3.3 eq6}. If $A$ is $\tau$-invariant, then $c=0$, and hence there is nothing to prove. Now assume that $A$ is not $\tau$-invariant. Since $\pi B^{\vee} \subset A \cap \tau(A) \subsetneq A$ and $\pi B^{\vee}$ is of index 1 in $A$, $A \cap \tau(A)$ should be $\pi B^{\vee}$. Also $A$ and $\tau(A)$ should have index 1 in $T_1A$. This shows \eqref{lemma3.3 eq2} when $i=1$.
	
	For \eqref{lemma3.3 eq1}, note that $\tau(A) \overset{1}{\subset} T_1A$ and $\tau(A) \overset{1}{\subset} \tau(T_1A)$. If $T_1A$ is $\tau$-invariant, then $c=1$. Therefore, there is nothing to show. Assume that $T_1A$ is not $\tau$-invariant. Then $T_1A \cap \tau(T_1A)=\tau(A)$. This shows \eqref{lemma3.3 eq1} for $i=1$.
	
	For arbitrary $i$, we can use the induction on $i$.
	\end{proof}

We now go back to the proof of Lemma \ref{lemma3.2}.

\begin{proof}[Proof of Lemma \ref{lemma3.2}]
	We will prove this lemma by dividing by 6 cases and their subcases.
	
	\textbf{Case 1}. If $B \in \CL_0$, then (1) holds.
	
	\textbf{Case 2}. If $A \in \CL_1$, then (2) holds.
	
	\textbf{Case 3}. Assume that $A$ is $\tau$-invariant, but not a vertex lattice in $\CL_1$. Then $A \nsubseteq \pi A^{\vee}$. Since $\pi A^{\vee}$ is of index 1 in $B$, and $A \subset B$, we have $B=A+\pi A^{\vee}.$ Since $A$ is $\tau$-invariant, $B$ is also $\tau$-invariant. Therefore, if $B \subset B^{\vee}$, then $B \in \CL_0$, and hence (1) holds. Therefore, it suffices to show that $B \subset B^{\vee}$. Assume that $B \nsubseteq B^{\vee}$. Since $\pi B^{\vee}$ is of index 1 in $A$ and $\pi B \subset A$, we have $A=\pi B + \pi B^{\vee}$. However, $\pi B^{\vee} \subset \pi A^{\vee}$ and $\pi B \subset \pi A^{\vee}$ implies that $A=\pi B + \pi B^{\vee} \subset \pi A^{\vee}$ which contradicts to our assumption that $A$ is not a vertex lattice.
	
	\textbf{Case 4}. Assume that $B$ is $\tau$-invariant, but not a vertex lattice in $\CL_0$. Then $B \nsubseteq B^{\vee}$. Since $\pi B^{\vee}$ is of index 1 in $A$ and $\pi B \subset A$, we have that $A=\pi B+ \pi B^{\vee}.$ In particular, $A$ is also $\tau$-invariant. Also, $\pi B^{\vee} \subset \pi A^{\vee}$ and $\pi B \subset \pi A^{\vee}$ implies that $A \subset \pi A^{\vee}$. Therefore, $A$ is vertex lattice in $\CL_1$ and (2) holds in this case.
	
	\textbf{Case 5}. Assume that $A,B$ are not $\tau$-invariant and $B \subset B^{\vee}$. In this case, we have
	\begin{equation}\label{3.4eq01}
	A \cap \tau(A)= \pi B^{\vee},
	\end{equation}
	\begin{equation}\label{3.4eq02}
	B \cap \tau(B)=\pi A^{\vee}.
	\end{equation}
	Also, note that
	\begin{equation*}
	B+\tau(B) \subset B^{\vee} \subset \pi^{-1} \tau(A),
	\end{equation*}
	\begin{equation*}
	\tau(B)+\tau^2(B) \subset \tau(B^{\vee}) \subset \pi^{-1} \tau(A).
	\end{equation*}
	Therefore, we have
	\begin{equation*}
	T_2B \subset \pi^{-1} \tau(A) \subset \pi^{-1} T_1A,
	\end{equation*}
	and,
	\begin{equation}\label{3.4eq4}
	T_dB \subset \pi^{-1} T_{d-1}A.
	\end{equation}
	
	\textbf{Case 5-1}. Assume that $d-1 < c$. Since $T_dB$ is $\tau$-invariant, \eqref{3.4eq4} implies that
	\begin{equation*}
	T_dB \subset \bigcap_{l \in \BZ}\pi^{-1} \tau^l(T_{d-1}A)\overset{\eqref{lemma3.3}}{=}\bigcap_{l \in \BZ}\pi^{-1} \tau^l(A)\overset{\eqref{3.4eq01}}{=}\bigcap_{l \in \BZ}\pi^{-1} \tau^l(\pi B^{\vee})=(T_dB)^{\vee}.
	\end{equation*}
	The last equality is induced by
	\begin{equation*}
	(T_dB)^{\vee}=B^{\vee} \cap \tau(B^{\vee}) \cap \cdots \cap \tau^{d-1}(B^{\vee}),
	\end{equation*}
	and the fact that $(T_dB)^{\vee}$ is $\tau$-invariant. Therefore, (1) holds in this case.

	\textbf{Case 5-2}. Assume that $d-1 \geq c$. Then, $T_cA \subset T_cB$ and $T_cA$ is $\tau$-invariant. Therefore, we have
	\begin{equation*}
	T_cA \subset \bigcap_{l \in \BZ} \tau^l(T_{c}B)\overset{\eqref{lemma3.3}}{=}\bigcap_{l \in \BZ} \tau^l(B)\overset{\eqref{3.4eq02}}{=}\bigcap_{l \in \BZ} \tau^l(\pi A^{\vee})=\pi (T_cA)^{\vee}.
	\end{equation*}
	The last equality is induced by
		\begin{equation*}
	(T_cA)^{\vee}=A^{\vee} \cap \tau(A^{\vee}) \cap \cdots \cap \tau^{c-1}(A^{\vee}),
	\end{equation*}
	and the fact that $(T_cA)^{\vee}$ is $\tau$-invariant. Therefore, (2) holds in this case.

	\textbf{Case 6}. Assume that $A,B$ are not $\tau$-invariant and $B \nsubseteq B^{\vee}$. In this case, $\eqref{3.4eq01}$ and $\eqref{3.4eq02}$ hold and we have $A=\pi B + \pi B^{\vee} \subset \pi A^{\vee}$ (see the case 4). By $\eqref{3.4eq02}$, we have $A \subset B$ and $A \subset \tau(B)$. Therefore, $T_1A \subset \tau(B)$ and
	\begin{equation*}
	T_cA \subset \tau(T_{c-1}B).
	\end{equation*}
	
	\textbf{Case 6-1} Assume that $c \leq d$. Then, we have
	\begin{equation*}
	T_cA \subset \bigcap_{l \in \BZ} \tau^l(T_{c-1}B)\overset{\eqref{lemma3.3}}{=}\bigcap_{l \in \BZ} \tau^l(B)\overset{\eqref{3.4eq02}}{=}\bigcap_{l \in \BZ} \tau^l(\pi A^{\vee})=\pi (T_cA)^{\vee}.
	\end{equation*}
	Therefore, (2) holds in this case.
	
	\textbf{Case 6-2} Assume that $d<c$. Then, $B \subset \pi^{-1}A$ implies that $T_dB \subset \pi^{-1}T_dA$. Therefore, we have
		\begin{equation*}
	T_dB \subset \bigcap_{l \in \BZ}\pi^{-1} \tau^l(T_{d}A)\overset{\eqref{lemma3.3}}{=}\bigcap_{l \in \BZ}\pi^{-1} \tau^l(A)\overset{\eqref{3.4eq01}}{=}\bigcap_{l \in \BZ}\pi^{-1} \tau^l(\pi B^{\vee})=(T_dB)^{\vee}.
	\end{equation*}
	This is a contradiction, since $B \nsubseteq B^{\vee}$ and $B \subset T_dB \subset (T_dB)^{\vee} \subset B^{\vee}$.
	
	This completes the proof of the Lemma \ref{lemma3.2}.
\end{proof}

Now, let us give the definition of the sets $R_{\Lambda}(k), S_{\Lambda}(k)$.

\begin{definition}
	(1) For a vertex lattice $\Lambda \in \CL_1$, we define the set
	
	\begin{displaymath}
R_{\Lambda}(k):=\left\{\begin{array}{c} 
O_{\breve{F}}$-lattices$\\
 A \overset{h}{\subset} B \subset N_{k,0}\\
 
\end{array}
\middle|
\begin{array}{cccccccc}
(i)&\pi B^{\vee} &\overset{1}{\subset} &A &\subset& \Lambda &\subset &\pi \Lambda^{\vee} \\
&\cup & &\cup & &&&\\
\pi^2 \Lambda^{\vee} \subset&\pi^2 A^{\vee}&\overset{1}{\subset}& \pi B& &  &&\\
(ii)&\pi B&\overset{n-h}{\subset}&  A & \overset{h}{\subset}&B&
\end{array}
\right\}
\end{displaymath}

	(2) For a vertex lattice $\Lambda \in \CL_0$, we define the set
\begin{displaymath} 
S_{\Lambda}(k):=\left\{\begin{array}{c} 
O_{\breve{F}}$-lattices$\\
A \overset{h}{\subset} B \subset N_{k,0}
\end{array}
\middle|
\begin{array}{cccccccc}
(i)&\pi A^{\vee} &\overset{1}{\subset} &B &\subset& \Lambda &\subset &\Lambda^{\vee} \\
&\cup & &\cup & &&&\\
\pi \Lambda^{\vee} \subset&\pi B^{\vee}&\overset{1}{\subset}&A & &  &&\\

(ii)&\pi B&\overset{n-h}{\subset}&  A & \overset{h}{\subset}&B&
\end{array}
\right\}
\end{displaymath}

\end{definition}

\begin{proposition}\label{proposition3.7}
	We have $\CN(k)=\bigcup_{\Lambda \in \CL_1} R_{\Lambda}(k) \cup \bigcup_{\Lambda \in \CL_0} S_{\Lambda}(k).$
\end{proposition}
\begin{proof}
	This is clear from the Lemma \ref{lemma3.2}.
\end{proof}

\begin{proposition}\label{proposition3.6}
If $\Lambda \in \CL_0$ and $S_{\Lambda}$ is not empty, then  $h+1 \leq t(\Lambda) \leq n$, and $t({\Lambda}) \equiv h+1$ mod 2.
\end{proposition}
\begin{proof}
	This is clear from the Lemma \ref{lemma3.2} (1).
\end{proof}

\begin{proposition}\label{proposition3.9}
	If $\Lambda \in \CL_1$ and $R_{\Lambda}$ is not empty, then $n-h+1 \leq t(\Lambda) \leq n$, and $t({\Lambda}) \equiv n-h+1$ mod 2.
\end{proposition}
\begin{proof}
	This is clear from the Lemma \ref{lemma3.2} (2).
\end{proof}

\begin{definition}
	We write $\CL_0^+$ for the set of lattices in $\CL_0$ with $t(\Lambda) \geq h+1$ and $\CL_0^-$ for the set of lattices in $\CL_0$ with $t(\Lambda) \leq h-1$. Similarly, we denote by $\CL_1^+$ the set of lattices in $\CL_1$ with $t(\Lambda) \geq n-h+1$ and $\CL_1^-$ the set of lattices in $\CL_1$ with $t(\Lambda) \leq n-h-1$.
\end{definition}

\begin{remark}\label{remark3.14}
	For $\Lambda_1 \in \CL_1^+$, we have $\pi (\pi \Lambda_1^{\vee})^{\vee}=\Lambda_1 \subset \pi \Lambda_1^{\vee} \subset \pi^{-1} \Lambda_1=(\pi \Lambda_1^{\vee})^{\vee}$. Therefore, we can regard $\pi \Lambda_1^{\vee}$ as the element of $\CL_0$. By this identification, we have a bijection from $\CL_0^+ \sqcup \CL_0^-$ to $\CL_0^+ \sqcup \CL_1^+$ by sending $\Lambda \in \CL_0^+$ to $\Lambda$, and $\Lambda \in \CL_0^-$ to $\pi \Lambda^{\vee}$.
\end{remark}

\begin{remark}\label{remark3.10}
	When $h=0$ (the case in \cite{VW}), $R_{\Lambda}(k)$ does not occur in $\CN(k)$ (by Proposition \ref{proposition3.9}). When $h=1$, for any pair $(A,B) \in R_{\Lambda}(k)$, $A$ should be $\Lambda$ and $t(\Lambda)=n$. In this case, $B$ can be any lattice satisfying $\Lambda \overset{1}{\subset} B \subset \pi^{-1} \Lambda$. Hence, we have $R_{\Lambda}(k) \simeq \BP^{n-1}(k)$. We should note that Kudla and Rapoport already proved this result in their unpublished notes \cite{KR4}.
\end{remark}

\begin{proposition}\label{proposition3.10}
	Let $\Lambda_1, \Lambda_2$ be elements in $\CL_0^+$.
	
	(1) If $\Lambda_1 \subset \Lambda_2$, then $S_{\Lambda_1}(k) \subset S_{\Lambda_2}(k).$
	
	(2) If $\Lambda_1 \cap \Lambda_2$ is in $\CL_0^+$, then $S_{\Lambda_1}(k) \cap S_{\Lambda_2}(k)=S_{\Lambda_1 \cap \Lambda_2}(k)$. Otherwise, it is empty.
	
\end{proposition}
\begin{proof}
	(1) is clear from its definition. 
	
	For (2), we will show that $S_{\Lambda_1}(k) \cap S_{\Lambda_2}(k) \subset S_{\Lambda_1 \cap \Lambda_2}(k)$. Let $(A,B)$ be the element in $S_{\Lambda_1}(k) \cap S_{\Lambda_2}(k).$ Note that $(A,B)$ satisfies the following diagrams,
	\begin{equation*}
	\begin{array}{cccccccc}
	&\pi A^{\vee} &\overset{1}{\subset} &B &\subset& \Lambda_1 & \subset&\Lambda_1^{\vee}\\
	&\cup & &\cup & &&&\\
	\pi \Lambda_1^{\vee} \subset&\pi B^{\vee}&\overset{1}{\subset}&A & &  &&,
	\end{array}
	\end{equation*}
	and
	\begin{equation*}
	\begin{array}{cccccccc}
	&\pi A^{\vee} &\overset{1}{\subset} &B &\subset& \Lambda_2 & \subset &\Lambda_2^{\vee}\\
	&\cup & &\cup & &&&\\
	\pi \Lambda_2^{\vee} \subset&\pi B^{\vee}&\overset{1}{\subset}&A & &  &&.
	\end{array}
	\end{equation*}
	
	These two diagrams imply that
	\begin{equation*}
	\begin{array}{cccccccc}
	&&\pi A^{\vee} &\overset{1}{\subset} &B &\subset& \Lambda_1 \cap\Lambda_2\subset \Lambda_1^{\vee}\subset (\Lambda_1 \cap\Lambda_2)^{\vee} \\
	&&\cup & &\cup & &\\
	\pi (\Lambda_1 \cap\Lambda_2)^{\vee}=&\pi \Lambda_1^{\vee}+\pi \Lambda_2^{\vee} \subset&\pi B^{\vee}&\overset{1}{\subset}&A& &  .
	\end{array}
	\end{equation*}
	Therefore, $\Lambda_1 \cap \Lambda_2$ is in $\CL_0^+$, and $(A,B)$ should be contained in $S_{\Lambda_1 \cap \Lambda_2}(k)$.
	
	Conversely, $ S_{\Lambda_1 \cap \Lambda_2}(k)\subset S_{\Lambda_1}(k) \cap S_{\Lambda_2}(k)$ is obvious from (1).
	 This completes the proof of the proposition.
\end{proof}

\begin{proposition}\label{proposition3.11}
	Let $\Lambda_1, \Lambda_2$ be elements in $\CL_1^+$.
	
	(1) If $\Lambda_1 \subset \Lambda_2$, then $R_{\Lambda_1}(k) \subset R_{\Lambda_2}(k).$
	
	(2) If $\Lambda_1 \cap \Lambda_2$ is in $\CL_1^+$, then $R_{\Lambda_1}(k) \cap R_{\Lambda_2}(k)=R_{\Lambda_1 \cap \Lambda_2}(k)$. Otherwise, it is empty.
	
\end{proposition}

\begin{proof}
	The proof is the same as the proof of Proposition \ref{proposition3.10}
\end{proof}

Now, let us consider the intersection $R_{\Lambda_1}(k) \cap S_{\Lambda_0}(k).$

\begin{proposition}\label{proposition3.13}
	Let $\Lambda_1 \in \CL_1^+, \Lambda_0 \in \CL_0^+$.
	
	(1) If $\pi \Lambda_1^{\vee} \nsubseteq \Lambda_0$, then $R_{\Lambda_1}(k) \cap S_{\Lambda_0}(k)=\emptyset$.
	
	(2) If $\pi \Lambda_1^{\vee} \subset \Lambda_0$, then\\
	\begin{equation*}
	R_{\Lambda_1}(k) \cap S_{\Lambda_0}(k)=\left\{\begin{array}{c} 
	O_{\breve{F}}$-lattices$\\
	A \overset{h}{\subset} B \subset N_{k,0}
	\end{array}
	\middle|
	\begin{array}{cccccccc}
	
	& \pi \Lambda_1^{\vee}&\subset&\pi A^{\vee}&\overset{1}{\subset} & B & \subset&\Lambda_0\\
	&\cup&&&&&&\\
	& \Lambda_1&\supset&A&\overset{1}{\supset} &  \pi B^{\vee}& \supset&\pi \Lambda_0^{\vee} \\
	\end{array}
	\right\}
	\end{equation*}
	\end{proposition}
\begin{proof}
This is clear from the definition.
\end{proof}

\begin{remark}\label{remark3.15}
	Let $h=1$, $\Lambda_1 \in \CL_1^+, \Lambda_0 \in \CL_0^+$, and $\pi \Lambda_1^{\vee} \subset \Lambda_0$. For any $(A,B) \in R_{\Lambda_1}(k)$, we have $A=\Lambda_1$ by Remark \ref{remark3.10}. Therefore,
	\begin{equation*}
		R_{\Lambda_1}(k) \cap S_{\Lambda_0}(k)=\left\{\begin{array}{c} 
	O_{\breve{F}}$-lattices$\\
	 B \subset N_{k,0}
	\end{array}
	\middle|
	\begin{array}{c}
	\pi \Lambda_1^{\vee} \overset{1}{\subset} B \subset \Lambda_0
	\end{array}
	\right\}.
	\end{equation*}
	
	This is isomorphic to $\BP^{m-1}(k)$, where $m=[\Lambda_0 : \pi \Lambda_1^{\vee}]$.
\end{remark}

\begin{remark}\label{remark2.20}
	We can apply our method for $\CN_{E/F}^0(2,2)$ which has been studied in \cite{HP}. We should note that all of the following descriptions of $k$-points is already obtained in loc.cit. with a different method. 
	
	By using the relative Dieudonne theory and similar steps in Section \ref{sec:section2}, we can show that there is a bijection between $\mathcal{N}(k)$ and the set
	\begin{displaymath}
	\left\{\begin{array}{c} 
	O_{\breve{F}}\text{-lattice } B \subset N_{k,0}
	\end{array}
	\middle|
	\begin{array}{c}
	\pi B^{\vee} \overset{2}{\subset} B \overset{2}{\subset} B^{\vee}
	
	\end{array}
	\right\}
	\end{displaymath}
	We can divide the set into three cases.
	
	\textbf{case 1} $B \cap \tau(B) \overset{1}{\subset} B$.
	
	\textbf{case 2} $B \cap \tau(B)=\pi B^{\vee}$ and $B \overset{1}{\subset} T_1B$.
	
	\textbf{case 3} $B \cap \tau(B)=\pi B^{\vee}$ and $T_1B=B^{\vee}$.
	
	In case 1, let $\pi A^{\vee}=B \cap \tau(B)$. Then, the pair $(A,B)$ satisfies
	\begin{displaymath}
	\begin{array}{c} 
	\pi A^{\vee} \overset{1}{\subset} B \overset{3}{\subset} A^{\vee};\\
	\pi B^{\vee} \overset{1}{\subset} A \overset{3}{\subset} B^{\vee};\\
	\pi B \overset{3}{\subset} A \overset{1}{\subset} B.
	\end{array}
	\end{displaymath}
	Therefore, by using Lemma \ref{lemma3.2}, we can show that at least one of the following is true.
	\begin{enumerate}
		\item $A$ is $\tau$-invariant and $A=\pi A^{\vee}$.
		\item $\Lambda_B \subset \Lambda_B^{\vee}$.
	\end{enumerate}
	
	In case 2, one can prove that $\Lambda_B \subset \Lambda_B^{\vee}.$
	
	In case 3, since $B \cap \tau(B) =\pi B^{\vee}$, we have $B^{\vee} + \tau(B^{\vee}) =\pi^{-1} \tau(B)$ by taking dual. Since $B^{\vee}=B+\tau(B)$, we have
	\begin{equation*}
	B+\tau(B)+\tau^2(B)=\pi^{-1} \tau(B).
	\end{equation*}
	Let $d$ be the smallest integer such that $T_dB$ is $\tau$-invariant. Then $T_dB=\pi^{-1}\tau(T_{d-2}B)$ is $\tau$-invariant, and this means that $T_{d-2}B$ is also $\tau$-invariant. This is possible only when $B$ is $\tau$-invariant.
	
	In summary, $B \cap \tau(B)$ is a vertex lattice of type $0$ or $\Lambda_B \subset \Lambda_B^{\vee}$ (hence $\Lambda_B$ is a vertex lattice). This is the analogue of Lemma \ref{lemma3.2}.
	
	Therefore, for each vertex lattice $\Lambda$, we can attach the following set.
	
	(1) If $\Lambda=\pi\Lambda^{\vee}$, then we attach the set,
	\begin{displaymath} 
	\left\{\begin{array}{c} 
	O_{\breve{F}}$-lattices$\\ 
 B \subset N_{k,0}
	\end{array}
	\middle|
	\begin{array}{c}
	\Lambda \overset{1}{\subset} B \overset{2}{\subset} B^{\vee} \overset{1}{\subset} \Lambda^{\vee}
	\end{array}
	\right\}.
	\end{displaymath}
	This is the set of $k$-points of a Fermat hypersurface.
	
	(2) If $\pi \Lambda^{\vee} \overset{2}{\subset} \Lambda$, then we attach the set,
		\begin{displaymath} 
	\left\{\begin{array}{c} 
	O_{\breve{F}}$-lattices$\\ 
	B \subset N_{k,0}
	\end{array}
	\middle|
	\begin{array}{c}
B=\Lambda
	\end{array}
	\right\}.
	\end{displaymath}
	This is one $k$-point.
	
	(3) If $\pi \Lambda^{\vee} \overset{4}{\subset} \Lambda$, then we attach the set,
		\begin{displaymath} 
	\left\{\begin{array}{c} 
	O_{\breve{F}}$-lattices$\\ 
	B \subset N_{k,0}
	\end{array}
	\middle|
	\begin{array}{c}
	\pi \Lambda^{\vee} \overset{1}{\subset} B \overset{2}{\subset} B^{\vee} \overset{1}{\subset} \Lambda = \Lambda^{\vee}
	\end{array}
	\right\}.
	\end{displaymath}
	This is the set of $k$-points of a Fermat hypersurface.
	
$\CN(k)$ is the union of the above sets and this is the same result as in \cite{HP}.
	
\end{remark}

\bigskip

\section{Subschemes $\CN_{\Lambda}$ of $\CN$}\label{sec:section3}
In this section, we will first define the subscheme $\CN_{\Lambda}$ for each vertex lattice $\Lambda$, and prove that $\CN_{\Lambda}$ is isomorphic to a generalized Deligne-Lusztig variety. Also, we will prove the regularity of $\CN^h_{E/F}(1,n-1) \otimes O_{\breve{E}}$. Before we begin, let us introduce some notation. In the end of the Section \ref{subsec:section2.1}, we showed that $\CN^h_{E/F}(1,n-1) \otimes O_{\breve{E}}$ is representable by a formal scheme over $\Spf O_{\breve{E}}$ and furthermore, $\CN^h_{E/F}(1,n-1)$ is representable by a formal scheme over $\Spf O_E$ if $F$ is unramified over $\BQ_p$. For this reason, we will use the following notation. Let $\BF=\BF_{q^2}$ if $F$ is an unramified extension of $\BQ_{p}$, and let $\BF=\overline{\BF}_{q^2}$ if $F$ is ramified over $\BQ_p$. Then $\CN^h_{E/F}(1,n-1) \otimes_{O_E} \BF$ is the special fiber of $\CN^h_{E/F}(1,n-1)$ (resp. $\CN^h_{E/F}(1,n-1) \otimes O_{\breve{E}}$) if $F$ is unramified over $\BQ_p$ (resp. if $F$ is ramified over $\BQ_p$).

\subsection{Strict formal $O_F$-modules $X_{\Lambda ^+}$ and $X_{\Lambda ^-}$}\label{subsec:section3.1}

In this subsection, we fix a vertex lattice $\Lambda \in \CL_i^+$, for $i=0,1$. We will define the strict formal $O_F$-modules $X_{\Lambda ^+}$, $X_{\Lambda ^-}$ over $\BF_{q^2}$ with $O_E$-action, polarizations $\lambda_{\Lambda ^{\pm}}$ and quasi-isogenies $\rho_{\Lambda^{\pm}}:X_{\Lambda^{\pm}} \rightarrow \BX$. For this, we will construct the following two Dieudonne submodules of $N$.

First, if $\Lambda \in \CL_0^+$, we define the lattices $\Lambda^+$ and $\Lambda^-$ by
\begin{equation*}
\Lambda_0^+=\Lambda
\end{equation*}
\begin{equation*}
\Lambda_1^+=\CV^{-1}(\Lambda)
\end{equation*}
\begin{equation*}
\Lambda_0^-=\pi \Lambda^{\vee}
\end{equation*}
\begin{equation*}
\Lambda_1^-=\CV(\Lambda^{\vee})
\end{equation*}
\begin{equation*}
\Lambda^+=\Lambda_0^+ \oplus \Lambda_1^+
\end{equation*}
\begin{equation*}
\Lambda^-=\Lambda_0^- \oplus \Lambda_1^-
\end{equation*}

Then, one can easily show that $\Lambda^-=(\Lambda^+)^{\perp}$. Since $\CF=\CV$ on $\Lambda^+$ and $\Lambda^-$, we have that $\Lambda^+$ and $\Lambda^-$ are Dieudonne submodules of $N$.

In case $\Lambda \in \CL_1^+$, we define the lattices $\Lambda^+$ and $\Lambda^-$ by
\begin{equation*}
\Lambda_0^+=\Lambda
\end{equation*}
\begin{equation*}
\Lambda_1^+=\CV^{-1}(\Lambda)
\end{equation*}
\begin{equation*}
\Lambda_0^-=\pi^2 \Lambda^{\vee}
\end{equation*}
\begin{equation*}
\Lambda_1^-=\pi \CV(\Lambda^{\vee})
\end{equation*}
\begin{equation*}
\Lambda^+=\Lambda_0^+ \oplus \Lambda_1^+
\end{equation*}
\begin{equation*}
\Lambda^-=\Lambda_0^- \oplus \Lambda_1^-
\end{equation*}

Then, we have $\Lambda^-=\pi(\Lambda^+)^{\perp}$. Again, these $\Lambda^+$ and $\Lambda^-$ are Dieudonne submodules of $N$.

For $\Lambda \in \CL_i^+$, we have $\Lambda \subset \pi^i \Lambda^{\vee}$. Therefore, the pairing $\pi^{-i+1} \langle \cdot, \cdot \rangle$ on $N$ induces a $W_{O_F}(\BF_{q^2})$-pairing on $\Lambda^+$ and $\Lambda^-$.

Now, let $X_{\Lambda^+}$ and $X_{\Lambda^-}$ be the strict formal $O_F$-modules associated to $\Lambda^+$ and $\Lambda^-$ with quasi-isogenies $\rho_{\Lambda^{\pm}}:X_{\Lambda^{\pm}} \rightarrow \BX$.

We will use these two strict formal $O_F$-modules to define the subschemes $\CN_{\Lambda}$ of $\CN$.

\subsection{Subschemes $\CN_{\Lambda}$ attached to vertex lattices $\Lambda$}\label{subsec:section3.2}

We fix $\Lambda \in \CL_i^+$, for $i=0,1$. Let $S$ be a $\BF$-scheme. We define $\CN_\Lambda$ as the subfunctor of $\CN \otimes_{O_E}\BF$ consisting of tuples $(X,i_X,\lambda_X,\rho_X) \in \CN(S)$ such that
\begin{equation*}
\rho_{X,\Lambda^+}:X \xrightarrow{\rho_X} \BX_s \xrightarrow{(\rho_{\Lambda^+})^{-1}_S} (X_{\Lambda^+})_S
\end{equation*}
\begin{equation*}
\rho_{X,\Lambda^-}:(X_{\Lambda^-})_S  \xrightarrow{(\rho_{\Lambda^-})_S} \BX_S  \xrightarrow{\rho_X^{-1}} X
\end{equation*}
are isogenies.

We have the following lemma.
\begin{lemma}\label{lemma4.1}
	The functor $\CN_{\Lambda}$ is representable by a projective $\BF$-scheme and the monomorphism $\CN_{\Lambda} \hookrightarrow \CN\otimes \BF$ is a closed immersion.
\end{lemma}
\begin{proof}
	See \cite[Lemma 4.2]{VW}.
\end{proof}

\begin{lemma}\label{lemma4.2}
	If $\Lambda \in \CL_0^+$, then $\CN_{\Lambda}(k)=S_{\Lambda}(k)$, and if $\Lambda \in \CL_1^+$, then $\CN_{\Lambda}(k)=R_{\Lambda}(k)$.
\end{lemma}
\begin{proof}
	This is clear from the definition of $\CN_{\Lambda}$.
\end{proof}

\subsection{Deligne-Lusztig varieties}\label{subsec:section3.3}
In this subsection, we will recall some results about Deligne-Lusztig varieties.

Let $G$ be a connected reductive group over a finite field $\FK$. Denote by $G_{\overline{\FK}}$ the base change of $G$ over $\overline{\FK}$, where $\overline{\FK}$ is a fixed algebraic closure of $\FK$. Let $\CF:G \rightarrow G$ be the Frobenius morphism with respect to $\FK$, and let $(W,S)$ be the Weyl system of $G_{\overline{\FK}}$. Then $\CF$ gives an automorphism on $W$. By Lang's theorem, $G$ is quasi-split, and hence $\CF(S)=S$.

For $I \subset S$, let $W_I$ be the subgroup of $W$ generated by $I$, and let $P_I=BW_IB$ be the corresponding standard parabolic subgroup of $G$.

For $I$, $J$ $\subset S$, we denote by $^IW^J$ the set of minimal length representatives $w\in W$ in the double coset $W_I \backslash W / W_J$.

Now, we define the generalized Deligne-Lusztig varieties as follows.

\begin{definition}
	Let $I \subset S$. For each $w \in W$, we define the generalized \textit{Deligne-Lusztig variety} $X_I(w)$ by
	\begin{equation*}
	X_I(w):=\{g \in G/P_I : g^{-1}\CF(g) \in P_IwP_{\CF(I)} \}.
	\end{equation*}
\end{definition}

We will need the following two results later.

\begin{proposition}\label{proposition4.4}
	(\cite[Lemma 2.1.3]{Hoe}) For $w\in$$^IW^{\CF(I)}$, the Deligne-Lusztig variety $X_I(w)$ is smooth of dimension $l(w)+l(W_{\CF(I)})-l(W_{I \cap^w\CF(I)})$, where $l(w)$ is the length of $w$, $l(W_I)=\max\lbrace l(w') \vert w' \in W_I \rbrace$, and $^w\CF(I)=w\CF(I)w^{-1}$.
\end{proposition}

\begin{proposition}\label{proposition4.5}
	(\cite{BR}) The following assertions are equivalent.\\
	(1) $X_I(w)$ is geometrically irreducible.\\
	(2) $X_I(w)$ is connected.\\
	(3) There exists no $J\subsetneq S$ with $\CF(J)=J$ such that $W_Iw \subset W_J$.
\end{proposition}

\subsection{The Deligne-Lusztig variety $Y_{\Lambda}$}\label{subsec:section3.4}
In this subsection, we will define the Deligne-Lusztig variety $Y_{\Lambda}$. For $i=0,1$ we fix a vertex lattice $\Lambda \in \CL_i^+$. We use the following notation. 

$\bullet$ Let $V_{\Lambda}$ be $\Lambda_0^+/\Lambda_0^-$ and let $(\cdot , \cdot )$ be the skew-hermitian form on $V_{\Lambda}$ induced by $\pi^{-i}\lbrace \cdot, \cdot \rbrace$. Note that $V_{\Lambda}$ is a $\BF_{q^2}$-vector space of dimension $d:=t(\Lambda)$.

$\bullet$ Let $J_{\Lambda}$ be the special unitary group associated to $(V, (\cdot, \cdot))$. This is a connected reductive group over $\BF_q$.

$\bullet$ Let $\CF:J_{\Lambda} \rightarrow J_{\Lambda}$ be the Frobenius morphism over $\BF_q$ and $(W,S)$ be the Weyl system of $J_{\Lambda}$. 

Note that 
\begin{equation*}
J_{\Lambda} \otimes_{\BF_q} \BF_{q^2} \simeq SL(V_{\Lambda})=SL_{d,\BF_{q^2}}.
\end{equation*}
Therefore, we can identify $W$ with the symmetric group $S_d$, and $S$ with $\lbrace s_1, \dots, s_d \rbrace$, where $s_i$ is the transposition of $i$ and $i+1$.

The Frobenius $\CF$ induces an automorphism of $W$, and this is given by the conjugation with $w_0 \in S_d$, where $w_0(i)=d+1-i$ for all $i$.

$\bullet$ For a $\BF_{q^2}$-algebra $R$, we denote by $V_{\Lambda, R}$ the base change $V_{\Lambda} \otimes_{\BF^{q^2}}R$. Let $\sigma$ be the Frobenius of $R$. For a $R$-module M, denote by $M^{(\sigma)}=M \otimes_{R,\sigma}R$, the Frobenius twist, and denote by $M^*=Hom_R(M,R)$. Let $U$ be a locally direct summand of $V_{\Lambda,R}$ of rank $m$. We define its dual module $U^{\curlyvee}$ as follows.
 Since $(\cdot,\cdot)$ induces an $R$-linear isomorphism
\begin{equation*}
\psi:(V_{\Lambda,R})^{(\sigma)} \simeq (V_{\Lambda,R})^*,
\end{equation*} $\psi(U^{(\sigma)})$ is a locally direct summand of $(V_{\Lambda,R})^*$ of rank $m$. Let $U^{\curlyvee}$ be the kernel of the composition
 \begin{equation*}
 V_{\Lambda,R} \simeq (V_{\Lambda,R})^{**} \twoheadrightarrow \psi(U^{(\sigma)})^*.
 \end{equation*}
This is a locally direct summand of $V_{\Lambda,R}$ of rank $d-m$.

In particular, if $R=k$, then
\begin{equation*}
U^{\curlyvee}=\lbrace x \in V_{\Lambda,k}:(x,U)=0\rbrace.
\end{equation*}

\begin{remark}
	Let $R=k$. For a lattice $A$ such that $\pi^{i+1}\Lambda^{\vee} \subset A \subset \Lambda$, the quotient $A/\pi^{i+1}\Lambda^{\vee}$ is a subspace of $V_{\Lambda,k}$. Then by definition, we have
	\begin{equation*}
	\pi^{i+1}A^{\vee}/\pi^{i+1}\Lambda^{\vee}=(A/\pi^{i+1}\Lambda^{\vee})^{\curlyvee}.
	\end{equation*}
\end{remark}

We will need the following lemma.

\begin{lemma}\label{lemma4.7}
	(\cite[Lemma 2.17]{Vol}) Fix $I \subset S$, and let $\FF \Fl$ be a flag in $J_{\Lambda}/P_I$. Then the Frobenius $\CF$ and the duality morphism $\FF \Fl \mapsto \FF \Fl^{\curlyvee}$ define the same morphism $J_{\Lambda}/P_I \rightarrow J_{\Lambda}/P_{\CF(I)}$, i.e. the dual flag $\FF \Fl^{\curlyvee}$ is equal to $\CF(\FF \Fl)$.
\end{lemma}

 Let $\Lambda \in \CL_0^+$ and $d=2l+h+1$ (recall that $h$ is from $\CN_{E/F}^h(1,n-1)$). We can take the set $I_{\Lambda} \subset S$ such that the elements in $J_{\Lambda}/P_{I_{\Lambda}}$ parametrize flags
 \begin{equation*}
 0 \overset{l+1}{\subset} \overline{A} \overset{h}{\subset} \overline{B} \overset{l}{\subset} V_{\Lambda},
 \end{equation*} 
  where $\overline{A}, \overline{B}$ are subspaces of $V_{\Lambda}$. For example, we take
  \begin{equation*}
  I_{\Lambda}=\lbrace s_1, \dots, s_l,s_{l+2},\dots,s_{l+h},s_{l+h+2},\dots,s_{2l+h} \rbrace,
  \end{equation*}
   where $h>1$, $l>1$.

In case $\Lambda \in \CL_1^+$, and $d=2l+(n-h)+1$, we take $I_{\Lambda} \subset S$ such that the elements in $J_{\Lambda}/P_{I_{\Lambda}}$ parametrize flags
\begin{equation*}
0 \overset{l+1}{\subset} \overline{\pi B} \overset{n-h}{\subset} \overline{A} \overset{l}{\subset} V_{\Lambda},
\end{equation*}

where $\overline{\pi B}, \overline{A}$ are subspaces of $V_{\Lambda}$.\\
 
\begin{definition}
In case $h=0, n$, we define $w_{\Lambda}=id$. In case $1 \leq h \leq n-1$, we define $w_{\Lambda}$ as follows. If $\Lambda \in \CL_0^+$, we define $w_{\Lambda}=s_{l+1}s_{l+2}\dots s_{l+h}$ or $w_{\Lambda}=(l+1,l+h+1)$, the transposition of $l+1$ and $l+h+1$. Note that these two $w_{\Lambda}$ gives the same coset in $W_{I_{\Lambda}} w_{\Lambda} W_{\CF(I_{\Lambda})}$. In case $\Lambda \in \CL_1^+$, we define $w_{\Lambda}=s_{l+1}s_{l+2}\dots s_{l+n-h}$.
\end{definition}

 Then we have the following proposition.
 
 \begin{proposition}\label{proposition4.8} We have the following bijections.
 \begin{enumerate}
 	\item  If $1 \leq h \leq n-1$ and $\Lambda \in \CL_0^+$, then
 	\begin{equation*}
 	S_{\Lambda}(k)=X_{I_{\Lambda}}(id)(k) \sqcup X_{I_{\Lambda}}(w_{\Lambda})(k).
 	\end{equation*}
 	
 	\item If $1 \leq h \leq n-1$ and $\Lambda \in \CL_1^+$, then
 	\begin{equation*}
 	R_{\Lambda}(k)=X_{I_{\Lambda}}(id)(k) \sqcup X_{I_{\Lambda}}(w_{\Lambda})(k)
  \end{equation*}
  
   \item If $h=0$ and $\Lambda \in \CL_0^+$, then 
   	\begin{equation*}
   S_{\Lambda}(k)=X_{I_{\Lambda}}(id)(k).
   \end{equation*}
   
   \item If $h=n$ and $\Lambda \in \CL_1^+$, then
   	\begin{equation*}
   R_{\Lambda}(k)=X_{I_{\Lambda}}(id)(k).
   \end{equation*}
\end{enumerate}	
 \end{proposition}

 \begin{proof}
 	(1) Let $(A \subset B) \in S_{\Lambda}(k)$. By sending this to
 	$(A/\pi \Lambda^{\vee} \subset B/ \pi \Lambda^{\vee})$, we have an element in $X_{I_{\Lambda}}(id)(k) \sqcup X_{I_{\Lambda}}(w_{\Lambda})(k)$ (here we use Lemma \ref{lemma4.7}).
 	
 	Indeed, if 
 	\begin{equation*}
 	0 \overset{l}{\subset} \pi B^{\vee} \overset{1}{\subset} A \overset{h-1}{\subset} \pi A^{\vee} \overset{1}{\subset} B \overset{l}{\subset} \Lambda,
 	\end{equation*}
 	
 	then $(A/\pi \Lambda^{\vee} \subset B/ \pi \Lambda^{\vee}) \in X_{I_{\Lambda}}(id)(k)$.
 	
 	And if
 	\begin{equation*}
 	\pi B^{\vee} \subset A \nsubseteq \pi A^{\vee} \subset B,
 	\end{equation*}
 	
 	then $(A/\pi \Lambda^{\vee} \subset B/ \pi \Lambda^{\vee}) \in X_{I_{\Lambda}}(w_{\Lambda})(k)$.
 	
 	The proofs of (2), (3), (4) are similar.
 \end{proof}
 
\begin{definition}
For $i=0,1$, let $\Lambda \in \CL_i^+$. If $1 \leq h \leq n-1$, then we define a $\BF_{q^2}$-scheme
 \begin{equation*}
Y_{\Lambda}:=X_{I_{\Lambda}}(id) \sqcup X_{I_{\Lambda}}(w_{\Lambda})=\overline{X_{I_{\Lambda}}(w_{\Lambda})}.
 \end{equation*}
 The second equality is from the property of the Bruhat order (see \cite[Lemma 3.7]{HP}).
 If $h=0$ and $\Lambda \in \CL_0^+$, then we define $Y_{\Lambda}:=X_{I_{\Lambda}}(id)$. Similarly, if $h=n$ and $\Lambda \in \CL_1^+$, then we define $Y_{\Lambda}:=X_{I_{\Lambda}}(id)$.
  By abuse of notation, we denote by $Y_{\Lambda}$ its base change $Y_{\Lambda} \otimes \BF$.
\end{definition}

 By Proposition \ref{proposition4.4} and Proposition \ref{proposition4.5}, we have the following proposition.
 
 \begin{proposition}\label{proposition4.9} For $\Lambda \in \CL_i^+$ ($i=0,1$), $Y_{\Lambda}$ is irreducible, and
 	
 	(1) if $\Lambda \in \CL_0^+$, the dimension of $Y_{\Lambda}$ is 
 	\begin{equation*}
 	\dfrac{t(\Lambda)-1-h}{2}+h,
 	\end{equation*}
 	
 	(2) if $\Lambda \in \CL_1^+$, the dimension of $Y_{\Lambda}$ is 
 	\begin{equation*}
 	\dfrac{t(\Lambda)-1-(n-h)}{2}+n-h.
 	\end{equation*}
 \end{proposition}
\bigskip

\subsection{Description of the points of $\CN_{\Lambda}$}\label{subsec:section3.5}

In this subsection, we will use the theory of $O_F$-windows in \cite{ACZ}, \cite{Ahs} to obtain a description of $\CN_{\Lambda}(k)$ for an arbitrary field extension $k$ of $\BF$ (For a perfect field $k$, we can use the relative Dieudonne theory as in Section \ref{subsec:section2.2}, \ref{subsec:section2.3}). This will be used to prove the Theorem \ref{theorem 4.11}. For simplicity we denote by $O$ the ring of integers $O_F$. Let $k$ be an arbitrary field extension of $\BF$, and let $W_{O}(k)$ be the ring of ramified Witt vectors. Let $W_{O,k}=(W_{O}(k), I_{O}(k), k, ^{\sigma},^{\CV^{-1}})$ and $W_{O,\BF}=(W_{O}(\BF), \pi W_{O}(\BF), \BF, ^{\sigma},^{\CV^{-1}})$ be Witt $O$-frames.

Let $(\BM ,\CF,\CV)$ be the relative Dieudonne module of $\BX$ defined in Section \ref{subsec:section2.2}. then $(\BM,\CV \BM, \CF, \CV^{-1})$ is the $W_{O,\BF}$-window of $\BX$. The inclusion $W_O(\BF) \hookrightarrow W_O(k)$ induces a morphism of $O$-frames $W_{O,\BF} \rightarrow W_{O,k}$. Then by base change, we get the $W_{O,k}$-window ($\BM_k,\BM_k',\CF_k,\CV_k^{-1})$ of $\BX \otimes k$. More precisely,

$\bullet$ $\BM_k=W_O(k) \otimes_{W_{O}(\BF)} \BM$.

$\bullet$ $\BM_k'= \Ker (w_0 \otimes \pr)$, where $w_0$ is 0-th Witt polynomial, and $\pr:\BM \rightarrow \BM/\CV \BM$.

$\bullet$ $\CF_k=^{\sigma} \otimes \CF$.

$\bullet$ $\CV_k^{-1}$ is the unique $^{\sigma}$-linear morphism which satisfies
\begin{equation*}
\CV_k^{-1}(w\otimes y)=^{\sigma}w \otimes \CV^{-1}y,
\end{equation*}
\begin{equation*}
\CV_k^{-1}(^{\CV}w\otimes y)=w \otimes \CF y,
\end{equation*}
for all $w \in W_O(k)$, $x \in \BM$, and $y\in \CV \BM$.

Let $N_k=\BM_k \otimes_{W_O(k)} \Frac(W_O(k))$. The $O_E$-action on $\BM$ induces the $O_E$-action on $N_k$.

The polarization $\lambda \otimes k$ on $\BX \otimes_{\BF} k$ induces a nondegenerate $\Frac(W_O(k))$-bilinear alternating form $\langle \cdot, \cdot \rangle$ on $N_k$
\begin{equation*}
\langle \cdot, \cdot \rangle : N_k \times N_k \rightarrow \Frac(W_O(k)),
\end{equation*}
such that for all $x,y \in N_k$ and $a\in E$, it satisfies
\begin{equation*}
\langle F_kx,F_ky \rangle=\pi \langle x, y \rangle^\sigma,
\end{equation*}
\begin{equation*}
\langle ax,y \rangle=\langle x, a^{*}y\rangle.
\end{equation*}

The $O_E$-action on $N_k$ induces $\BZ/2\BZ$-grading
\begin{equation*}
N_k=N_{k,0} \oplus N_{k,1}.
\end{equation*}

Each $N_{k,i}$ is totally isotropic with respect to $\langle \cdot, \cdot \rangle$ and $F_k$ is homogeneous of degree 1 with respect to the decomposition. For a $W_O(k)$-lattice $M=M_0 \oplus M_1 \subset N_k$, we define the dual lattice $M^{\perp}=M_1^{\perp} \oplus M_0^{\perp}$ as 
\begin{equation*}
M_i^{\perp}=\lbrace x \in N_{k,i+1} \vert \langle x, M_i \rangle \in W_O(k) \rbrace, i=0,1.
\end{equation*}

Let $(\Lambda_k^{\pm},\CV\Lambda_k^{\pm},\CF_k,\CV_k^{-1})$ be the $W_{O,k}$-windows of $X_{\Lambda^{\pm}}\otimes k$. Then by the theory of $O$-windows, we have the following proposition.

\begin{proposition}\label{proposition4.10}
	There is a bijection between the set $\CN_{\Lambda}(k)$ and the set of $W_O(k)$-lattices $M=M_0 \oplus M_1$ in $N_k$ such that
	
	(1) $M$ is $F_k$ and $O_E$-invariant.
	
	(2) $M_0 \overset{h}{\subset} M_1^{\perp} \overset{n-h}{\subset} \pi^{-1}M_0$, $M_1 \overset{h}{\subset} M_0^{\perp} \overset{n-h}{\subset} \pi^{-1}M_1.$ 
	
	(3) $\pi M_0 \overset{n-1}{\subset} M'_0 \overset{1}{\subset} M_0$, $\pi M_1 \overset{1}{\subset} M'_1 \overset{n-1}{\subset} M_1$, where $M'=M'_0 \oplus M'_1=\Ker (M \rightarrow \Lambda^+_k/\CV\Lambda^+_k)$.
	
	(4) $\Lambda_k^- \subset M \subset \Lambda_k^+.$
\end{proposition}

\begin{proof}
	The first condition is obvious. The condition (2) is from the condition on polarization: $\Ker\lambda \subset \BX[\pi]$ and the order of $\Ker\lambda$ is $q^{2h}$. The condition (3) is the determinant condition. The last condition is from the definition of $\CN_{\Lambda}$.
\end{proof}

\bigskip

\subsection{The isomorphism between $\CN_{\Lambda}$ and $Y_{\Lambda}$}\label{subsec:section3.6}
Let $\Lambda \in \CL^+_i$. In this subsection, we will prove that $\CN_{\Lambda}$ and $Y_{\Lambda}$ are isomorphic. Let $S$ be a $\BF$-scheme, and let $X$ be a strict formal $O_F$-module over $S$. We denote by $D(X)$ the Lie algebra of the universal extension of $X$ in the sense of \cite{ACZ}. Recall that $X \mapsto D(X)$ is the functor from the category of $\pi$-divisible formal $O$-module over $S$ to the category of locally free $O_S$-modules. This is compatible with base change.

Now, we will define a morphism $f: \CN_{\Lambda} \rightarrow Y_{\Lambda}$.
Let $R$ be a $\BF$-algebra, and $(X,i_X,\lambda_X,\rho_X)\in \CN_{\Lambda}(R)$.
By definition of $\CN_{\Lambda}$, we have two isogenies
\begin{equation*}
X_{\Lambda^-,R} \xrightarrow{\rho_{X,\Lambda^-}} X_R \xrightarrow{\rho_{X,\Lambda^+}} X_{\Lambda^+,R}
\end{equation*}

Let $\BB_{\Lambda}=\Lambda^+/\Lambda^-$, $E(X):=\Ker{(D(\rho_{X,\Lambda^-}))}$.
Then by \cite[Corollary 4.7]{VW}, $E(X)$ is a direct summand of the $R$-module $\BB_{\Lambda} \otimes_{\BF}R$. By the $O_E$-action on $\BB_{\Lambda}$ and on $E(X)$, we have the following decompositions
\begin{equation*}
\BB_{\Lambda}=\BB_{\Lambda,0} \oplus \BB_{\Lambda,1},
\end{equation*}
\begin{equation*}
E(X)=E_0(X) \oplus E_1(X).
\end{equation*}
We write $\langle \cdot, \cdot \rangle'$ for the alternating form $\pi^{-i+1}\langle \cdot, \cdot \rangle$ on $\BB_{\Lambda}$.

\begin{remark}\label{remark4.10}
	Let $R=k$ be an algebraically closed field. If $\Lambda \in \CL_0^+$, then $E_0(X)=A/\pi \Lambda^{\vee}$ and $E_1(X)^{\perp'}=B/\pi \Lambda^{\vee}$ ($^{\perp'}$ means the dual with respect to $\langle \cdot, \cdot \rangle'$) with the notation use in the proof of Proposition \ref{proposition4.8}. Therefore, $E_0(X) \subset E_1(X)^{\perp'}$. Similarly, if $\Lambda \in \CL_1^+$, then $E_0(X)=A/\pi^2\Lambda^{\vee}$ and $E_1(X)^{\perp'}=\pi B/\pi^2 \Lambda^{\vee}$. Therefore, we have $E_1(X)^{\perp'} \subset E_0(X).$
\end{remark}

From the remark, we obtain a map $f:\CN_{\Lambda}(R) \rightarrow Y_{\Lambda}(R)$ by sending $(X,i_X,\lambda_X,\rho_X)$ to $(E_0(X) \subset E_1(X)^{\perp'})$ where $\Lambda \in \CL_0^+$, and to $(E_1(X)^{\perp'} \subset E_0(X))$ where $\Lambda \in \CL_1^+$ (note that both $E_0(X), E_1(X)^{\perp'}$ are subspaces of $\BB_{\Lambda,0}=V_{\Lambda}$ in Section \ref{subsec:section3.4}). Since this map commutes with base change, it gives the desired morphism $f:\CN_{\Lambda} \rightarrow Y_{\Lambda}$.

\begin{theorem} \label{theorem 4.11}
	The morphism $f$ is an isomorphism.
\end{theorem}
\begin{proof}
	The proof is the same as the proof of [VW] Theorem 4.8.
	Indeed, $f$ gives a bijection on $k$-valued points, where $k$ is algebraically closed field by Lemma \ref{lemma4.2}, Proposition \ref{proposition4.8}. Therefore, $f$ is universally bijective. Since $\CN_{\Lambda}$ is proper (by Lemma \ref{lemma4.1}) and $Y_\Lambda$ is separated, we have that $f$ is proper. Therefore, $f$ is a universal homeomorphism.
	Now, for an arbitrary field extension $k$ of $\BF$, we can work systematically using Proposition \ref{proposition4.10} to show that $f$ is a bijection on $k$-valued points, and hence $f$ is birational.
	Therefore $f$ is proper, finite, birational morphism, and $Y_{\Lambda}$ is normal (See \cite[Fact 2.1]{Gor1}). Now, by Zariski's main theorem, $f$ is an isomorphism.
\end{proof}

\subsection{Regularity of $\CN$}\label{subsec:section3.7}

In this subsection, we will prove that $\CN_{ O_{\breve{E}}}:=\CN^h_{E/F}(1,n-1)_{ O_{\breve{E}}}$ is regular, where $E=\BQ_{p^2}$. Therefore, in this subsection, $\pi=p$, $F=\BQ_p$, $E=\BQ_{p^2}$, but, we will use the general notation. See Proposition \ref{proposition5.12} for the general case. First, note that $\CN_{ O_{\breve{E}}}=\CN^0_{E/F}(1,n-1)_{ O_{\breve{E}}}$ is formally smooth over $\Spf O_{\breve{E}}$ (see \cite{VW}). This shows that $\CN^n_{E/F}(1,n-1)_{ O_{\breve{E}}}$ is formally smooth over $\Spf O_{\breve{E}}$, since $\CN^0 \simeq \CN^n$ (see Remark \ref{remark5.2}). Therefore, we can assume that $ 1 \leq h \leq n-1$. When $h=1$, the regularity of $\CN_{ O_{\breve{E}}}$ is proved in \cite[Theorem 5.1]{RSZ1}. We can use the same method to prove the regularity of $\CN_{O_{\breve{E}}}$, where $h\geq 2$. 
To prove this, we need the local model for $\CN$ as in \cite[Definition 3.27]{RZ} and \cite{PRS}. We will follow the definition in \cite{RSZ1}. Let $l(\cdot,\cdot)$ be a $E/F$-hermitian form on $E^n$ given by the matrix 	
\begin{displaymath}
\left(\begin{array}{cc} 
\pi I_h &\\
& I_{n-h}
\end{array}
\right)
\end{displaymath}

Fix an element $\delta \in O_E^{\times}$ such that $\delta^{*}=-\delta$. Let $(\cdot,\cdot)$ be the $F$-bilinear alternating form on $E^n$ defined by
\begin{equation*}
(x,y)=\frac{1}{2} Tr_{E/F}(\delta l(x,y)),\quad x,y\in E^n.
\end{equation*}
Let $\Lambda_0:=O_E^n$ and $\Lambda_1 :=\pi^{-1}O_E^h \oplus O_E^{n-h}$. Then $\Lambda_0$ is the dual lattice of $\Lambda_1$ with respect to $(\cdot,\cdot)$.
The local model $\CN^{loc}$ is the scheme over $O_E$ representing the functor which sends each $O_E$-scheme $S$ to the set of pairs $(\CM_0,\CM_1)$ satisfying the following conditions:

$\bullet$ For each $i=0,1$, $\CM_i$ is an $O_E \otimes_{O_F}O_S$-subsheaf of $\Lambda_i \otimes_{O_F}O_S$ which Zariski locally on $S$ is an $O_S$-direct summand of rank $n$;

$\bullet$ The natural maps $\Lambda_0 \otimes_{O_F}O_S \rightarrow \Lambda_1 \otimes_{O_F}O_S$ and $\Lambda_1 \otimes_{O_F}O_S \xrightarrow{(\pi^h, 1^{n-h})} \Lambda_0 \otimes_{O_F}O_S$ carry $\CM_0$ into $\CM_1$ and $\CM_1$ into $\CM_0$, respectively;

$\bullet$ $\CM_0^{\perp}=\CM_1$ with respect to the natural perfect pairing $(\Lambda_0 \otimes_{O_F}O_S) \times (\Lambda_1 \otimes_{O_F}O_S) \rightarrow O_S$ induced by $(\cdot,\cdot)$;

$\bullet$ It satisfies the determinant condition of signature $(n-1,1)$ 
\begin{equation*}
\Charpol(a\otimes 1 \vert \CM_i)=(T-a)^{n-1}(T-a^{*}) \in O_S[T]
\end{equation*}
for all $a\in O_E$, $i=0,1$.

As in \cite{RSZ1}, the base change $(\CN^{loc})_{O_{\breve{E}}}$ is the local model for $\CN_{O_{\breve{E}}}$. Therefore, we can use this to prove the following local property of $\CN_{ O_{\breve{E}}}$.
\begin{proposition}\label{proposition4.13}
	If $1 \leq h \leq n-1$, then the formal scheme $\CN_{ O_{\breve{E}}}$ has semistable reduction. In particular, $\CN_{ O_{\breve{E}}}$ is regular.
\end{proposition}
\begin{proof}
By \cite[Proposition 3.33]{RZ}, it suffices to show that local model $\CN^{loc}$ has semistable reduction. Let $S$ be a $O_E$-scheme. Consider the decomposition
	\begin{equation*}
	O_E \otimes_{O_F} O_S \overset{~}{\rightarrow} O_S \times O_S
	\end{equation*}
	\begin{equation*}
	a \otimes b \longmapsto (ab, a^*b).
	\end{equation*}

For any $(\CM_0,\CM_1) \in \CN^{loc}(S)$, the above decomposition induces decompositions
	\begin{equation*}
	\CM_i=\CM_i' \oplus \CM_i'' \subset \Lambda_i \otimes_{O_F}O_S=(\Lambda_i \otimes_{O_F}O_S)'\oplus(\Lambda_i \otimes_{O_F}O_S)'', i=0,1.
	\end{equation*}
By the determinant condition, $\CM_i' \subset (\Lambda_i \otimes_{O_F}O_S)'$ is $O_S$-locally direct summand of rank $n-1$. Since $\CM_0=\CM_1^{\perp}$, we have that $\CM_0'$ and $\CM_1'$ determine $\CM_1''$ and $\CM_0''$, respectively. Therefore, the map $(\CM_0,\CM_1) \mapsto (\CM_0',\CM_1')$ is an isomorphism from $\CN^{loc}$ to the standard local model over $O_E$ in \cite{Gor2} for the group $GL_n$, the cocharacter $\mu=(1^{(n-1)},0)$, and the periodic lattice chain determined by $(\Lambda_0 \otimes_{O_F}O_E)' \subset (\Lambda_1 \otimes_{O_F}O_E)'.$ By \cite[4.4.5]{Gor2} (in case $k=h, r=n-1$ or $k=h, r=1$, since two cases are isomorphic by Lemma 4.8 in loc. cit), this standard local model has semistable reduction.
\end{proof}

\bigskip

\subsection{The global structure of $\CN$: the Bruhat-Tits stratification}\label{subsec:subsection3.8}

In this section, we will study the global structure of $\CN=\CN^h_{E/F}(1,n-1)$. Let $\CN_{red}$ be the underlying reduced subscheme of $\CN$. We define

\begin{displaymath} 
t_{\text{max}}=\left\{
\begin{array}{cc} 
n \quad&\text{if } (n-h) \text{ is odd};\\
n-1  \quad&\text{if } (n-h) \text{ is even},
\end{array}
\right.
\end{displaymath}

\begin{displaymath} 
t_{\text{min}}=\left\{
\begin{array}{cc} 
0 \quad&\text{if } h \text{ is odd};\\
1 \quad&\text{if } h \text{ is even}.
\end{array}
\right.
\end{displaymath}

Let $\CA$ be the set of lattices in $\CL_0$ of type $t_{\text{min}}$, and $\CB$ the set of lattices in $\CL_0$ of type $t_{\text{max}}$. By Remark \ref{remark3.14}, we have a bijective map from $\CL_0^+ \sqcup \CL_0^-$ to $\CL_0^+ \sqcup \CL_1^+$. This map sends an element $\Lambda \in \CA$, to $\pi \Lambda^{\vee}$ which is an element of $\CL_1^+$ of type $n-t_{\text{min}}$. We have the following theorem.

\begin{theorem}\label{theorem3.14}
	The map sending $\Lambda \in \CA$ to $\CN_{\pi \Lambda^{\vee}}$ and $\Lambda \in \CB$ to $\CN_{\Lambda}$ is a bijective map from $\CA \cup \CB$ to the set of irreducible components of $\CN_{red}$. For $\Lambda \in \CA$, $\CN_{ \pi \Lambda^{\vee}}$ is an irreducible component of dimension 
	\begin{equation*} \dfrac{h-1-t_{\text{min}}}{2}+(n-h).
	\end{equation*}
	For $\Lambda \in \CB$, $\CN_{\Lambda}$ is an irreducible component of dimension
	\begin{equation*}
	\dfrac{t_{\text{max}}-1-h}{2}+h.
	\end{equation*}
\end{theorem}
\begin{proof}
	This is clear from Proposition \ref{proposition3.10}, Proposition \ref{proposition3.11}, Lemma \ref{lemma4.2}, Proposition \ref{proposition4.9}.
\end{proof}

Let $\tilde{J}=SU(N_0,\lbrace \cdot, \cdot \rbrace)$ (recall that $N=N_0 \oplus N_1$ is the rational relative Dieudonne module of $\BX$ and $\lbrace \cdot, \cdot \rbrace$ is a form defined in Section \ref{subsec:section2.3}). This is an algebraic group over $F$ . We denote by $\CB(\tilde{J},F)$ the abstract simplicial complex of the Bruhat-Tits building of $\tilde{J}$. By \cite[Theorem 3.6]{Vol} and \cite[Section 4.1]{VW}, we can identify $\CL_0$ with the set of vertices of $\CB(\tilde{J},F)$. 
Proposition \ref{proposition3.10}, Proposition \ref{proposition3.11}, Lemma \ref{lemma4.2} show that the intersection behavior of $\CN_{\Lambda}$ ($\Lambda \in \CL_0^+)$, $\CN_{ \pi \Lambda^{\vee}}$ ($\Lambda \in \CL_0^-)$ is closely related to the Bruhat-Tits building structure of $\CB(\tilde{J},F)$. For example, let
\begin{equation*}
\Lambda_{\text{min}} \overset{1}{\subset} \dots \overset{1}{\subset} \Lambda \overset{1}{\subset} \Lambda' \overset{1}{\subset} \dots \overset{1}{\subset} \Lambda_{\text{max}},
\end{equation*}
be a chain in $\CL_0$, where $\Lambda_{\text{min}}, \Lambda, \Lambda', \Lambda_{\text{max}}$ are of type $t_{\text{min}}, h-1, h+1, t_{\text{max}}$, respectively. Then we have
\begin{equation*}
\CN_{\pi \Lambda^{\vee}} \subset \dots \subset \CN_{\pi \Lambda_{\text{min}}^{\vee}},
\end{equation*}
\begin{equation*}
\CN_{\Lambda'} \subset \dots \subset \CN_{\Lambda_{\text{max}}}.
\end{equation*}

By the above Theorem \ref{theorem3.14}, $\CN_{\pi \Lambda_{\text{min}}^{\vee}}, \CN_{\Lambda_{\text{max}}}$ are irreducible components of $\CN_{red}$.
For an algebraically closed field $k$ containing $\BF$, we have
\begin{equation*}
\CN_{ \pi \Lambda^{\vee}}(k) \cap \CN_{\Lambda'}(k)=\lbrace(\pi\Lambda_k^{\vee}, \Lambda'_k)\rbrace \neq \emptyset.
\end{equation*}

Also, we have the following proposition.

\begin{proposition} \label{proposition3.15}
	Let $\Lambda_0, \Lambda_0' \in \CL_0^+$, $\Lambda_1, \Lambda_1' \in \CL_1^+$.
	\begin{enumerate}
		\item The following assertions are equivalent.
		\begin{enumerate}
			\item $\CN_{\Lambda_0} \cap \CN_{\Lambda_0'} \neq \emptyset$.
			\item $\Lambda_0 \cap \Lambda_0' \in \CL_0^+$. 
		\end{enumerate}
		In this case, we have
		\begin{equation*}
		\CN_{\Lambda_0} \cap \CN_{\Lambda_0'}=\CN_{\Lambda_0 \cap \Lambda_0'}.
		\end{equation*}
		\item The following assertions are equivalent.
		\begin{enumerate}
			\item $\CN_{\Lambda_1} \cap \CN_{\Lambda_1'} \neq \emptyset$.
			\item $\Lambda_1 \cap \Lambda_1' \in \CL_1^+$.
		\end{enumerate}
		In this case, we have
		\begin{equation*}
		\CN_{\Lambda_1} \cap \CN_{\Lambda_1'}=\CN_{\Lambda_1 \cap \Lambda_1'}.
		\end{equation*}
		\item The following assertions are equivalent.
		\begin{enumerate}
			\item $\CN_{\Lambda_0} \cap \CN_{\Lambda_1} \neq \emptyset$.
			\item $\pi \Lambda_1^{\vee} \subset \Lambda_0$.
		\end{enumerate}
		\item For an algebraically closed field $k$ containing $\BF$, we have
		\begin{equation*}
		\CN(k)= \bigcup_{\Lambda \in \CL_0^+ \cup \CL_1^+} \CN_{\Lambda}(k).
		\end{equation*}
	\end{enumerate}
\end{proposition}
\begin{proof}
	(1), (2), (3) are clear from Proposition \ref{proposition3.10}, Proposition \ref{proposition3.11}, Proposition \ref{proposition3.13}. (4) is clear from Proposition \ref{proposition3.7}, Lemma \ref{lemma4.2}.
\end{proof}

For $i=0,1$ and $\Lambda \in \CL_i^+$, we define a set
\begin{equation*}
\CL^+_{\Lambda}:=\lbrace \Lambda' \in \CL_i^+ \vert \Lambda' \subsetneq \Lambda \rbrace,
\end{equation*}
and let
\begin{equation*}
\CN_{\Lambda}^0:=\CN_{\Lambda} \backslash \bigcup_{\Lambda' \in \CL^+_{\Lambda}} \CN_{\Lambda'}.
\end{equation*}

We have the following analogue of \cite[Proposition 5.3]{VW}.
\begin{proposition}\label{proposition3.16}
The subset $\CN^0_{\Lambda}$ is open and dense in $\CN_{\Lambda}$.
\end{proposition}
\begin{proof}
	The proof is the same as the proof of \cite[Proposition 5.3]{VW}.
\end{proof}

By definition, we have a disjoint union of locally closed subschemes
\begin{equation*}
\CN_{\Lambda}=\CN_{\Lambda}^0 \sqcup \bigsqcup_{\Lambda' \in \CL^+_{\Lambda}} \CN^0_{\Lambda'}.
\end{equation*}
This gives a locally finite stratification $(\CN^0_{\Lambda})_{\Lambda \in \CL_i^+,i=0,1}$ of $\CN$.

\begin{definition}\label{definition3.17}
	The stratification $(\CN^0_{\Lambda})_{\Lambda \in \CL_i^+,i=0,1}$ of $\CN$ is called the \textsl{Bruhat-Tits stratification}. The closed subschemes $\CN_{\Lambda}$ are called the \textsl{closed Bruhat-Tits strata}.
\end{definition}

\bigskip
\subsection{The moduli space $\CN_{E/K}^h(r,n-r)$}\label{subsec:section3.9}

Let $K$ be a finite extension of $\BQ_p$ contained in $F$, with ring of integers $O_K$, and residue field $\BF_s$. We fix a uniformizer $\omega$. In this subsection, we will define the moduli space $\CN_{E/K}^h(r,n-r)$. For this, we imitate the construction in \cite{Mih}. We will use the notation in Section \ref{sec:section2}. Also, we will use the theory of $O$-display in \cite{ACZ}. 

Let $F^u$ (resp. $E^u$) be the maximal unramified extension of $K$ in $F$ (resp. $E$). Let $[F:K]=ef$, where $f=[F^u:K]$ is the inertia degree, and $e=[F:F^u]$ is the ramification index. We denote by $\breve{K}$ the completion of a maximal unramified extension of $K$, and $^F:\breve{K} \rightarrow \breve{K}$ the Frobenius automorphism. We choose a decomposition $\Psi:=\Hom_K(E^u,\breve{K})=\Psi_0 \sqcup \Psi_1$ such that $(\Psi_0)^{*}=\Psi_1$, where $^{*}$ is the nontrivial Galois automorphism of $E$ over $F$. We fix an element $\psi_0 \in \Psi_0$, and $\breve{E}:=E \otimes_{E^u,\psi_0} \breve{K}$.

\begin{definition}(\cite[Definition 2.7]{Mih})
	For $a \in E$, we define the following polynomials,
	\begin{equation*}
		P^{E/K}_{(0,1)}(a;t)=\prod_{\psi \in \Psi_1}\psi(\Charpol_{E/E^u}(a;t)) \in E^u[t];
	\end{equation*}
		\begin{equation*}
	P^{E/K}_{(1,0)}(a;t)= P^{E/K}_{(0,1)}(a;t)(t-a)(t-a^{*})^{-1} \in E[t];
	\end{equation*}
		\begin{equation*}
	P^{E/K}_{(r,n-r)}(a;t)=(P^{E/K}_{(1,0)}(a;t))^r(P^{E/K}_{(0,1)}(a;t))^{n-r} \in E[t].
	\end{equation*}
\end{definition}
\bigskip

\begin{definition} (cf. \cite[Definition 3.1]{Mih})
	Let $S$ be a scheme over $\Spf{O_E}$. A \textit{(supersingular) hermitian} $O_E \text{-} O_K \text{-} h$\textit{-module} over $S$ is a triple $(X,i_X,\lambda_X)$, where $X/S$ is a supersingular strict formal $O_K$-module, $i_X$ is an $O_E$-action on X, and $\lambda_X:X \rightarrow X^{\vee}$ is a polarization such that its Rosati involution induces the involution $^{*}$ on $O_E$. Also, $\Ker{\lambda_X} \subset X[\pi]$ and the order of $\Ker{\lambda_X}$ is $s^{2fh}=q^{2h}$.
	
	An isomorphism (resp. quasi-isogeny) of two hermitian $O_E \text{-} O_K \text{-} h$ modules $(X,i_X,\lambda_X)$ and $(Y, i_{Y},\lambda_Y)$ is an $O_E$-linear isomorphism (resp. quasi-isogeny) $\alpha :X \rightarrow Y$ of the underlying strict formal $O_K$-modules and $\alpha^{\vee} \circ \lambda_Y \circ \alpha$ differs locally on $S$ from $\lambda_X$ by a scalar in $O_K^{\times}$.

	We say that a hermitian $O_E \text{-} O_K \text{-} h$-module $(X,i_X,\lambda_X)$ is of \textit{rank} $n$ if the $K$-height of $X$ is $n[E:K]$.\\
\end{definition}	
	
	Let $X$ be a hermitian $O_E \text{-} O_K \text{-} h$\textit{-module} over a $\Spf{O_E}$-scheme $S$. Then by $O_E$-action, we have the grading
	\begin{equation*}
	\Lie(X)=\bigoplus_{\psi \in \Psi} \Lie_{\psi}(X).
	\end{equation*}
	Here $\Lie_{\psi}(X)$ is the direct summand on which $O_{E^u}$ acts via $\psi$.
	We define the following determinant condition.
	\begin{definition}(cf. \cite[Definition 2.8]{Mih})
	Let $S$ be a scheme over $\Spf{O_E}$. A hermitian $O_E \text{-} O_K \text{-} h$-module $(X,i_X,\lambda_X)$ of rank $n$ over $S$ is of \textit{signature} $(r,n-r)$ if for all $a \in O_E$,
	\begin{equation}\label{eq5.01}
	\Charpol(i_X(a) \vert \Lie X)= P^{E/K}_{(r,n-r)}(a;t),
	\end{equation}
	\begin{equation}\label{eq5.02}
	(i_X(a)-a) \vert_{\Lie_{\psi_0}(X)}=0.
	\end{equation}
	
	Here, we view $P^{E/K}_{(r,n-r)}(a;t)$ as an element of $O_S[t]$ via the structure morphism. The second equation means that $O_E$ acts on $\Lie_{\psi_0}(X)$ via the structure morphism. Note that \eqref{eq5.01} implies \eqref{eq5.02} if $E$ is unramified over $\BQ_p$.
	\end{definition}

\bigskip

Let $(\BX,i_{\BX},\lambda_{\BX})$ be a hermitian $O_E \text{-} O_K \text{-} h$-module of signature $(r,n-r)$ over $\BF_{q^2}$. Let $\CN_{E/K}^h(r,n-r)$ be the set-valued functor on (Nilp) which sends a scheme $S \in$(Nilp) to the set of isomorphism classes of tuples $(X,i_X,\lambda_X,\rho_X)$. Here $(X,i_X,\lambda_X)$ is a hermitian $O_E \text{-} O_K \text{-} h$-module of signature $(r,n-r)$ over $S$ and $\rho_X$ is a $O_E$-linear quasi-isogeny
\begin{equation*}
\rho_X: X\times_S \overline{S} \rightarrow \BX \times_{\BF^{q^2}}\overline{S}
\end{equation*}
of height 0.

Furthermore, we require that the following diagram commutes up to a constant in $O_K^{\times}$,
\begin{center}
	\begin{tikzcd}
		X_{\overline{S}} \arrow{r}{\lambda_{X_{\overline{s}}}}
		\arrow{d}{\rho_X}
		&X^{\vee}_{\overline{S}} \\
		\mathbb{X}_{\overline{S}}
		\arrow{r}{\lambda_{\BX_{\overline{S}}}}  &\mathbb{X}_{\overline{S}}^{\vee}
		\arrow{u}{\rho_X^{\vee}}.
	\end{tikzcd}
\end{center}

Two quadruples $(X,i_X,\lambda_X,\rho_X)$ and $(Y,i_Y,\lambda_Y,\rho_Y)$ are isomorphic if there exists an $O_E$-linear isomorphism $\alpha:X \rightarrow Y$ with $\rho_Y \circ (\alpha \times_S \overline{S})=\rho_X$ and $\alpha^{\vee} \circ \lambda_Y \circ \alpha$ differs locally on $S$ from $\lambda_X$ by a scalar in $O_K^{\times}$.

The functor $\CN_{E/K}^h(r,n-r) \otimes O_{\breve{E}}$ is representable by a formal scheme which is locally formally of finite type over $\Spf{O_{\breve{E}}}$ (See \cite{Mih}).

\begin{remark}\label{remark5.4}
Let us fix a hermitian $O_E \text{-} \BZ_p \text{-} h$-module $(\BX,i_{\BX},\lambda_{\BX})$ of signature $(r,n-r)$ over $\BF_{q^2}$ such that its rational Dieudonne module $(N,\CF)$ generated by elements $\eta \in N$ satisfying $\CF^{2f}\eta=p^f \eta$, where $f$ is a inertia degree of $F/\BQ_p$. Such a triple exists by \cite[Lemma 2.10]{Mih} with slight modification of the polarization and the base field. This is decent in the sense of \cite[Definition 2.13]{RZ}, and hence we can use \cite[Theorem 2.16]{RZ}. Therefore, if we fix such a triple, then the functor $\CN_{E/\BQ_p}^h(r,n-r)$ is representable by a formal scheme which is locally formally of finite type over $\Spf O_E$.
\end{remark}

\begin{remark}\label{remark5.6}
	One can see that there is a unique hermitian $O_E \text{-} \BZ_p \text{-} h$-module $(\BX,i_{\BX},\lambda_{\BX})$ of signature $(r,n-r)$ over $k$ up to quasi-isogeny, where $k$ is an algebraic closure of $\BF_{q^2}$. This can be proved by using \cite[Proposition 2.5]{Mih}, \cite[Lemma 2.10]{Mih} with slight modification of the polarization.
\end{remark}

\begin{remark} \label{remark5.5}
	The definition of $\CN_{E/F}^h(r,n-r)$ in Section \ref{sec:section2} coincides with the definition in this section.
\end{remark}

\begin{definition}(cf. \cite[Definition 4.2]{Mih})
	We denote by $O_E \text{-} O_K \text{-} h$-Herm the stack of hermitian $O_E \text{-} O_K \text{-} h$-modules $(X,i_X,\lambda_X)$ over $\Sch / \Spf O_E$ such that locally for Zariski topology, it is of signature $(r,n-r)$ for some $r$. The morphisms in this category are the $O_E$-linear morphisms of $p$-divisible groups.
\end{definition}

Now, let $S=\Spec{R}$ be an affine scheme over $\Spf{O_E}$ and $(X,i_X,\lambda_X)$ be an hermitian $O_E \text{-} O_K \text{-} h$-module of signature $(r,n-r)$ over $S$. Let $(P,Q,F,F_1)$ be the $O_K$-display (i.e., $O_K$-window over $W_{O_K,R}$) of $(X,i_X,\lambda_X)$. We denote by $\langle \cdot, \cdot \rangle : P \times P \rightarrow W_{O_K}(R)$ the $W_{O_K}(R)$-bilinear alternating form induced by $\lambda_X$. From the $O_E$-action, we have the decomposition
\begin{equation*}
O_E \otimes_{O_K} W_{O_K}(R) \simeq \prod_{\psi \in \Psi}O_E \otimes_{O_{E^u}} W_{O_K}(R).
\end{equation*}
This decomposition gives gradings
\begin{equation*}
\begin{split}
P=\prod_{\psi \in \Psi}P_{\psi}=\prod_{\psi \in \Psi_0}P_{\psi} \oplus P_{\psi^*},\\
Q=\prod_{\psi \in \Psi}Q_{\psi}=\prod_{\psi \in \Psi_0}Q_{\psi} \oplus Q_{\psi^*}.
\end{split}
\end{equation*}

Let $(P^{\vee},Q^{\vee},F^{\vee},F_1^{\vee})$ be the dual $O_K$-window of $(P,Q,F,F_1)$ (see \cite[Section 11]{Mih}), and consider its gradings
\begin{equation*}
P^{\vee}=\prod_{\psi \in \Psi}P^{\vee}_{\psi}=\prod_{\psi \in \Psi_0}P^{\vee}_{\psi} \oplus P^{\vee}_{\psi^*},
\end{equation*}
\begin{equation*}
Q^{\vee}=\prod_{\psi \in \Psi}Q^{\vee}_{\psi}=\prod_{\psi \in \Psi_0}Q^{\vee}_{\psi} \oplus Q^{\vee}_{\psi^*}.
\end{equation*}

Let $P_{\psi,\BQ}:=P_{\psi} \otimes \BQ$, and let $\langle \cdot, \cdot \rangle_{\BQ}=\langle \cdot, \cdot \rangle \otimes \BQ$. Note that our pairing satisfies
\begin{equation*}
\langle \cdot, \cdot \rangle_{\BQ}|_{P_{\psi,\BQ} \times P_{\psi',\BQ}}\equiv 0 \qquad \text{if} \quad \psi' \neq \psi^*.
\end{equation*}

Therefore, we have
\begin{equation*}
P_{\psi}^{\vee}=\lbrace x \in P_{\psi,\BQ} | \langle x, P_{\psi^*} \rangle_{\BQ} \subset W_{O_K}(R) \rbrace.
\end{equation*}

Also, we have the following lemma.

\begin{lemma}\label{lemma5.7}
	The order of $\Ker{\lambda_X}$ is $q^{2h}=s^{2fh}$ if and only if $P_{\psi} \overset{h}{\subset} P_{\psi}^{\vee}$, $\forall \psi \in \Psi$ .
\end{lemma}
\begin{proof}
	Let $P_{\psi}=L_{\psi} \oplus T_{\psi}$, $Q_{\psi}=L_{\psi}+I_{O_K}(R)T_{\psi}$ be a normal decomposition. By the signature condition, we have
	\begin{equation*}
	\begin{split}
	L_{\psi}=P_{\psi}, T_{\psi}=0, \text{\quad if $\psi \in \Psi_0\backslash \lbrace \psi_0\rbrace$},\\
	L_{\psi}=0, T_{\psi}=P_{\psi}, \text{\quad if $\psi \in \Psi_1\backslash \lbrace \psi_0^*\rbrace$}.
	\end{split}
	\end{equation*}
	From the normal decomposition, we get a $^F$-linear isomorphism
\begin{displaymath}
	\begin{array}{ccc}
		\Phi_{\psi}:P_{\psi}=L_{\psi}\oplus T_{\psi} & \rightarrow
		&P_{^F\psi} \\
		\quad\qquad\qquad (l,t)& \mapsto & (F_1(l)+F(t))
	\end{array}
\end{displaymath}
	By our special signature condition, we have
\begin{displaymath}
\begin{array}{cccc}
\Phi_{\psi}:P_{\psi} & \rightarrow
&P_{^F\psi}& \\
\qquad x& \mapsto  & F_1(x),& \text{        if        } \psi \in \Psi_0 \backslash \lbrace \phi_0 \rbrace
\end{array}
\end{displaymath}
\begin{displaymath}
\begin{array}{cccc}
\Phi_{\psi}:P_{\psi} & \rightarrow
&P_{^F\psi}& \\
\qquad x& \mapsto  & F(x),& \text{        if        } \psi \in \Psi_1 \backslash \lbrace \phi_0^* \rbrace
\end{array}
\end{displaymath}

We claim that if $P_{^F\psi_0} \overset{k}{\subset} P_{^F\psi_0}^{\vee}$ for some $k$, then for all $\psi \in \Psi$ we have 
\begin{equation*}
P_{\psi} \overset{k}{\subset} P_{\psi}^{\vee}.
\end{equation*}

First, note that $\Phi_{\psi}$ is a $^F$-linear isomorphism, hence
\begin{equation*}
\begin{split}
\Phi(P_{^{F^i}\psi_0})=P_{^{F^{i+1}}\psi_0},\\
\Phi(P_{^{F^i}\psi_0^*})=P_{^{F^{i+1}}\psi_0^*}.
\end{split}
\end{equation*}

We will show that $\Phi(P^{\vee}_{^{F^i}\psi_0})=P^{\vee}_{^{F^{i+1}}\psi_0}$ for $1 \leq i \leq f-1$.
Note that
\begin{equation*}
\begin{split}
x \in P^{\vee}_{^{F^{i+1}}\psi_0} \Leftrightarrow \langle x, P_{^{F^{i+1}}\psi_0^*}\rangle \subset W_{O_K}(R)\\
	\Leftrightarrow \langle x, \Phi(P_{^{F^{i}}\psi_0^*}) \rangle \subset W_{O_K}(R).
\end{split}
\end{equation*}
First, assume that $^{F^i}\psi_0^* \in \Psi_0 \backslash \lbrace \psi_0 \rbrace$, then $\Phi=F_1$ on $P_{^{F^{i}}\psi_0^*}$. Therefore,

\begin{equation*}
\begin{split}
\langle x, \Phi(P_{^{F^{i}}\psi_0^*})\rangle \subset W_{O_K}(R)\\
\Leftrightarrow \langle x, F_1(P_{^{F^{i}}\psi_0^*}) \rangle \subset W_{O_K}(R).\\
\Leftrightarrow \langle \Phi(\Phi^{-1}(x)), F_1(P_{^{F^{i}}\psi_0^*}) \rangle \subset W_{O_K}(R).
\end{split}
\end{equation*}
Since $x \in P^{\vee}_{^{F^{i+1}}\psi_0}$ and $^{F^i}\psi_0 \in \Psi_1 \backslash \lbrace \psi_0^* \rbrace$, we have

\begin{equation*}
\begin{split}
\langle \Phi(\Phi^{-1}(x)), F_1(P_{^{F^{i}}\psi_0^*}) \rangle \subset W_{O_K}(R)\\
\Leftrightarrow \langle F(\Phi^{-1}(x)), F_1(P_{^{F^{i}}\psi_0^*}) \rangle \subset W_{O_K}(R)\\
\Leftrightarrow ^F\langle \Phi^{-1}(x), P_{^{F^{i}}\psi_0^*} \rangle \subset W_{O_K}(R)\\
\Leftrightarrow \Phi^{-1}(x) \in P^{\vee}_{^{F^{i}}\psi_0}.
\end{split}
\end{equation*} Here, we used the fact that $\langle F \cdot, F_1 \cdot \rangle=
\langle F_1 \cdot, F \cdot \rangle=^F\langle  \cdot, \cdot \rangle$.

In the case that $^{F^i}\psi_0^* \in \Psi_1 \backslash \lbrace \psi_0^* \rbrace$, we can prove the claim in the same way.

Therefore, $\Phi(P^{\vee}_{^{F^i}\psi_0})=P^{\vee}_{^{F^{i+1}}\psi_0}$ for $1 \leq i \leq f-1$.

Now, assume that $P_{^F\psi_0} \overset{k}{\subset} P_{^F\psi_0}^{\vee}$, then we can show inductively that
\begin{equation*}
P_{^{F^{i+1}}\psi_0}=\Phi(P_{^{F^i}\psi_0})\overset{k}{\subset} \Phi(P^{\vee}_{^{F^i}\psi_0})=P^{\vee}_{^{F^{i+1}}\psi_0}, \quad \forall 1 \leq i \leq f-1.
\end{equation*}
Since $P_{\psi} \overset{k}{\subset} P^{\vee}_{\psi}$ if and only if $P_{\psi^*} \overset{k}{\subset} P^{\vee}_{\psi^*}$, we can conclude that the claim holds.

By this claim, we have
\begin{equation*}
\begin{split}
|\Ker \lambda_X| =s^{2fh}
\Leftrightarrow P \overset{2fh}{\subset} P^{\vee}
\Leftrightarrow P_{\psi} \overset{h}{\subset} P^{\vee}_{\psi}, \quad \forall \psi \in \Psi.
\end{split}
\end{equation*}

\end{proof}

With this lemma, we can follow the whole steps in \cite[Chatper 4]{Mih}. Indeed, the only difference is the polarization, hence with the above lemma, one can show the following analogue of \cite[Proposition 4.4]{Mih}. Let $Sch / \Spf O_E$ (resp. $Sch'/ \Spf O_E$) be the category of schemes (resp. locally noetherian schemes) over $\Spf O_E$ together with the Zariski topology.

\begin{proposition}\label{proposition5.8}(cf. \cite[Proposition 4.4]{Mih})
	There is an isomorphism of stacks over $\Sch/\Spf O_E$
	\begin{equation*}
	\CC_{K,F^u}: O_E \text{-} O_K \text{-} h \text{-} \text{Herm} \overset{\simeq}{\rightarrow} O_E \text{-} O_{F^u} \text{-} h\text{-} \text{Herm}
	\end{equation*}
	that is equivariant for the Rosati involutions and sends objects of signature $(r,n-r)$ to objects of signature $(r,n-r)$.
\end{proposition}
\begin{proof}
	One can follow the proof of \cite[Proposition 4.4]{Mih} with Lemma \ref{lemma5.7}. Also see \cite[Remark 4.5]{Mih}.
\end{proof}

In addition, we can show the following analogue of \cite[proposition 4.6]{Mih}.

\begin{proposition}\label{proposition5.9}(cf. \cite[Proposition 4.6]{Mih})
	There is an isomorphism of stacks over $\Sch' /\Spf O_{\breve{E}}$
	\begin{equation*}
	\CC_{F^u,F}: (O_E \text{-} O_{F^u} \text{-} h \text{-} \text{Herm})_{O_{\breve{E}}} \overset{\simeq}{\rightarrow} (O_E \text{-} O_F \text{-} h\text{-} \text{Herm})_{O_{\breve{E}}}
	\end{equation*}
		that is equivariant for the Rosati involutions and sends objects of signature $(r,n-r)$ to objects of signature $(r,n-r)$. Here, $(-)_{O_{\breve{E}}}$ means the base change to $O_{\breve{E}}$.
\end{proposition}
\begin{proof}
	One can follow the proof of \cite[Proposition 4.6]{Mih} with Lemma \ref{lemma5.7}.
\end{proof}

The following proposition is an analogue of \cite[Theorem 4.1]{Mih}.

\begin{proposition}\label{proposition5.10}(cf. \cite[Theorem 4.1]{Mih})
	For any intermediate field $\BQ_p \subset K \subset F$, we have an isomorphism
	\begin{equation*}
	c_{K,F}:(\CN_{E/F}^h(r,n-r))_{O_{\breve{E}}} \simeq (\CN_{E/K}^h(r,n-r))_{O_{\breve{E}}}.
	\end{equation*}
	Furthermore, if $F$ is unramified over $\BQ_p$, then
	\begin{equation*}
	c_{K,F}:\CN_{E/F}^h(r,n-r) \simeq \CN_{E/K}^h(r,n-r).
	\end{equation*}
\end{proposition}
\begin{proof}
	This follows from the above two propositions, and by fixing framing objects. See the proof of \cite[Theorem 4.1]{Mih}.
\end{proof}

\begin{remark}\label{remark5.11}
	Let $F$ be an unramified extension of $\BQ_{p}$. Let $(\BX,i_{\BX},\lambda_{\BX})$ be a hermitian $O_E \text{-} \BZ_p \text{-} h$-module in Remark \ref{remark5.4} and consider a hermitian $O_E \text{-} O_F \text{-} h$-module $\CC_{\BQ_p,F}((\BX,i_{\BX},\lambda_{\BX}))$ by using Proposition \ref{proposition5.8}. By Remark \ref{remark5.4}, we have that $\CN_{E/\BQ_p}^h(r,n-r)$ is representable by a formal scheme over $\Spf O_E$ which is locally formally of finite type, with the framing object $(\BX,i_{\BX},\lambda_{\BX})$. Therefore, by Proposition \ref{proposition5.10}, $\CN_{E/F}^h(r,n-r)$ is representable by a formal scheme over $\Spf O_E$ which is locally formally of finite type with the framing object $\CC_{\BQ_p,F}((\BX,i_{\BX},\lambda_{\BX}))$.
\end{remark}

\begin{remark}
	One can see that there is a unique hermitian $O_E \text{-} O_K \text{-} h$-module $(\BX,i_{\BX},\lambda_{\BX})$ of signature $(r,n-r)$ over $k$ up to quasi-isogeny, where $k$ is an algebraic closure of $\BF_{q^2}$. This can be proved by using Remark \ref{remark5.6}, Proposition \ref{proposition5.8}, Proposition \ref{proposition5.9}.
\end{remark}

\begin{proposition}\label{proposition5.12}
	If $h=0, n$, the formal scheme $\CN_{E/F}^h(1,n-1)_{O_{\breve{E}}}$ is formally smooth over $\Spf O_{\breve{E}}$. If $1 \leq h \leq n-1$, then $\CN_{E/F}^h(1,n-1)_{O_{\breve{E}}}$ has semistable reduction. In particular, it is regular, for all $h$.
\end{proposition}
\begin{proof}
	When $h=0$, it is proved in \cite[Proposition 2.14]{Mih}. Since $\CN_{E/F}^0(1,n-1)_{O_{\breve{E}}}$ and $\CN_{E/F}^n(1,n-1)_{O_{\breve{E}}}$ are isomorphic (see Remark \ref{remark5.2}), $\CN_{E/F}^n(1,n-1)_{O_{\breve{E}}}$ is also formally smooth over $\Spf O_{\breve{E}}$. Now assume that $1 \leq h \leq n-1$. By Proposition \ref{proposition5.10}, it suffices to show that $\CN_{E/\BQ_p}^h(1,n-1)_{O_{\breve{E}}}$ has semistable reduction. Since this moduli problem is PEL-type, it suffices to show that its local model has semistable reduction (\cite[Proposition 3.33]{RZ}). To define the local model $\CN^{loc}$ in our case, we need to use the notation in Section \ref{subsec:section3.7} (here, we follow \cite[Appendix B]{RSZ2}). Let $l (\cdot, \cdot)$ be a $E/F$-hermitian form on $E^n$ given by the matrix
	\begin{displaymath}
	\left(\begin{array}{cc} 
	\pi I_h &\\
	& I_{n-h}
	\end{array}
	\right).
	\end{displaymath}
	
Fix an element $\delta \in O_E^{\times}$ such that $\delta^*=-\delta$. Let $\theta^{-1}_{F/\BQ_p}$ be a generator of the inverse different of $F/\BQ_p$. Let $(\cdot, \cdot)$ be the $\BQ_p$-bilinear alternating form,
\begin{equation*}
(x,y)=Tr_{E/\BQ_p}(\theta^{-1}_{F/\BQ_p} \delta l(x,y)), \quad x,y \in E^n.
\end{equation*}
Let $\Lambda_0=O_E^n$ and $\Lambda_1=\pi^{-1}O_E^h \oplus O_E^{n-h}$. Then the dual $\Lambda_1^{\vee}$ of the lattice $\Lambda_1$ with respect to $(\cdot, \cdot)$ is $\Lambda_0$. Now, let $\CL$ be the self-dual lattice chain 
\begin{equation*}
\lbrace \dots \subset \pi \Lambda_1 \subset \Lambda_0 \subset \Lambda_1= \Lambda_0^{\vee} \subset \pi^{-1} \Lambda_0 \subset \dots \rbrace
\end{equation*}

Then $\CN^{loc}$ is the functor which sends each $O_{\breve{E}}$-schemes $S$ to the set of isomorphism classes of families $(\Lambda \otimes_{\BZ_p} O_S \twoheadrightarrow \CP_{\Lambda})_{\Lambda \in \CL}$ such that

$\bullet$ For each $\Lambda$, $\CP_{\Lambda}$ is an $O_E \otimes_{\BZ_p} O_S$-linear quotient of $\Lambda \otimes_{\BZ_p} O_S$, locally free on $S$ as an $O_S$-module.

$\bullet$ For each inclusion $\Lambda \subset \Lambda'$ in $\CL$, the arrow $\Lambda \otimes_{\BZ_p} O_S \rightarrow \Lambda' \otimes_{\BZ_p} O_S$ induces an arrow $\CP_{\Lambda} \rightarrow \CP_{\Lambda'}$.

$\bullet$ For each $\Lambda$, the isomorphism $\Lambda \otimes_{\BZ_p} O_S \xrightarrow{\pi \otimes 1} (\pi \Lambda) \otimes_{\BZ_p} O_S$ identifies $\CP_{\Lambda} \overset{~}{\rightarrow} \CP_{\pi\Lambda}$.

$\bullet$ For each $\Lambda$, the perfect pairing $(\Lambda \otimes_{\BZ_p} O_S) \times (\Lambda^{\vee} \otimes_{\BZ_p} O_S) \xrightarrow{(\cdot, \cdot) \otimes O_S} O_S$ identifies $(\Ker(\Lambda \otimes_{\BZ_p} O_S \twoheadrightarrow \CP_{\Lambda}))^{\perp}$ with $\Ker(\Lambda^{\vee} \otimes_{\BZ_p} O_S \twoheadrightarrow P_{\Lambda^{\vee}}).$\\

We need to impose one more condition.

By the $O_E$-action on $S$, there is a natural identification
\begin{equation*}
O_{E^u} \otimes_{\BZ_p} O_S \xrightarrow{~} \prod_{\psi \in \Psi} O_S.
\end{equation*}

This induces a decomposition,
\begin{equation*}
\CP_{\Lambda} \xrightarrow{~} \bigoplus_{\psi \in \Psi} \CP_{\Lambda, \psi}.
\end{equation*}

$\bullet$ For each $\Lambda$, $\CP_{\Lambda}$ satisfies
\begin{equation}\label{equation5.1}
\Charpol_{O_S}(a\otimes 1 \vert \CP_{\Lambda})= P^{E/\BQ_p}_{(1,n-1)}(a;t),
\end{equation}

\begin{equation}\label{equation5.2}
(a \otimes 1 - 1 \otimes a) \vert P_{\Lambda, \psi_0}=0.
\end{equation}
Here, $P_{\Lambda, \psi_0}$ is the direct summand on which $O_{E^u}$ acts via $\psi_0$. These two conditions follow from the conditions \eqref{eq5.01} and \eqref{eq5.02}. \\

Now, fix a scheme $S$ over $O_{\breve{E}}$, and let $(\Lambda \otimes_{\BZ_p} O_S \twoheadrightarrow \CP_{\Lambda})_{\Lambda \in \CL} \in \CN^{loc}(S)$. 
By the signature condition \eqref{equation5.1}, we have
	\begin{displaymath}
\left\{
\begin{array}{l}
\CP_{\Lambda, \psi_0} \text{ is locally free of rank 1 over } O_S,\\
\CP_{\Lambda, \psi_0^*}=\CP_{\Lambda, \psi_0}^{\perp} \subset (\Lambda \otimes_{\BZ_p} O_S)_{\psi_0^*},\\
\CP_{\Lambda, \psi}=0 \qquad\text{ if } \psi \in \Psi_0 \backslash \lbrace \psi_0 \rbrace,\\
\CP_{\Lambda, \psi^*}=(\Lambda \otimes_{\BZ_p} O_S)_{\psi^*} \qquad \text{if } \psi \in \Psi_1 \backslash \lbrace \psi_0^* \rbrace.\\
\end{array}
\right.
\end{displaymath}
Therefore, $(\Lambda \otimes_{\BZ_p} O_S \twoheadrightarrow \CP_{\Lambda})_{\Lambda \in \CL}$ is determined by $(\CP_{\Lambda,\psi_0})_{\Lambda \in \CL}$.

Also, by the condition \eqref{equation5.2}, $O_E$ acts on $\CP_{\Lambda, \psi_0}$ via the structure morphism, therefore $\CP_{\Lambda, \psi_0}$ is a quotient of 
\begin{equation*}
A_{\Lambda} :=(\Lambda \otimes_{\BZ_p} O_S) \otimes_{O_E \otimes_{\BZ_p} O_S} O_S,
\end{equation*}
which is locally free of rank $n$ over $O_S$.

It follows that the map $(\Lambda \otimes_{\BZ_p} O_S \twoheadrightarrow \CP_{\Lambda})_{\Lambda \in \CL} \mapsto  (A_{\Lambda} \twoheadrightarrow \CP_{\Lambda, \psi_0})_{\Lambda \in \CL}$ is an isomorphism from $\CL^{loc}$ to the standard local model over $\Spec O_{\breve{E}}$ in Proposition \ref{proposition4.13} (i.e. the standard local model with the group $GL_n$, the cocharacter $\mu=(1^{(n-1)},0)$, and the lattice chain $\CL$). Therefore, by \cite[4.4.5]{Gor2} (in case $k=h, r=1$) again, this local model has semistable reduction.
\end{proof}

\bigskip

\section{Uniformization of unitary Shimura varieties}\label{sec:section4}

In this section, we will define a Shimura variety and study its basic locus. This Shimura variety is studied in \cite{RSZ2}. In this section, we use the notation $\BA$ for the adele rings and $\BA_f$ for the ring of finite adeles and $\BA_f^p$ for the finite adeles away from the prime $p$.

Let $F$ be a CM field over $\BQ$ and $F^+$ be its totally real subfield of index $2$. We fix a presentation $F=F^+(\sqrt{\Delta})$. Denote by $d$ the dimension of $F^+$ over $\BQ$. We denote by $a \mapsto \bar{a}$ the nontrivial automorphism of $F/F^+$. Denote by $\Phi_{F^+}$ (resp. $\Phi_F$) the set of real (resp. complex) embeddings of $F^+$ (resp. $F$). We define $\Phi$ as the CM type of $F$ determined by $\sqrt{\Delta}$, i.e.,
\begin{equation*}
\Phi:=\lbrace \phi \in \Phi_F \text{ } \vert \text{ } \phi(\sqrt{\Delta}) \in \BR_{>0}\sqrt{-1} \rbrace.
\end{equation*} 
 
 We have a natural projection $\pi:\Phi_F \rightarrow \Phi_{F^+}$. For every $\tau \in \Phi_{F^+}$, denote by $\tau^-$ (resp. $\tau^+$) the unique element in $\Phi$ (resp. $\Phi_F \backslash \Phi$) whose image under $\pi$ is $\tau$. We fix a distinguished element $\tau_1 \in \Phi_{F^+}$ (resp. $\tau_1^- \in \Phi$).\\

\subsection{The Shimura data}\label{subsec:section4.1}
We first define the Shimura data $(G,\lbrace h_G \rbrace)$ as follows. Let $V$ be a $F/F^+$-hermitian vector space of dimension $n$ with the hermitian form
\begin{equation*}
( \cdot, \cdot)_V :V\times V \rightarrow F,
\end{equation*}
that is $F$-linear in the first variable. Let $U(V)$ be the unitary group of $V$. This is a reductive group over $F^+$ such that for every $F^+$-algebra $R$,
\begin{equation*}
U(V)(R)=\lbrace g \in \Aut_R(V \otimes_{F^+} R) \vert (gv,gw)_V=(v,w)_V, \quad\forall v,w \in V\otimes_{F^+} R \rbrace.
\end{equation*}

We assume that for $\tau_1$, the signature of $V\otimes_{F^+,\tau_1} \BR$ is $(1,n-1)$ and for $\tau \in \Phi_{F^+}\backslash \lbrace \tau_1 \rbrace$, the signature of $V \otimes_{F^+,\tau} \BR$ is $(0,n)$.

Let $G := \Res_{F^+/\BQ} U(V)$. We define the Hodge map
\begin{equation*}
h_G:\Res_{\BC/\BR} \BG_{m,\BC} \rightarrow G_{\BR}
\end{equation*}
by the map sending $z \in \BC^{\times}=\Res_{\BC/\BR}\BG_{m,\BC}(\BR)$ to
\begin{displaymath}
\left( \left( \begin{array}{cc}
z/\bar{z} &  \\
 & I_{n-1}
\end{array} \right), \left( \begin{array}{c}
I_n
\end{array} \right) ,\cdots, \left( \begin{array}{c}
I_n
\end{array} \right)\right),
\end{displaymath}
where we identify $G_{\BR}(\BR)$ as a subgroup of $GL_n(\BC)^d$ via $\lbrace \tau_1^-, \cdots,\tau_d^- \rbrace = \Phi$.
Then we have a Shimura data $(G, \lbrace h_G \rbrace)$.

Now, we will define the second Shimura data $(Z,\lbrace h_Z \rbrace)$. Let $Z$ be the torus
\begin{equation*}
Z:=\lbrace z \in \Res_{F^+/\BQ}\BG_m \vert \Nm_{F/F^+}(z) \in \BG_m\rbrace.
\end{equation*}

We define the Hodge map
\begin{equation*}
h_Z:\Res_{\BC/\BR} \BG_{m,\BC} \rightarrow Z_{\BR}
\end{equation*}
by the map sending $z \in \BC^{\times}=\Res_{\BC/\BR}\BG_{m,\BC}(\BR)$ to
\begin{displaymath}
\left( \left( \begin{array}{c}
\bar{z}
\end{array} \right) ,\cdots, \left( \begin{array}{c}
\bar{z} 
\end{array} \right); z\bar{z}\right),
\end{displaymath}
where we identify $Z_{\BR}(\BR)$ as a subgroup of $GL_1(\BC)^d \times \BC^{\times}$ via $\lbrace \tau_1^-, \cdots,\tau_d^- \rbrace.$

Then we have the second Shimura data $(Z,h_Z)$.

Now, we consider the reductive group $\tilde{G}=G\times Z $ over $\BQ$. We define its Hodge map
\begin{equation*}
h_{\tilde{G}}:\Res_{\BC/\BR} \BG_{m,\BC} \xrightarrow{(h_G,h_Z)} \tilde{G}_{\BR}.
\end{equation*}
Then $(\tilde{G}, \lbrace h_{\tilde{G}} \rbrace)$ is the product Shimura data, which is defined in \cite{RSZ2} (with the same notation). Denote by $E$ its reflex field. This is the fixed field of the following subgroup
\begin{equation*}
\Aut(\BC/E):=\lbrace \sigma \in \Aut(\BC) \vert \sigma \circ \Phi =\Phi \text{ and } \sigma \tau_1^-=\tau_1^- \rbrace.
\end{equation*}

This Shimura variety has a moduli interpretation over $\Spec E$. We recall this moduli problem from \cite[Section 3.2]{RSZ2}. First, we need to define an auxiliary moduli problem $\CM_0^{\Fa}$ over $O_E$, where $\Fa$ is a fixed nonzero ideal of $O_{F^+}$. We denote by $M^{\Fa}_0$ its generic fiber. For a locally noetherian $O_E$-scheme, we define $\CM_0^{\Fa}(S)$ to be the groupoid of triples $(A_0,i_0,\lambda_0)$, such that

$\bullet$ $A_0$ is an abelian scheme over $S$ with an $O_F$-action $i_0:O_F \rightarrow \End(A_0)$, which satisfies the Kottwitz condition of signature $((0,1)_{\tau\in \Phi_{F^+}})$, i.e.,
\begin{equation*}
\Charpol(i(a) \vert \Lie(A_0))=\prod_{\tau \in \Phi_{F^+}}(T-\tau^+(a)) \text{, for all }a\in O_F.
\end{equation*}

$\bullet$ $\lambda_0$ is a polarization of $A_0$ such that $\Ker \lambda_0 =A_0[\Fa]$. Also, $\lambda_0$'s Rosati involution induces on $O_F$, via $i_0$, the nontrivial Galois automorphism of $F/{F^+}$.

A morphism between two objects $(A_0,i_0,\lambda_0)$ and $(A_0',i_0',\lambda_0')$ is an $O_F$-linear isomorphism $\mu_0:A_0 \rightarrow A_0'$ under which $\lambda_0'$ pulls back to $\lambda_0$.

This $\CM_0^{\Fa}$ is a Deligne-Mumford stack, finite and \'etale over $\Spec O_E$. Also, we can choose an ideal $\Fa$ such that $\CM_0^{\Fa}$ is nonempty (\cite[Remark 3.3]{RSZ2}).

Let $K_Z \subset Z(\BA_{f})$ be the unique maximal compact subgroup $Z(\hat{\BZ})$.

If $F^+=\BQ$, then $\CM_0^{\Fa} \otimes \BC$ is isomorphic to the Shimura variety $Sh_{K_Z}(Z,h_Z)$. In general, $\CM_0^{\Fa} \otimes \BC$ is copies of $Sh_{K_Z}(Z,h_Z)$ and each copy corresponds to a similarity class of a certain 1-dimensional hermitian space. More precisely, we define $\CR^{\Fa}_{\Phi}(F)$ as the set of isomorphism classes of pairs $(W,\langle \cdot, \cdot \rangle)$ where $W$ is a 1-dimensional $F$-vector spaces and $\langle \cdot, \cdot \rangle$ is a nondegenerate alternating form $\langle \cdot, \cdot \rangle :W \times W \rightarrow \BQ$ such that

$\bullet$ $\langle ax, y \rangle= \langle x, \bar{a}y \rangle$  for all  $x,y \in W$, $a \in F$;

$\bullet$ $x \rightarrow \langle \sqrt{\Delta}x,x \rangle$ is a negative definite quadratic form on $W$;

$\bullet$ $W$ contains an $O_F$-lattice $\Lambda$ whose dual $\Lambda^{\perp}$ with respect to $\langle \cdot, \cdot \rangle$ is $\Fa^{-1}\Lambda$.

We denote by $\CR^{\Fa}_{\Phi}(F)/_{\sim}$ the set of similarity classes of elements of $\CR^{\Fa}_{\Phi}(F)$ by a factor in $\BQ^{\times}$.

Then, we have a disjoint union decomposition
\begin{equation*}
\CM_0^{\Fa} \simeq \bigsqcup_{W \in \CR^{\Fa}_{\Phi}(F)/_{\sim}} \CM_0^{\Fa,W},
\end{equation*}
and each $\CM_0^{\Fa,W} \otimes \BC$ is isomorphic to the Shimura variety $Sh_{K_Z}(Z,h_Z)$. We denote by $M_0^{\Fa,W}$ the generic fiber of $\CM_0^{\Fa,W}$.

From now on, we fix an element $W \in \CR^{\Fa}_{\Phi}/_{\sim}$

Now, we consider open compact subgroups $K_{\tilde{G}} \subset \tilde{G}(\BA_{f})$ of the form
\begin{equation*}
K_{\tilde{G}}=K_G \times K_Z \subset G(\BA_{{F^+},f})\times Z(\BA_{f}),
\end{equation*}
where $K_G$ is an open compact subgroup of $G(\BA_{{F^+},f})$.

We now define a moduli functor $M_{K_{\tilde{G}}}(\tilde{G})$ on the category of locally noetherian schemes over $E$ as follows. For every such scheme $S$, let $M_{K_{\tilde{G}}}(\tilde{G})(S)$ be the groupoid of tuples $(A_0,i_0,\lambda_0,A,i,\lambda,\bar{\eta})$, where

$\bullet$ $(A_0,i_0,\lambda_0)$ is an object of $\CM_0^{\Fa,W}(S)$.

$\bullet$ $A$ is an abelian scheme over $S$ with an $F$-action $i:F \rightarrow \End(A)_{\BQ}$ satisfying the Kottwitz condition of signature $((1,n-1)_{\tau_1}, (0,n)_{\tau \in \Phi_{F^+} \backslash \lbrace \tau_1 \rbrace})$, i.e., for all $a \in F$,
\begin{equation*}
\Charpol(i(a) \vert \Lie(A))=(T-\tau_1^-(a))(T-\tau_1^+(a))^{n-1}\prod_{\tau \in \Phi_{F^+} \backslash \lbrace \tau_1 \rbrace}(T-\tau^+(a))^n.
\end{equation*}

$\bullet$ $\lambda$ is a polarization of $A$, whose Rosati involution induces on $F$, via $i$, the nontrivial Galois automorphism of $F/{F^+}$.

$\bullet$ $\bar{\eta}$ is a $K_{\tilde{G}}$-level structure. This is a $K_G$-orbit of $\BA_{F,f}$-linear isometries
\begin{equation*}
\eta:\Hom_F(\hat{V}(A_0), \hat{V}(A)) \simeq -V \otimes_F \BA_{F,f}.
\end{equation*}
Here, $-V$ is the same $E$-vector space as $V$, but its hermitian form multiplied by $-1$. We write $\hat{V}(A)$ for the full rational Tate module of $A$. Also, we considered $\Hom_F(\hat{V}(A_0), \hat{V}(A)) $ as a hermitian space with the hermitian form $h_A$,
\begin{equation*}
h_A(x,y)=\lambda_0^{-1} \circ y^{\vee} \circ \lambda \circ x \in \End_{\BA_{F,f}}(\hat{V}(A_0))=\BA_{F,f}.
\end{equation*}

A morphism between two objects
\begin{equation*}
(A_0,i_0,\lambda_0,A,i,\lambda,\bar{\eta}) \rightarrow (A_0',i_0',\lambda_0',A',i',\lambda',\bar{\eta}'),
\end{equation*} 
is given by an isomorphism $\mu_0:(A_0, i_0, \lambda_0) \simeq (A_0',i_0',\lambda_0')$ in $M_0^{\Fa,W}$ and an $F$-linear isogeny $\mu:A\rightarrow A'$ pulling $\lambda'$ back to $\lambda$ and $\bar{\eta}'$ back to $\bar{\eta}$.

Now, we can state the following proposition.
\begin{proposition}\label{proposition4.4.1}
	(\cite[Proposition 3.5]{RSZ2}) $M_{K_{\tilde{G}}}(\tilde{G})$ is a Deligne-Mumford stack smooth of relative dimension $n-1$ over $\Spec E$. The coarse moduli scheme of $M_{K_{\tilde{G}}}(\tilde{G})$ is a quasi-projective scheme over $\Spec E$, naturally isomorphic to the canonical model of $Sh_{K_{\tilde{G}}}(\tilde{G},\lbrace h_{\tilde{G}} \rbrace)$. For $K_{\tilde{G}}$ sufficiently small, the forgetful morphism $M_{K_{\tilde{G}}}(\tilde{G}) \rightarrow M_0^{\Fa,W}$ is relatively representable.
\end{proposition}
\bigskip
\subsection{Integral models}\label{subsec:section4.2}

	In this subsection, we will imitate the semi-global integral model in \cite[Section 4]{RSZ2}. Our case is related to AT parahoric level. We use the following notation. Fix a prime $p \neq 2$ and an embedding $\tilde{v}:\bar{\BQ} \rightarrow \bar{\BQ}_p.$ This embedding determines  places $u$ of $E$, $v_0$ of ${F^+}$, and $w_0$ of $F$ via $\tau_1^-$. Denote by $S_p$ the set of places $v$ of ${F^+}$ over $p$. Let $F_v:= F \otimes_{F^+} F^+_v$. Then, $F_v$ is a quadratic field extension of $F_v^+$ (resp. $F_v \simeq F^+_v \times F^+_v$), if $v$ is nonsplit (resp. split). Denote by $\pi_v$ a uniformizer in $F_v$ (when $v$ splits, this uniformizer is an ordered pair of uniformizers on the right side of the isomorphism $F_v \simeq F^+_v \times F^+_v$). Assume that $v_0$ is unramified over $p$ and inert in $F$. We assume that the ideal $\Fa$ in the definition of $\CM_0^{\Fa}$ is prime to $p$ and we fix an element $W \in \CR^{\Fa}_{\Phi}/_{\sim}$.
	
	Now, we choose lattices $\Lambda_{v} \subset V_v$ such that
	\begin{equation*}
	\Lambda_{v} \subset \Lambda_{v}^{\perp} \subset \pi_v^{-1}\Lambda_v,
	\end{equation*}
	where $\Lambda_v^{\perp}$ means the dual lattice of $\Lambda_v$ with respect to the hermitian form. Let $h$ be the index of $\Lambda_{v_0}$ in $\Lambda_{v_0}^{\perp}$, i.e., $[\Lambda_{v_0}^{\perp}:\Lambda_{v_0}]=h$.
	
	We take open compact subgroup $K_{\tilde{G}} \subset \tilde{G}(\BA_{f})$ as follows.
	\begin{equation*}
	K_{\tilde{G}}=K_G \times K_Z= K_G^p \times K_{G,p} \times K_Z,
	\end{equation*}
	where $K_G^p \subset G(\BA^p_{F^+,f})$ is arbitrary, and
	\begin{equation*}
	K_{G,p}:=\prod_{v \in S_p} K_{G,v} \subset \prod_{v \in S_p}G(F^+_v),
	\end{equation*}
	where $K_{G,v}$ is the stabilizer of $\Lambda_v$ in $G(F^+_v)$.
	
	Now, we can formulate a moduli problem over $\Spec O_{E,(u)}$ as follows.
	For a locally noetherian scheme $S$ over $\Spec O_{E,(u)}$, we associate the set of isomorphism classes of tuples $(A_0,i_0,\lambda_0,A,i,\lambda,\bar{\eta}^p)$, where
	
	$\bullet$ $(A_0,i_0,\lambda_0)$ is an object of $\CM_0^{\Fa,W}(S)$.
	
	$\bullet$ $A$ is an abelian scheme over $S$.
	
	$\bullet$ $i$ is an $O_F \otimes \BZ_{(p)}$-action satisfying the Kottwitz condition of signature $((1,n-1)_{\tau_1}, (0,n)_{\tau \in \Phi_{F^+} \backslash \lbrace \tau_1 \rbrace})$, i.e., for all $a \in F$,
	\begin{equation}\label{kottwitz}
	\Charpol(i(a) \vert \Lie(A))=(T-\tau_1^-(a))(T-\tau_1^+(a))^{n-1}\prod_{\tau \in \Phi_{F^+} \backslash \lbrace \tau_1 \rbrace}(T-\tau^+(a))^n.
	\end{equation}
	
	$\bullet$ $\lambda$ is a polarization of $A$, whose Rosati involution induces on $O_F \otimes \BZ_{(p)}$ the nontrivial Galois automorphism of $F/F^+$. Also, we impose the following condition. The action of $O_{F^+} \otimes \BZ_p \simeq \prod_{v \in S_p} O_{F^+,v}$ induces a decomposition of $p$-divisible group,
	\begin{equation*}
	A[p^{\infty}]=\prod_{v \in S_p}A[v^{\infty}].
	\end{equation*}
	Since Rosati involution of $\lambda$ fixes $O_{F^+}$, $\lambda$ induces a polarization $\lambda_v:A[v^{\infty}] \rightarrow A^{\vee}[v^{\infty}] \simeq A[v^{\infty}]^{\vee}$ for each $v$. We impose the condition that $\Ker \lambda_v$ is contained in $A[i(\pi_v)]$ of rank $\vert \Lambda_v^{\perp}/\Lambda_v \vert$ for each $v \in S_p$.
	
	$\bullet$ $\bar{\eta}^p$ is a $K_G^p$-orbit of $\BA^p_{F,f}$-linear isometries
\begin{equation*}
\eta:\Hom_F(\hat{V}^p(A_0), \hat{V}^p(A)) \simeq -V \otimes_F \BA_{F,f}^p.
\end{equation*}
Here, $-V$ is the same $E$-vector space as $V$, but its hermitian form multiplied by $-1$. We write $\hat{V}^p(A)$ for the rational prime-to-$p$ Tate module of $A$. Also, we considered $\Hom_F(\hat{V}^p(A_0), \hat{V}^p(A)) $ as a hermitian space with the hermitian form $h_A^p$,
\begin{equation*}
h_A^p(x,y)=\lambda_0^{-1} \circ y^{\vee} \circ \lambda \circ x \in \End_{\BA_{F,f}^p}(\hat{V}^p(A_0))=\BA_{F,f}^p.
\end{equation*}
	
For $v \neq v_0$, we impose the Eisenstein condition and the sign condition. Before we explain these conditions, we define a function $r:\Hom(F,\BC) \rightarrow \lbrace 0,1,n-1,n \rbrace$ such that,
\begin{displaymath} 
\tau \mapsto r_{\tau}:=\left\{
\begin{array}{cc}
1 & \tau=\tau_1^-;\\
0 & \tau \in \Phi \backslash \lbrace\tau_1^-\rbrace;\\
n-r_{\bar{\tau}} & \tau \notin \Phi.
\end{array}
\right.
\end{displaymath}
	
First, we recall the Eisenstein condition from \cite[Section 4.1]{RSZ2}. We impose the Eisenstein condition only when the base scheme $S$ has nonempty special fiber. In this case, we may base change via $\tilde{v}:O_{E,(u)} \rightarrow \bar{\BZ}_p$ (the ring of integers of $\bar{\BQ}_p$), and pass to completions and assume that $S$ is a scheme over $\Spf \bar{\BZ}_p$. We have a decomposition of the $p$-divisible group 
	\begin{equation*}
	A[p^{\infty}]=\prod_{w \vert p}A[w^{\infty}].
	\end{equation*}
where $w$ runs over the places of $F$ over $p$.	Since we assume that $p$ is locally nilpotent on $S$, there is a natural isomorphism
\begin{equation*}
\Lie A \simeq \Lie A[p^{\infty}] = \bigoplus_{w \vert p} A[w^{\infty}].
\end{equation*}
By using the embedding $\tilde{v}:\bar{\BQ} \rightarrow \bar{\BQ}_p$, we can identify
\begin{equation*}
\Hom_{\BQ}(F,\bar{\BQ}) \simeq \Hom_{\BQ}(F,\bar{\BQ}_p),
\end{equation*}
and this gives an identification
\begin{equation}\label{eq7.2.2}
\lbrace \tau \in \Hom_{\BQ}(F,\bar{\BQ}) \vert \tilde{v} \circ \tau =w \rbrace \simeq \Hom_{\BQ}(F_w,\bar{\BQ}_p).
\end{equation}
For each place $w$, by the Kottwitz condition \eqref{kottwitz}, the $p$-divisible group $A[w^{\infty}]$ is of height $n[F_w:\BQ_p]$ and dimension
\begin{equation*}
\dim A[w^{\infty}]=\sum_{\tau \in \Hom_{\BQ}(F_w,\bar{\BQ}_p)}r_{\tau}.
\end{equation*}

For each place $w$ such that $w \vert v$ and $v\neq v_0$, the action of $F$ on $A[w^{\infty}]$ is of a banal signature type in the sense of \cite[Appendix B]{RSZ2}. In other words, $r_{\tau}$ is $0$ or $n$ for all $\tau \in \Hom_{\BQ}(F_w,\bar{\BQ}_p)$. Let $\pi =\pi_w$ be a uniformizer in $F_w$ and let $F_w^u$ be the maximal unramified extension of $\BQ_p$ in $F_w$.
For each $\psi \in \Hom_{\BQ}(F_w^u, \bar{\BQ}_p)$, let
\begin{equation*}
A_{\psi}:=\lbrace \tau \in \Hom_{\BQ}(F_w,\bar{\BQ}_p) \vert \tau\vert_{F_w^u}=\psi\text{ and } r_{\tau}=n \rbrace.
\end{equation*}
Let
\begin{equation*}
Q_{A_{\psi}}:=\prod_{\tau \in A_{\psi}}(T-\tau(\pi)).
\end{equation*}
Then, the Eisenstein condition at $v (\neq v_0)$ is as follows. For each place $w$ that divides $v$, and for all $\psi \in \Hom_{\BQ}(F_w^u,\bar{\BQ}_p),$
\begin{equation*}
Q_{A_{\psi}}(i(\pi) \vert \Lie A[w^{\infty}])=0.
\end{equation*}

Now, we will define the sign condition at $v (\neq v_0)$. We impose this condition only when $v$ does not split in $F$. The sign condition at $v$ is the condition that for every point $s$ of $S$,
\begin{equation*}
\inv_v^r(A_{0,s},i_{0,s},\lambda_{0,s},A_s,i_s,\lambda_s)=\inv_v(-V_v).
\end{equation*}
We need to explain these two factors. For the left one, we refer to \cite[Appendix A]{RSZ2}. Also, we define
\begin{equation*}
\inv_v(-V_v):=(-1)^{n(n-1)/2} \det{(-V_v)} \in F_v^{+,\times}/\Nm F_v^{+,\times},
\end{equation*}
where $\det({-V_v}) \in F_v^{+,\times}/\Nm F_v^{+,\times}$ is the class of the determinant of any hermitian matrix of the hermitian space $-V_v$.

A morphism between two objects
\begin{equation*}
(A_0,i_0,\lambda_0,A,i,\lambda,\bar{\eta}^p) \rightarrow (A_0',i_0',\lambda_0',A',i',\lambda',\bar{\eta}'^p),
\end{equation*}
is given by an isomorphism $(A_0,i_0,\lambda_0) \simeq (A_0',i_0',\lambda_0')$ in $\CM_0^{\Fa,W}(S)$ and a quasi-isogeny $A \rightarrow A'$ which induces an isomorphism
\begin{equation*}
A[p^{\infty}] \simeq A'[p^{\infty}],
\end{equation*}
compatible with $i$ and $i'$, with $\lambda$ and $\lambda'$, and with $\bar{\eta}^p$ and $\bar{\eta}'^p$.

\begin{proposition}\label{proposition4.4.2}
	The moduli problem defined above is representable by a Deligne-Mumford stack $\CM_{K_{\tilde{G}}}(\tilde{G})$ flat over $\Spec O_{E,(u)}$. For $K^p_{G}$ small enough, $\CM_{K_{\tilde{G}}}(\tilde{G})$ is relatively representable over $\CM_0^{\Fa,W}$. The generic fiber $\CM_{K_{\tilde{G}}}(\tilde{G}) \times_{\Spec O_{E,(u)}} \Spec E$ is canonically isomorphic to $M_{K_{\tilde{G}}}(\tilde{G})$. Furthermore, if $h=0, n$, then $\CM_{K_{\tilde{G}}}(\tilde{G})$ is smooth over $\Spec O_{E,(u)}$. If $h \neq 0,n$, then $\CM_{K_{\tilde{G}}}(\tilde{G})$ has semistable reduction over $\Spec O_{E,(u)}$ provided that $E_{u}$ is unramified over $\BQ_p$.
\end{proposition}
\begin{proof} The representability and the statement for the generic fiber and the smoothness when $h=0$ (and hence when $h=n$) are proved in \cite[Theorem 4.1]{RSZ2}. Therefore, it suffice to show that this has semistable reduction over $\Spec O_{E,(u)}$ where $h \neq 0,n$ and $E_{u}$ is unramified over $\BQ_p$. To prove this we need to use the theory of the local model as in \cite[Theorem 4.10]{RSZ2}. The local model corresponding to $A_0$ is \'etale because $\CM_0^{\Fa,W}$ is. Let $M$ be the local model corresponding to $A$. Before we prove that $M$ has semistable reduction, we introduce some notation. By the identification \eqref{eq7.2.2}, we have
	\begin{equation}\label{eq7.2.3}
	\Hom_{\BQ}(F, \bar{\BQ}) \simeq \bigsqcup_{v \in S_p} \Hom_{\BQ_p}(F_v,\BQ_p).
	\end{equation}	
Let $r\vert_v : \Hom_{\BQ}(F_v,\bar{\BQ}_p) \rightarrow \lbrace 0,1,n-1,n \rbrace$ be the restriction of the function $r$ to $\Hom_{\BQ}(F_v,\bar{\BQ}_p)$. Let
	\begin{equation*}
	\sig_{r\vert_v}:=\sum_{\tau \in \Hom_{\BQ}(F_v,\bar{\BQ}_p)}r_{\tau} \tau,
	\end{equation*}
which is an element of $\BN[\Phi_F]$, the commutative monoid freely generated by $\Phi_F$. Note that the Galois group $\Gal(\BC/\BQ)$ acts on $\Phi_F$ hence on $\BN[\Phi_F]$. Let $E_{{r\vert v}}$ be the fixed field of the stabilizer in $\Gal(\BC/\BQ)$ of the element $\sig_{r\vert_v}$.
	
	Then we have a decomposition
	\begin{equation*}
	M=\prod_{v \in S_p}M_v \times_{\Spec O_{E_{r\vert v}}} \Spec O_{E_u},
	\end{equation*}
which is induced from \eqref{eq7.2.3}.

	For $v \neq v_0$, by our Kottwitz condition, $M_v$ is a banal local model as in \cite[Appendix B]{RSZ2}. Therefore, $M_v = \Spec O_{E_{r\vert v}}$. Also, $M_{v_0}$ is a local model which appears in the proof of Proposition \ref{proposition5.12} (here, we used the condition that $v_0$ is unramified, and therefore the condition \eqref{eq5.02} follows from the condition \eqref{eq5.01} which follows from the Kottwitz condition). Therefore, it has semistable reduction over ${\Spec O_{E_{r\vert v}}}$. Since $E_u$ is unramified over $\BQ_p$ (hence, over $E_{r\vert v}$) and semistable reduction is stable under an unramified base change, $M$ has semistable reduction over $\Spec O_{E_u}$,
\end{proof} 
\bigskip

\subsection{The uniformization theorem}\label{subsec:section4.3}

In this subsection, we will relate the basic locus of the special fiber of $\CM_{K_{\tilde{G}}}(\tilde{G})$ to the (relative) Rapoport-Zink space $\CN_{F_{w_0}/F^+_{v_0}}^h(1,n-1)$ in Section \ref{sec:section2}, via the non-archimedean uniformization theorem of Rapoport and Zink. We will follow the proof of \cite[Theorem 8.15]{RSZ2}. In order to simplify notation, we write $\CM$ for $\CM_{K_{\tilde{G}}}(\tilde{G})$, and $\CN$ for $\CN_{F_{w_0}/F^+_{v_0}}^h(1,n-1)$.
	
	Let $\breve{E}_u$ be the completion of a maximal unramified extension of $E_u$, and $k$ be the residue field of $O_{\breve{E}_u}$. Let $\CM_{O_{\breve{E}_u}} = \CM \otimes_{O_{E,(u)}} O_{\breve{E}_u}$. We denote by $\CM^{ss}$ the basic locus of $\CM \otimes_{O_{E,(u)}} k$ and by $\widehat{\CM^{ss}}$ the completion of $\CM_{O_{\breve{E}_u}}$ along $\CM^{ss}$.
	
	Choose a point $(\bm{A}_0,\bm{i}_0,\bm{\lambda}_0,\bm{A},\bm{i},\bm{\lambda},\bm{\bar{\eta}})$ of $\CM^{ss}(O_{\breve{E}_u})$. Let 
	\begin{equation*}
	\begin{array}{l}
	\BX_{0}=\bm{A}_0[p^{\infty}]=\prod_{v \in S_p} \bm{A}_0[v^{\infty}],\\
	\BX=\bm{A}[p^{\infty}]=\prod_{v \in S_p} \bm{A}[v^{\infty}],
	\end{array}
	\end{equation*}
	and $i_{\BX_{0}},\lambda_{\BX_{0}}, i_{\BX},\lambda_{\BX}$ be the induced $O_{F} \otimes \BZ_p$-actions and polarizations. This choice gives us the following non-archimedean uniformization morphism along the basic locus by \cite[Theorem 6.30]{RZ},
	\begin{equation*}
	\Theta:I(\BQ) \backslash \CN' \times \tilde{G}(\BA^p_f) /K^p_{\tilde{G}} \simeq \widehat{\CM^{ss}}.
	\end{equation*}
	Here the group $I$ is an inner form of $\tilde{G}$ associated to the hermitian space $V'$, where $V'$ is negative definite at all archimedean places and isomorphic to $V$ at all non-archimedean places except at $v_0$ (hence, by the product formula and the Hasse principle, $V'$ is determined), and $\CN'$ is the corresponding Rapoport-Zink space whose framing object is $(\BX_{0}, i_{\BX_{0}},\lambda_{\BX_{0}},\BX, i_{\BX},\lambda_{\BX})$. 

By \cite[Lemma 8.16]{RSZ2}, we have
\begin{equation*}
\CN' \simeq (Z(\BQ_p)/K_{Z,p}) \times (\CN_{F_{w_0}/\BQ_p}^h(1,n-1))_{O_{\breve{E}_u}} \times \prod_{v \in S_p \backslash \lbrace v_0 \rbrace} U(V)(F^+_v)/K_{G,v}.
\end{equation*}
Also, by Proposition \ref{proposition5.10}, $\CN_{O_{\breve{E}_u}} \simeq (\CN_{F_{w_0}/\BQ_p}^h(1,n-1))_{O_{\breve{E}_u}}$.

The following theorem summarizes the above discussion.

\begin{theorem}\label{theorem}
	There is a non-archimedean uniformization isomorphism
		\begin{equation*}
	\Theta: I(\BQ) \backslash  \CN' \times \tilde{G}(\BA^p_f) /K^p_{\tilde{G}} \simeq \widehat{\CM^{ss}},
	\end{equation*}
	where
	\begin{equation*}
	\CN' \simeq (Z(\BQ_p)/K_{Z,p}) \times \CN_{O_{\breve{E}_u}} \times \prod_{v \in S_p \backslash \lbrace v_0 \rbrace} U(V)(F^+_v)/K_{G,v}.
	\end{equation*}

\end{theorem}

\begin{proof}
	This is essentially the same as the proof of \cite[Theorem 6.30]{RZ}. For the convenience of the reader, we will construct the inverse morphism of $\Theta$. Let $S$ be a $O_{\breve{E}_u}$-scheme such that $p$ is locally nilpotent. Let $s$ be a geometric point of $S$. Choose a point $P=(A_0,i_0,\lambda_0,A,i,\lambda,\bar{\eta}) \in \CM^{ss}(S)$. By \cite[Proposition 6.29]{RZ}, we can choose $O_{F}$-linear quasi-isogenies
	\begin{equation*}
	\begin{array}{l}
	\tilde{\rho}_0 :A_0 \times_S S_k \rightarrow \bm{A}_{0k}  \times_k S_k,\\
	\tilde{\rho}:A \times_S S_k \rightarrow \bm{A}_k \times_k S_k,
	\end{array}
	\end{equation*}
	compatible with polarizations. Then, we have the induced quasi-isogenies
	\begin{equation*}
	\begin{array}{l}
	\rho_{0}: A_0[p^{\infty}]\times_S S_k \rightarrow \BX_{0k} \times_k S_k,\\
	\rho: A[p^{\infty}]\times_S S_k \rightarrow \BX_{k} \times_k S_k,
	\end{array}
	\end{equation*}
	The tuple $(A_0[p^{\infty}], A[p^{\infty}], \rho_{0}, \rho)$ (with the induced $O_{F} \otimes \BZ_p$-actions and the induced polarizations) gives an element in $\CN'(S)$ and this is the $\CN'$ part of $\Theta^{-1}(P)$.
	
	Now, we should find an element $(z, g) \in Z(\BA_f^p) \times G(\BA_f^p)=\tilde{G}(\BA_f^p)$ such that $\Theta^{-1}(P)=((A_0[p^{\infty}], A[p^{\infty}], \rho_{0}, \rho), (z, g))$.
	
	The element $z$ in $Z(\BA_f^p)$ comes from the moduli space $\CM_0^{\Fa,W}$. More precisely, by definition of $\CM_0^{\Fa,W}$, we have two $O_E \otimes \BA_f^p$-linear similitudes
	\begin{equation*}
	\begin{array}{l}
	\xi:\hat{V}^p(A_{0s}) \rightarrow W \otimes \BA^p_f,\\
	\zeta:\hat{V}^p(\bm{A}_{0k}) \rightarrow W \otimes \BA^p_f.
	\end{array}
	\end{equation*}
	Therefore, the composite
	\begin{equation*}
	W \otimes \BA^p_f \xrightarrow{\xi^{-1}} \hat{V}^p(A_{0s}) \xrightarrow{\rho_{0}} \hat{V}^p(\bm{A}_{0k}) \xrightarrow{\zeta} W \otimes \BA^p_f
	\end{equation*}
	gives an element $z$ in $Z(\BA_f^p)$.
	
	For the element $g$, consider the composite
	\begin{equation*}
	\begin{array}{ll}
	-V \otimes_F \BA_{F,f}^p \xrightarrow{\eta^{-1}} &\Hom_F(\hat{V}^p({A}_{0s}),\hat{V}^p({A}_s))\\ &\xrightarrow{(\rho_0^{-1},\rho)}\Hom_F(\hat{V}^p(\bm{A}_{0k}),\hat{V}^p(\bm{A}_k))  \xrightarrow{\bm{\eta}}-V \otimes_F \BA_{F,f}^p.
	\end{array}
	\end{equation*}
	This is an isometry which gives rise to an element $g$ in $G(\BA_f^p)$.
	
	The construction of $\Theta$ is identical to the arguments in \cite[Chapter 6]{RZ}.
\end{proof}

\bigskip

\section{Special cycles and arithmetic intersection numbers}\label{sec:section5}
In this section, we use the notation in Section \ref{sec:section2}. Also we denote by $k=\bar{\BF}_p$ and by $val$ the valuation of $E$. We will define the special cycles and study their intersections.

Let $(\overline{\BY},i_{\overline{\BY}},\lambda_{\overline{\BY}})$ be a strict formal $O_F$-module of $F$-height 2 over $k$, with an action $i_{\overline{\BY}}:O_E \rightarrow \End(\overline{\BY})$ and with principal polarization $\lambda_{\overline{\BY}}$. Also, we assume that it satisfies the determinant condition of signature $(0,1)$. Let $\CN^0(0,1)$ be the corresponding moduli space. To simplify notation, we write $\CN^0$ for $\CN^0(0,1)_{O_{\breve{E}}}$, $\CN$ for $\CN^h_{E/F}(1,n-1) _{O_{\breve{E}}}$ and $\widehat{\CN}$ for $\CN^{n-h}_{E/F}(1,n-1)_{O_{\breve{E}}}$.

\begin{definition}
	The \textit{space of special homomorphisms} is the $E$-vector space
	\begin{equation*}
	\BV:=\Hom_{O_E}(\overline{\BY},\BX)\otimes_{\BZ}\BQ.
	\end{equation*}
	For $x,y \in \BV$, we define a hermitian form $h$ on $\BV$ as
	\begin{equation*}
	h(x,y)=\lambda_{\overline{\BY}}^{-1} \circ y^{\vee}  \circ \lambda_{\BX} \circ x \in \End_{O_E}(\overline{\BY}) \otimes \BQ \overset{i_{\overline{\BY}}^{-1}}{\simeq} E.
	\end{equation*}
	We often omit $i_{\overline{\BY}}^{-1}$ via the identification $\End_{O_E}(\overline{\BY}) \otimes \BQ {\simeq} E$.
\end{definition}

\begin{remark}\label{remark5.2}
	We have an isomorphism between $\CN$ and $\widehat{\CN}$. For each $O_{\breve{E}}$-scheme $S$, the isomorphism sends $(X,i_X,\lambda_X,\rho_X) \in \CN(S)$ to 
	\begin{equation*}
	(X^{\vee}, \overline{i}_X^{\vee},\lambda'_X,(\rho_X^{\vee})^{-1}) \in \widehat{\CN}(S).
	\end{equation*}
	Here $\lambda'_X : X^{\vee} \rightarrow X$ is the unique polarization such that $\lambda'_X \circ \lambda_X =i_X(\pi)$, and for $a \in O_E$, we define $\overline{i}_X^{\vee}(a):=i_X(\overline{a})^{\vee}$.
\end{remark}	
\begin{definition}
	We write $\theta:\CN \rightarrow \widehat{\CN}$ for the isomorphism which is defined in Remark \ref{remark5.2}.
\end{definition}

\begin{definition}\label{definition5.1.4}
	$\quad$
	\begin{enumerate}
	\item For a given special homomorphism $x \in \BV$, we define the special cycle $\CZ(x)$ associated to $x$ in $\CN^0 \times \CN$ as the subfunctor of collections $\xi=(\overline{Y},i_{\overline{Y}},\lambda_{\overline{Y}},\rho_{\overline{Y}},X,i_X,\lambda_X,\rho_X)$ in $(\CN^0 \times \CN)(S)$ such that the quasi-homomorphism
	\begin{equation*}
	\rho_X^{-1}\circ x \circ \rho_{\overline{Y}}:\overline{Y} \times_S \overline{S} \rightarrow X \times_S \overline{S}
	\end{equation*}
	extends to a homomorphism from $\overline{Y}$ to $X$.
	
	\item For a given special homomorphism $y \in \BV$, we define the special cycle $\CY(y)$ associated to $y$ in $\CN^0 \times \CN$ as follows. First, consider the cycle  $\CZ(\lambda_{\BX}\circ y)$ in $\CN^0 \times \widehat{\CN}$. This is the subfunctor of collections $\xi=(\overline{Y},i_{\overline{Y}},\lambda_{\overline{Y}},\rho_{\overline{Y}},X^{\vee}, \overline{i}_X^{\vee},\lambda'_X,(\rho_X^{\vee})^{-1})$ in $(\CN^0 \times \widehat{\CN})(S)$ such that the quasi-homomorphism
		\begin{equation*}
	\rho_X^{\vee}\circ \lambda_{\BX}\circ y \circ \rho_{\overline{Y}}:\overline{Y} \times_S \overline{S} \rightarrow X^{\vee} \times_S \overline{S}
	\end{equation*}
	extends to a homomorphism from $\overline{Y}$ to $X^{\vee}$. We define $\CY(y)$ as $(id \times \theta^{-1})(\CZ(\lambda_{\BX}\circ y))$ in $\CN^0 \times \CN$.
\end{enumerate}

We note that $\CN^0$ can be identified with $\Spf O_{\breve{E}}$, hence $\CZ(x), \CY(y)$ can be identified with closed formal subschemes of $\CN$. Also, by abuse of notation, we often write $x:\overline{Y} \rightarrow X$ for the extension of quasi-homomorphism $\rho_X^{-1}\circ x \circ \rho_{\overline{Y}}$.
\end{definition}

Let $\overline{\BM}^0=\overline{\BM}^0_0 \oplus \overline{\BM}^0_1$ be the Dieudonne module of $\overline{\BY}$. As in \cite[Remark 2.5]{KR2}, it is easy to see that $\overline{\BM}^0_0= O_{\breve{F}} \overline{1}_0$ and $\overline{\BM}^0_1=O_{\breve{F}} \overline{1}_1$, where $\CF \overline{1}_1=\overline{1}_0$, $\CF \overline{1}_0=\pi \overline{1}_1$ and $\lbrace \overline{1}_0, \overline{1}_0 \rbrace = \pi$. We write $N^0$ for $\overline{\BM}^0 \otimes \BQ$.

Now, let $x \in \BV$. This induces a homomorphism from $N^0$ to $N$. We also write $x$ for the induced homomorphism. Note that we can write $x=x_0 + x_1$, where $x_0:N^0_0 \rightarrow N_0$ and $x_1:N^0_1 \rightarrow N_1$, since the morphism $x$ has degree $0$ with respect to the decompositions $N^0_0 \oplus N^0_1$ and $N_0 \oplus N_1$. 

To study the sets of $k$-points $\CZ(x)(k), \CY(y)(k)$,  $x,y \in \BV$, recall that we have a bijection between $\CN(k)$ and the set of lattices $(A,B)$ in $N_{k,0}$ (see Proposition \ref{proposition2.4}). Now, we can state the following analogue of \cite[Proposition 3.10]{KR2}.
\begin{proposition}\label{proposition5.5.5}(cf. \cite[Proposition 3.10]{KR2})
	For $x, y \in \BV$, we have the following bijections.
	\begin{enumerate}
		\item \begin{equation*}
		\CZ(x)(k)=	\left\{\begin{array}{c} 
		O_{\breve{F}} \text{-lattices} \\
	 A \overset{h}{\subset} B \subset N_{k,0}
		\end{array}
		\middle|
		\begin{array}{c}
		\pi B^{\vee} \overset{1}{\subset} A \overset{n-1}{\subset} B^{\vee},\\
		\pi A^{\vee} \overset{1}{\subset} B \overset{n-1}{\subset} A^{\vee},\\
		\pi B \subset A \subset B,\\
		x_0(\overline{1}_0) \in \pi B^{\vee}.
		
		\end{array}\right\}
		\end{equation*}
		
			\item \begin{equation*}
		\CY(y)(k)=	\left\{\begin{array}{c} 
		O_{\breve{F}} \text{-lattices} \\
		A \overset{h}{\subset} B \subset N_{k,0}
		\end{array}
		\middle|
		\begin{array}{c}
		\pi B^{\vee} \overset{1}{\subset} A \overset{n-1}{\subset} B^{\vee},\\
		\pi A^{\vee} \overset{1}{\subset} B \overset{n-1}{\subset} A^{\vee},\\
		\pi B \subset A \subset B,\\
		y_0(\overline{1}_0) \in \pi A^{\vee}.
		
		\end{array}\right\}
		\end{equation*}
		
	\end{enumerate}
\end{proposition}
\begin{proof} The proof of (1) is identical to the proof of \cite[Proposition 3.10]{KR2}. For (2), note that for the Dieudonne module $M=A\oplus B^{\perp}$ of $(X,i_X,\lambda_X,\rho_X) \in \CN(k)$, its dual $M^{\perp}=B\oplus A^{\perp}$ is the Dieudonne module of $X^{\vee}$ (here, $^\perp$ means the dual with respect to $\langle \cdot, \cdot \rangle$ in Section \ref{subsec:section2.2}). Therefore, (2) can be proved in the same way.
\end{proof}

\begin{lemma}(\cite[Lemma 1.16]{Vol}) Let $t \in O_E$ with $t^*=-t$ and let $V$ be a $E$-vector space of dimension $n$. Let $I_n$ be the identity matrix of rank $n$ and let $J_n$ be the matrix
\begin{displaymath} 
J_n:=\left( \begin{array}{cccc} 
\pi & & & \\
  &1 & & \\
 & & \ddots & \\
 & & & 1
\end{array}\right).
\end{displaymath}
There exist two perfect skew-hermitian forms on $V$ up to isomorphism. These forms correspond to $t I_n$ and to $t J_n$ respectively. Furthermore, if $M$ is a lattice in $V$ and $i \in \BZ$ with
\begin{equation*}
\pi^{i+1} M^{\vee} \overset{r}{\subset} M \overset{n-r}{\subset} \pi^i M^{\vee},
\end{equation*}
then $n-r \equiv ni \mod 2$ in the first case and $n-r \not\equiv ni \mod 2$ in the second case.
\end{lemma}
\begin{proof}
	See \cite[Lemma 1.16]{Vol}. Note that $F$ is a finite extension of $\BQ_p$, therefore the above statement is more general. But, the proof is identical.
\end{proof}

\begin{remark}\label{remark5.5.6}
	Recall that the $E$-vector space $N_{k,0}^{\tau}$ in Section \ref{subsec:section2.3} has a lattice $M$ with
	\begin{equation*}
	\pi M^{\vee} \overset{h+1}{\subset} M \overset{n-h-1}{\subset} M^{\vee}.
	\end{equation*}
	This fact follows from Lemma \ref{lemma3.2}. Therefore, by the above lemma, the form $\lbrace \cdot, \cdot \rbrace$ is isomorphic to $t I_n$ if $n-h-1 \equiv 0 \mod 2$ and is isomorphic to $t J_n$ if  $n-h-1 \not\equiv 0 \mod 2$.
\end{remark}

We need the following analogue of \cite[Lemma 3.7]{KR2}.
\begin{lemma}\label{lemma5.5.7} Assume that $h \neq 0, n$. Then we have 
	\begin{equation*}
	\bigcap_{\Lambda} \Lambda=(0),
	\end{equation*} 
	where $\Lambda$ runs over all vertex lattices of type $h+1$.
\end{lemma}
\begin{proof}
	First, assume that $n=h+1+2k$ for some integer $k \geq 0$, and $h+1$ is odd. Then by Remark \ref{remark5.5.6}, the form $\lbrace \cdot, \cdot \rbrace$ is isomorphic to $t I_n$. Choose a basis $\lbrace e_1, \dots, e_n \rbrace$ such that $\lbrace e_i, e_j \rbrace = t\delta_{ij}$. Choose any $h+1$ elements $\lbrace f_1, \dots, f_{h+1} \rbrace$ in $\lbrace e_1, \dots, e_n \rbrace$ and rename $\lbrace e_1, \dots, e_n \rbrace$ to $\lbrace f_1, \dots, f_n \rbrace$.
	
	Let $\alpha$, $\beta$ be elements in $E$ such that $\alpha \alpha^*=-1$ and $\beta \beta^*=1/2$.
	
	We define
	\begin{equation*}
	\begin{array}{ll}
	g_{h+1}:=f_{h+1},&\\
	g_{2i+1}:= \beta(f_{2i+1} + \alpha f_{2i+2}),&\\
	g_{2i+2}:=\beta(f_{2i+1} - \alpha f_{2i+2}), & \forall 0 \leq i \leq \frac{h}{2}-1.
	\end{array}
	\end{equation*}
	
	Then we have
	\begin{equation*}
	\begin{array}{ll}
	\lbrace g_{2i+1}, g_{2i+1} \rbrace=0,&\lbrace g_{2i+2}, g_{2i+2} \rbrace=0, \\
	\lbrace g_{2i+1}, g_{2i+2} \rbrace=t, & \forall 0 \leq i \leq h/2-1.
	\end{array}
	\end{equation*}
	
	Now consider an element $\gamma \in E$ such that $1+\gamma \gamma^*=\pi$, and define
	\begin{equation*}
	\begin{array}{ll}
	h_{h+1+2i+1}:=f_{h+1+2i+1} + \gamma f_{h+1+2i+2}&\\
	h_{h+1+2i+2}:=\gamma^* f_{h+1+2i+1} -  f_{h+1+2i+2},& \forall 0 \leq i \leq k-1.
	\end{array}
	\end{equation*}
	Also, we define
	\begin{equation*}
	\begin{array}{ll}
	g_{h+1+2i+1}:=\beta(h_{h+1+2i+1} + \alpha h_{h+1+2i+2})&\\
	g_{h+1+2i+2}:=\beta(h_{h+1+2i+1} - \alpha h_{h+1+2i+2}),& \forall 0 \leq i \leq k-1.
	\end{array}
	\end{equation*}
 	Then we have
\begin{equation*}
\begin{array}{ll}
\lbrace g_{h+1+2i+1}, g_{h+1+2i+1} \rbrace=0,&\lbrace g_{h+1+2i+2}, g_{h+1+2i+2} \rbrace=0, \\
\lbrace g_{h+1+2i+1}, g_{h+1+2i+2} \rbrace=t\pi, & \forall 0 \leq i \leq k-1.
\end{array}
\end{equation*}
	
	For $I:=(a_1,\dots,a_{h/2},b_1, \dots, b_k) \in \BZ^{h/2} \times \BZ^k$, we set
	\begin{equation*}
	\begin{array}{lll}
	\Lambda_{\lbrace g_1,\dots, g_n\rbrace, I}:=&[\pi^{a_1}g_1,\pi^{-a_1}g_2,\dots,&\pi^{a_{h/2}}g_{h-1},\pi^{-a_{h/2}}g_{h},\\
	& &g_{h+1}, \pi^{b_1}g_{h+2}, \dots, \pi^{-b_k}g_n ].
	\end{array}
	\end{equation*}
	Then, this is a vertex lattice of type $h+1$ and we have
	\begin{equation*}
	\bigcap_{\lbrace g_1,\dots, g_n\rbrace, I} \Lambda_{\lbrace g_1,\dots, g_n\rbrace, I}=(0),
	\end{equation*}
	where $\lbrace g_1,\dots, g_n\rbrace$ runs over all choices and $I$ runs through $\BZ^{h/2} \times \BZ^k$.
	
	This proves the lemma in the case that $n=h+1+2k$ for some integer $k \geq 0$, and $h+1$ is odd.
	
	Similar arguments work for the other cases.
\end{proof}
	
\begin{proposition}\label{proposition5.5}
	The functors $\CZ(x)$ and $\CY(y)$ are represented by closed formal subschemes of $\CN^0 \times \CN$. In fact, $\CZ(x)$ and $\CY(y)$ are relative divisors in $\CN^0 \times \CN$ (or empty) for any $x, y \in \BV \backslash \lbrace 0 \rbrace$.
\end{proposition}
\begin{proof}
	If $h=0$ (resp. $h=n$), then we have $\CZ(x)=\CY(x)$ (resp. $\CZ(\pi x)=\CY(x)$). Therefore, the case where $h=0$ is proved in \cite[Proposition 3.5]{KR2} (the case that $h=n$ is the same since we have the isomorphism $\theta$). For the other cases, we can follow the proof of \cite[Proposition 3.5]{KR2} with Lemma \ref{lemma5.5.7}. Indeed, we only need to show that $\CZ(x)(k)$ cannot be $\CN(k)$. If $\CN(k) \subset \CZ(x)(k)$, then we have
	\begin{equation*}
	x \in \bigcap_{\Lambda} \pi \Lambda^{\vee},
	\end{equation*}
	where $\Lambda$ runs over all vertex lattices of type $h+1$. This fact follows from Lemma \ref{lemma3.2} and Proposition \ref{proposition5.5.5}. Now, since we have
	\begin{equation*}
	\bigcap_{\Lambda} \pi \Lambda^{\vee} \subset \bigcap_{\Lambda}  \Lambda=(0),
	\end{equation*}
	by Lemma \ref{lemma5.5.7}, we have that $x$ should be $0$. This finishes the proof of the proposition.
\end{proof}

We have the following analogue of the remarks after \cite[Lemma 5.2]{KR2} (and also in \cite{KR4}).

\begin{proposition}\label{proposition5.6} $\quad$
\begin{enumerate}
\item If $val(h(x,x))=0$, then $\CZ(x) \simeq \CN^{h}_{E/F}(1,n-2)_{O_{\breve{E}}}$.
\item If $val(h(y,y))=-1$, then $\CY(y) \simeq \CN^{h-1}_{E/F}(1,n-2)_{O_{\breve{E}}}$.
\end{enumerate}
\end{proposition}
\begin{proof}
	(1) For an $O_{\breve{E}}$-scheme $S$, assume that $(X,i_X,\lambda_X,\rho_X) \in \CZ(x)(S)$. We can take a rescaled $x$ by an element in $O_E^{\times}$ such that $h(x,x)=1$. We denote by $x^*$ the element $\lambda_{\overline{Y}}^{-1} \circ x^{\vee}  \circ \lambda_{X}.$ Then we have that $e:=x \circ x^*$ is an idempotent in $\End_{O_E}(X)$, so that $X = e(X) \times (1-e)(X)$. Via this decomposition, we have the decomposition of the action $i_X=i_1 \times i_2$. Also, note that we have the canonical isomorphisms $e^{\vee}(X^{\vee})=(eX)^{\vee}$ and $(1-e^{\vee})(X)=((1-e)(X))^{\vee}$. By this identification, we have that the polarization $\lambda_X$ decomposes into the product of polarizations $\lambda_1=\lambda_X \circ e$ and $\lambda_2=\lambda_X \circ (1-e)$ of $eX$ and $(1-e)(X)$ respectively. Let $\rho_1=e \circ \rho_X$, $\rho_2=(1-e) \circ \rho_X$, the quasi-isogenies of $e(X)$ and $(1-e)X$,  respectively. Then $x$ defines an isomorphism $\overline{Y} \simeq e(X)$ compatible with polarizations, and $((1-e)(X), i_2, \lambda_2,\rho_2)$ gives an element in $\CN^{h}_{E/F}(1,n-2)_{O_{\breve{E}}}(S)$.
	
	Conversely, for an element $(X_2,i_2,\lambda_2,\rho_2) \in  \CN^{h}_{E/F}(1,n-2)_{O_{\breve{E}}}(S)$, we can take $X=\overline{Y} \times X_2$ with $x=\inc_1:\overline{\BY} \rightarrow \BX$, the action $i_X=i_{\overline{Y}} \times i_2$, the polarization $\lambda_{\overline{Y}} \times \lambda_2$ and the quasi-isogeny $\rho_{\overline{Y}} \times \rho_2$. Then this gives an element in $\CZ(x)(S)$. This construction gives the inverse of the previous one up to isomorphism.
	
	(2) For an $O_{\breve{E}}$-scheme $S$, let $(X,i_X,\lambda_X,\rho_X) \in \CY(y)(S)$. Consider 
	\begin{equation*}
		\theta((X,i_X,\lambda_X,\rho_X))=(X^{\vee},\overline{i}_X^{\vee},\lambda_X',(\rho_X^{\vee})^{-1}).
	\end{equation*}
	For $z=\lambda_X \circ y$, let $z^*=\lambda^{-1}_{\overline{Y}} \circ z^{\vee} \circ \lambda_X'$. Then we have
	\begin{equation*}
	\begin{array}{ll}
	z^*\circ z&=\lambda^{-1}_{\overline{Y}} \circ y^{\vee}  \circ \lambda_X^{\vee}  \circ \lambda_X' \circ \lambda_X \circ y\\
	&=\lambda^{-1}_{\overline{Y}} \circ y^{\vee}  \circ (-\lambda_X) \circ \lambda_X' \circ \lambda_X \circ y\\
	&=-\pi h(y,y).
	\end{array}
	\end{equation*}
	Therefore, $val(z^*\circ z)=0$. We can take rescaled $y$ by an element in $O_E^{\times}$ such that $z^* \circ z=1$. Then we have that $e:=z \circ z^*$ is an idempotent in $\End_{O_E}(X^{\vee})$. Now, as in the proof of (1), we have that
	\begin{equation*}
	((1-e)X^{\vee},\overline{i}_X^{\vee},(1-e^{\vee})\lambda_X',(1-e)(\rho_X^{\vee})^{-1}) \in \CN^{n-h}_{E/F}(1,n-2)_{O_{\breve{E}}}(S).
	\end{equation*}
	Therefore, by taking $\theta^{-1}((1-e)X^{\vee},\overline{i}_X^{\vee},(1-e^{\vee})\lambda_X',(1-e)(\rho_X^{\vee})^{-1}))$, we have an element of $\CN^{h-1}_{E/F}(1,n-2)_{O_{\breve{E}}}$.
	
	Now, let $(X_2,i_2,\lambda_2,\rho_2) \in \CN^{h-1}_{E/F}(1,n-2)_{O_{\breve{E}}}$. We will construct the inverse of the above construction. First, consider
	\begin{equation*}
	\theta ((X_2,i_2,\lambda_2,\rho_2)) =(X_2^{\vee}, \overline{i}_2^{\vee}, \lambda_2',(\rho_2^{\vee})^{-1}) \in \CN^{n-h}_{E/F}(1,n-2)_{O_{\breve{E}}}
	\end{equation*}
	Then we define
	\begin{equation*}
	\begin{array}{l}
	X^{\vee}:=\overline{Y} \times X_2^{\vee},\\
	\overline{i}_X^{\vee}:=i_{\overline{Y}} \times \overline{i}_2^{\vee},\\
	\lambda_X':=\lambda_{\overline{Y}} \times \lambda_2',\\
	(\rho_X^{\vee})^{-1}:=\rho_{\overline{Y}} \times (\rho_2^{\vee})^{-1}.
	\end{array}
	\end{equation*}
	This $(X^{\vee},\overline{i}_X^{\vee},\lambda_X',(\rho_X^{\vee})^{-1})$ is an element of $\CN^{n-h}_{E/F}(1,n-1)_{O_{\breve{E}}}$
	
	Now, we define $(X,i_X,\lambda_X,\rho_X)=\theta^{-1}((X^{\vee},\overline{i}_X^{\vee},\lambda_X',(\rho_X^{\vee})^{-1})$, with 
	\begin{equation*}
	\lambda_{\BX} \circ y:=\inc_1:\overline{\BY} \rightarrow \BX^{\vee}.
	\end{equation*}
	Then, this $(X,i_X,\lambda_X,\rho_X)$ gives an element in $\CY(y)$ and this construction inverts the previous one up to isomorphism.
\end{proof}

\begin{proposition}\label{proposition5.5.8}
	Assume that $val(h(x,x))=0, val(h(y,y))=-1$. Assume further that by rescaling as in Proposition \ref{proposition5.6}, $x^*\circ x=1, (\lambda_{\BX} \circ y)^*\circ (\lambda_{\BX} \circ y)=1$. We define $e_x:=x \circ x^*$ and $e_y:=(\lambda_{\BX} \circ y) \circ (\lambda_{\BX} \circ y)^*$. Fix isomorphisms
	\begin{equation*}
	\begin{split}
	\Phi:\CZ(x) \simeq \CN^h_{E/F}(1,n-2)_{O_{\breve{E}}},\\
	\Psi:\CY(y) \simeq \CN^{h-1}_{E/F}(1,n-2)_{O_{\breve{E}}},
	\end{split}
	\end{equation*}
	as in Proposition \ref{proposition5.6}. Then the following statements hold.
	\begin{enumerate}
		\item For $z\in \BV$ such that $h(x,z)=0$, let $z':=(1-e_x)\circ z$. Then, we have $\Phi(\CZ(x) \cap \CZ(z))=\CZ(z')$ in $\CN^h_{E/F}(1,n-2)$ and $h(z',z')=h(z,z)$.
		
		\item For $w\in \BV$ such that $h(x,w)=0$, let $w':=(1-e_x)\circ w$. Then, we have $\Phi(\CZ(x) \cap \CY(w))=\CY(w')$ in $\CN^h_{E/F}(1,n-2)$ and $h(w',w')=h(w,w)$.
		
		\item For $z\in \BV$ such that $h(y,z)=0$, let $z':=(1-e_y^{\vee})\circ z$. Then, we have $\Psi(\CY(y) \cap \CZ(z))=\CZ(z')$ in $\CN^{h-1}_{E/F}(1,n-2)$ and $h(z',z')=h(z,z)$.
		
		\item For $w\in \BV$ such that $h(y,w)=0$, let $w':=(1-e_y^{\vee})\circ w$. Then, we have $\Psi(\CY(y) \cap \CY(w))=\CY(w')$ in $\CN^{h-1}_{E/F}(1,n-2)$ and $h(w',w')=h(w,w)$.
	\end{enumerate}
\end{proposition}
\begin{proof}
	We will prove (3). Similar arguments work for (1), (2), (4). For an element $(X,i_X,\lambda_X,\rho_X)$ in $\CY(y) \cap \CZ(z)$, we denote by $(X_2,i_{X_2},\lambda_{X_2},\rho_{X_2})$ the element $\Psi((X,i_X,\lambda_X,\rho_X))$ in $\CN^{h-1}_{E/F}(1,n-2)_{O_{\breve{E}}}$. Also, we denote by $(\BX_2, i_{\BX_2}, \lambda_{\BX_2})$ the framing object of $\CN^{h-1}_{E/F}(1,n-2)_{O_{\breve{E}}}$. By definition of $\CY(y)$ and $\CZ(z)$, we have that $e_y$ can be extended to a morphism in $\End(X^{\vee})$, and $z:\overline{\BY} \rightarrow \BX$ can be extended to a morphism $z:\overline{Y} \rightarrow X$. Therefore, $z'=(1-e_y^{\vee}) \circ z$ can be extended to a morphism $\overline{Y} \rightarrow X_2=(1-e_y^{\vee})X$. This proves that $\Psi(\CY(y) \cap \CZ(z))\subset \CZ(z')$.
	
	Conversely, for a given element $(X_2,i_{X_2},\lambda_{X_2},\rho_{X_2})$ in $\CZ(z')$, we can use the construction in Proposition \ref{proposition5.6}, with 
	\begin{equation*}
	z=\inc_2 \circ z':\overline{\BY} \rightarrow \BX_2 \rightarrow \BX=\overline{\BY}^{\vee} \times \BX_2.
	\end{equation*}
	This construction gives an element in $\CY(y) \cap \CZ(z)$, and it is the element $\Psi^{-1}((X_2,i_{X_2},\lambda_{X_2},\rho_{X_2}))$. Therefore, we have $\Psi(\CY(y) \cap \CZ(z))=\CZ(z')$.
	
	Now, it remains to show that $h(z',z')=h(z,z)$. We have
	\begin{equation*}
		\begin{array}{ll}
			h(z',z')&=\lambda_{\overline{\BY}}^{-1} \circ (z')^{\vee} \circ \lambda_{\BX_2} \circ z'\\
			&=\lambda_{\overline{\BY}}^{-1} \circ (z^{\vee} \circ (1-e_y)) \circ ((1-e_y) \circ \lambda_{\BX}) \circ ((1-e_y^{\vee})\circ z)\\
			&=\lambda_{\overline{\BY}}^{-1} \circ z^{\vee} \circ (1-e_y) \circ \lambda_{\BX} \circ z.\\
			&=\lambda_{\overline{\BY}}^{-1} \circ z^{\vee} \circ \lambda_{\BX} \circ z - \lambda_{\overline{\BY}}^{-1} \circ z^{\vee} \circ e_y \circ \lambda_{\BX} \circ z \\
			&=h(z,z)-\lambda_{\overline{\BY}}^{-1} \circ z^{\vee} \circ e_y \circ \lambda_{\BX} \circ z.
		\end{array}
	\end{equation*}
	Here, we used $e_y \circ \lambda_{\BX} = \lambda_{\BX} \circ (e_y^{\vee})$. Now, it remains to show that 
	\begin{equation*}
	\lambda_{\overline{\BY}}^{-1} \circ z^{\vee} \circ e_y \circ \lambda_{\BX} \circ z=0.
	\end{equation*}
	Note that
	\begin{equation*}
	\begin{array}{l}
	e_y=\lambda_{\BX} \circ y \circ \lambda_{\overline{\BY}}^{-1} \circ y^{\vee} \circ \lambda_{\BX}^{\vee} \circ \lambda_{\BX}'\\
	\end{array}.
	\end{equation*}
	Therefore, we have
	\begin{equation*}
	\begin{array}{l}
	\lambda_{\overline{\BY}}^{-1} \circ z^{\vee} \circ e_y \circ \lambda_{\BX} \circ z\\
	=\lambda_{\overline{\BY}}^{-1} \circ z^{\vee} \circ \lambda_{\BX} \circ y \circ \lambda_{\overline{\overline{\BY}}}^{-1} \circ y^{\vee} \circ \lambda_{\BX}^{\vee} \circ \lambda_{\BX}' \circ \lambda_{\BX} \circ z\\
	=-h(y,z) h(z,y) \pi \\
	=0.
	\end{array}
	\end{equation*}
	The last equality follows from our assumption $h(y,z)=0$. This finishes the proof of (3).
\end{proof}

\begin{lemma}
	Assume that $x_1,x_2,y_1,y_2$ are linearly independent special homomorphisms in $\BV$ and 
	\begin{equation*}
	val(h(x_1,x_1))=0, val(h(y_1,y_1))=-1.
	\end{equation*}
	Then we have the following assertions.
	\begin{enumerate}
		\item $O_{\CZ(x_1)} \otimes^{\BL}_{O_{\CN}} O_{\CZ(x_2)}=O_{\CZ(x_1)} \otimes_{O_{\CN}} O_{\CZ(x_2)}$.
		\item $O_{\CZ(x_1)} \otimes^{\BL}_{O_{\CN}} O_{\CY(y_2)}=O_{\CZ(x_1)} \otimes_{O_{\CN}} O_{\CY(y_2)}$.
		\item $O_{\CY(y_1)} \otimes^{\BL}_{O_{\CN}} O_{\CZ(x_2)}=O_{\CY(y_1)} \otimes_{O_{\CN}} O_{\CZ(x_2)}$.
		\item $O_{\CY(y_1)} \otimes^{\BL}_{O_{\CN}} O_{\CY(y_2)}=O_{\CY(y_1)} \otimes_{O_{\CN}} O_{\CY(y_2)}$.
	\end{enumerate}
Here, we write $\otimes^{\BL}$ for the derived tensor product of $O_{\CN}$-modules.
\end{lemma}
\begin{proof}
	(1) By Terstiege's proof in \cite[Lemma 3.1]{Ter}, it suffices to show that $\CZ(x_1)$ and $\CZ(x_2)$ have no common component. By Proposition \ref{proposition5.6}, $\CZ(x_1) \simeq \CN^{h}_{E/F}(1,n-2)_{O_{\breve{E}}}$, and by Proposition \ref{proposition5.5.8}, $\CZ(x_1) \cap \CZ(x_2)=\CZ(x_2')$ in $\CN^{h}_{E/F}(1,n-2)_{O_{\breve{E}}}$. Therefore, by Proposition \ref{proposition5.5}, $\CZ(x_1) \cap \CZ(x_2)$ is a divisor in $\CN^{h}_{E/F}(1,n-2)_{O_{\breve{E}}}$. This implies that $\CZ(x_1) \cap \CZ(x_2)$ has codimension $2$ in $\CN$ and hence, $\CZ(x_1)$ and $\CZ(x_2)$ have no common component.
	
	The proof of (2),(3),(4) are similar.
\end{proof}

\begin{remark}\label{remark5.5.9}
	Let $\lbrace x_1, \dots ,x_{n-h}, y_1, \dots, y_h\rbrace$ be an orthogonal basis of $\BV$. If $val(h(x_1,x_1))=0$, then by above lemma, we have
	\begin{equation*}
	\begin{array}{ll}
	&O_{\CY(y_1)} \otimes^{\BL}_{O_{\CN}} \dots \otimes^{\BL}_{O_{\CN}} O_{\CY(y_{h})} \otimes^{\BL}_{O_{\CN}}O_{\CZ(x_1)} \otimes^{\BL}_{O_{\CN}} \dots \otimes^{\BL}_{O_{\CN}} O_{\CZ(x_{n-h})} \\ \\
	=&(O_{\CZ(x_1)} \otimes^{\BL}_{O_{\CN}} O_{\CY(y_1)}) \otimes^{\BL}_{O_{\CZ(x_1)}} \dots \otimes^{\BL}_{O_{\CZ(x_1)}} (O_{\CZ(x_1)} \otimes^{\BL}_{O_{\CN}} O_{\CZ(x_{n-h})})\\\\
	=&(O_{\CZ(x_1)} \otimes_{O_{\CN}} O_{\CY(y_1)}) \otimes^{\BL}_{O_{\CZ(x_1)}} \dots \otimes^{\BL}_{O_{\CZ(x_1)}} (O_{\CZ(x_1)} \otimes_{O_{\CN}} O_{\CZ(x_{n-h})})\\\\
	=&O_{\CZ(x_1) \cap \CY(y_1)} \otimes^{\BL}_{O_{\CZ(x_1)}} \dots \otimes^{\BL}_{O_{\CZ(x_1)}} O_{\CZ(x_1)\cap \CZ(x_{n-h})}\\\\
	=&O_{\CY(y_1')} \otimes^{\BL}_{O_{\CN^h(1,n-2)}} \dots O_{\CZ(x_2')} \otimes^{\BL}_{O_{\CN^h(1,n-2)}}\dots \otimes^{\BL}_{O_{\CN^h(1,n-2)}}O_{\CZ(x_{n-h}')}.
	\end{array}
	\end{equation*}
	In the last line, we regard the special cycles $\CY(y_1'), \dots \CZ(x_h')$ as the cycles in $\CN^h(1,n-2)$ via the identification $\CZ(x_1)=\CN^h(1,n-2)$ as in Proposition \ref{proposition5.5.8}.
	
	Similarly, we can do the same reduction, when $val(h(y_1,y_1))=-1$. In this case, we have an intersection in $\CN^{h-1}(1,n-2)$
\end{remark}

Let $[\textbf{x},\textbf{y}]:=[ x_1, \dots ,x_{n-h}, y_1, \dots, y_h]$ be an orthogonal basis of $\BV$. We will compute the intersection number 
\begin{equation*}
\chi(O_{\CY(y_1)} \otimes^{\BL}_{O_{\CN}} \dots \otimes^{\BL}_{O_{\CN}} O_{\CY(y_{h})} \otimes^{\BL}_{O_{\CN}}O_{\CZ(x_1)} \otimes^{\BL}_{O_{\CN}} \dots \otimes^{\BL}_{O_{\CN}} O_{\CZ(x_{n-h})}),
\end{equation*}
in some special cases. Here, we write $\chi$ for the Euler-Poincare characteristic (\cite{KR5}, \cite{Zha}). More precisely, for the structure morphism $\omega:\CN \rightarrow \Spf O_{\breve{E}}$ and for a sheaf of $O_{\CN}$-modules $\CH$, we define
\begin{equation*}
\chi(\CH):=\sum_{i} (-1)^i \length_{O_{\breve{E}}}(R^i\omega_*\CH).
\end{equation*}
For a bounded complex of sheaves $\CH^{\bullet}$ of $O_{\CN}$-modules, we define
\begin{equation*}
\chi(\CH^{\bullet}):=\sum_i (-1)^i \chi(\CH^i).
\end{equation*}

\begin{theorem}\label{theorem5.5.12}
		Let $\lbrace x_1, \dots ,x_{n-h}, y_1, \dots, y_h\rbrace$ be an orthogonal basis of $\BV$. Assume that
	\begin{equation*}
	\begin{array}{ll}
	val(h(x_i,x_i))=0  &\text{ for all } 3 \leq i \leq n-h,\\
	val(h(y_j,y_j))=-1 &\text{ for all } 1 \leq j \leq h,\\
	\end{array}
	\end{equation*}
	and write $a:=val(h(x_1,x_1))$, $b:=val(h(x_2,x_2))$. We assume that $a \leq b$ and $a \not\equiv b \mod 2$. Then we have
	\begin{equation*}
	\chi(O_{\CY(y_1)} \otimes^{\BL}_{O_{\CN}} \dots \otimes^{\BL}_{O_{\CN}} O_{\CZ(x_h)})=\dfrac{1}{2} \sum_{l=0}^{a} q^l(a+b+1-2l).
	\end{equation*}
	
	More generally, consider another basis $[\td{\tb{x}},\td{\tb{y}}]:=[ \td{x}_1, \dots ,\tilde{x}_{n-h}, \tilde{y}_1, \dots, \tilde{y}_h]$ of $\BV$ such that $\td{\tb{x}}=\td{x} g_1, \td{\tb{y}}=\td{y} g_2$ for $g_1 \in \tx{GL}_{n-h}(O_E)$ and $g_2 \in \tx{GL}_h(O_E)$. Then we have
	\begin{equation*}
	\chi(O_{\CY(\td{y}_1)} \otimes^{\BL}_{O_{\CN}} \dots \otimes^{\BL}_{O_{\CN}} O_{\CZ(\td{x}_h)})=\dfrac{1}{2} \sum_{l=0}^{a} q^l(a+b+1-2l).
	\end{equation*}
\end{theorem}
\begin{proof}
		By applying Remark \ref{remark5.5.9} repeatedly, the problem reduces to the case of $n=2$ and we need to compute the intersection number
	\begin{equation*}
	\chi(O_{\CZ(z_1)} \otimes^{\BL}_{O_{\CN^0(1,1)}} O_{\CZ(z_2)}).
	\end{equation*}
	This intersection number is computed in \cite[Theorem 4.13]{Liu}. Indeed,
		\begin{equation*}
	\chi(O_{\CZ(z_1)} \otimes^{\BL}_{O_{\CN^0(1,1)}} O_{\CZ(z_1)})=\dfrac{1}{2} \sum_{l=0}^{a} q^l(a+b+1-2l).
	\end{equation*}
	
	For the general cases, first we need to show that $(\CY(\td{y}_1)\cap \dots \cap\CZ(\td{x}_h))(k)$ is a single point. By Proposition \ref{proposition5.5.5}, $(\CY(\td{y}_1)\cap \dots \cap\CZ(\td{x}_h))(k)$ is
	\begin{equation}\label{eq5.0.1}
	\left\{\begin{array}{c} 
	O_{\breve{F}}$-lattices $ A \overset{h}{\subset} B \subset N_{k,0}
	\end{array}
	\middle|
	\begin{array}{c}

	\pi B^{\vee} \overset{1}{\subset} A \overset{n-1}{\subset} B^{\vee}\\
	 \pi A^{\vee} \overset{1}{\subset} B \overset{n-1}{\subset} A^{\vee};\\
	\pi B \subset A \subset B;\\
	\td{x}_1(\overline{1}_0), \dots, \td{x}_{n-h}(\overline{1}_0) \in \pi B^{\vee};\\
	\td{y}_1(\overline{1}_0), \dots, \td{y}_h(\overline{1}_0) \in \pi A^{\vee}.
	\end{array}
	\right\}.
	\end{equation}
 It is easy to see that this is the same as $(\CY({y}_1)\cap \dots \cap\CZ(x_h))(k)$, since the above conditions in \eqref{eq5.0.1} are invariant under the linear combination $\td{\tb{x}}=\td{x} g_1, \td{\tb{y}}=\td{y} g_2$. Also, by Remark \ref{remark5.5.9}, we know that this is a single point. Therefore, we can use the length of a deformation ring to compute our intersection number as in \cite[Section 5]{KR2}, and this is invariant under the linear combination $[\td{\tb{x}},\td{\tb{y}}]=[\tb{x} g_1,\tb{y} g_2].$ Therefore, we have
	\begin{equation*}
	\begin{array}{ll}
	\chi(O_{\CY(\td{y}_1)} \otimes^{\BL}_{O_{\CN}} \dots \otimes^{\BL}_{O_{\CN}} O_{\CZ(\td{x}_h)})
	&=\chi(O_{\CY(y_1)} \otimes^{\BL}_{O_{\CN}} \dots \otimes^{\BL}_{O_{\CN}} O_{\CZ(x_h)})\\
	&=\dfrac{1}{2} \sum_{l=0}^{a} q^l(a+b+1-2l).
	\end{array}
	\end{equation*}
\end{proof}

\begin{theorem}\label{theorem5.5.13}
		Let $\lbrace x_1, \dots ,x_{n-h}, y_1, \dots, y_h\rbrace$ be an orthogonal basis of $\BV$. Assume that
	\begin{equation*}
	\begin{array}{l}
	val(h(x_i,x_i))=0  \text{ for all } 1 \leq i \leq n-h,\\
	val(h(y_j,y_j))=-1 \text{ for all } 3 \leq j \leq h,
	\end{array}
	\end{equation*}
	and write $a:=val(h(y_1,y_1))$, $b:=val(h(y_2,y_2))$. We assume that $a \leq b$ and $a \not\equiv b \mod 2$. Then we have,
	\begin{equation*}
	\chi(O_{\CY(y_1)} \otimes^{\BL}_{O_{\CN}} \dots \otimes^{\BL}_{O_{\CN}} O_{\CZ(x_h)})=\dfrac{1}{2} \sum_{l=0}^{a+1} q^l(a+b+3-2l).
	\end{equation*}
	
		More generally, consider another basis $[\td{\tb{x}},\td{\tb{y}}]:=[ \td{x}_1, \dots ,\tilde{x}_{n-h}, \tilde{y}_1, \dots, \tilde{y}_h]$ of $\BV$ such that $\td{\tb{x}}=\td{x} g_1, \td{\tb{y}}=\td{y} g_2$ for $g_1 \in \tx{GL}_{n-h}(O_E)$ and $g_2 \in \tx{GL}_h(O_E)$. Then 
	\begin{equation*}
	\chi(O_{\CY(\td{y}_1)} \otimes^{\BL}_{O_{\CN}} \dots \otimes^{\BL}_{O_{\CN}} O_{\CZ(\td{x}_h)})=\dfrac{1}{2} \sum_{l=0}^{a+1} q^l(a+b+3-2l).
	\end{equation*}
\end{theorem}
\begin{proof}
	By applying Remark \ref{remark5.5.9} repeatedly, the problem reduces to the case of $n=2$ and we need to compute the intersection number
	\begin{equation*}
	\chi(O_{\CY(y_1)} \otimes^{\BL}_{O_{\CN^2(1,1)}} O_{\CY(y_2)}).
	\end{equation*}
	By applying $\theta$, we can change our problem to the problem of computing the intersection number
	\begin{equation*}
	\chi(O_{\CZ(\lambda_{\BX} \circ y_1)} \otimes^{\BL}_{O_{\CN^0(1,1)}} O_{\CZ(\lambda_{\BX} \circ y_2)}).
	\end{equation*}
	Note that $\lambda_{\BX} \circ y_1$, $\lambda_{\BX} \circ y_2$ have orders $a+1$ and $b+1$, respectively. Therefore, by \cite[Theorem 4.13]{Liu}, we have
	\begin{equation*}
	\chi(O_{\CZ(\lambda_{\BX} \circ y_1)} \otimes^{\BL}_{O_{\CN^0(1,1)}} O_{\CZ(\lambda_{\BX} \circ y_2)})=\dfrac{1}{2} \sum_{l=0}^{a+1} q^l(a+b+3-2l).
	\end{equation*}
	The proof of the general case is the same as Theorem \ref{theorem5.5.12}.
\end{proof}

\begin{remark}\label{remark5.5.16}
 Assume that
\begin{equation*}
\begin{array}{ll}
val(h(x_i,x_i))=0  &\text{for all } 1 \leq i \leq n-h-1,\\
val(h(y_j,y_j))=-1 &\text{for all } 1 \leq j \leq h-1.
\end{array}
\end{equation*} In this case, by the above remark, we can reduce the problem to the intersection problem in $\CN^1(1,1)$ that is the Drinfeld upper half-plane. In this case all intersection numbers of special cycles (even in the case of improper intersection) can be computed explicitly (see \cite{San} or \cite{KR5}). We will compute this in forthcoming work.
\end{remark}

\end{document}